\documentclass[10pt]{article}

\input xy

\xyoption{all}

\usepackage{amssymb,amsbsy,amsthm,amsmath,mathrsfs,graphicx}
\usepackage{times}

\newtheorem{thm}{Theorem}[section]
\newtheorem{lem}[thm]{Lemma}
\newtheorem{prop}[thm]{Proposition}
\newtheorem{cor}[thm]{Corollary}

\newtheorem{dfn}[thm]{Definition}

\newtheorem{ques}[thm]{Question}
\newtheorem*{thm*}{Theorem}
\theoremstyle{remark}
\newtheorem{ex}[thm]{Example}
\newtheorem*{rmk}{Remark}
\newtheorem*{rmks}{Remarks}

\renewcommand{\rm}[1]{\mathrm{#1}}
\renewcommand{\cal}[1]{\mathcal{#1}}

\newcommand{\bbN}{\mathbb{N}}

\newcommand{\bbQ}{\mathbb{Q}}
\newcommand{\bbR}{\mathbb{R}}
\newcommand{\bbT}{\mathbb{T}}
\newcommand{\bbZ}{\mathbb{Z}}

\newcommand{\bfU}{\mathbf{U}}

\newcommand{\sfC}{\mathsf{C}}

\newcommand{\rmH}{\mathrm{H}}

\renewcommand{\d}{\mathrm{d}}

\newcommand{\m}{\mathrm{m}}

\newcommand{\A}{\mathcal{A}}
\newcommand{\B}{\mathcal{B}}
\newcommand{\C}{\mathcal{C}}

\newcommand{\F}{\mathcal{F}}

\renewcommand{\P}{\mathcal{P}}

\newcommand{\Z}{\mathcal{Z}}

\newcommand{\scrM}{\mathscr{M}}
\newcommand{\scrN}{\mathscr{N}}
\newcommand{\scrP}{\mathscr{P}}
\newcommand{\scrQ}{\mathscr{Q}}
\newcommand{\scrR}{\mathscr{R}}
\newcommand{\scrS}{\mathscr{S}}

\newcommand{\frP}{\mathfrak{P}}
\newcommand{\frQ}{\mathfrak{Q}}
\newcommand{\frR}{\mathfrak{R}}

\newcommand{\G}{\Gamma}

\renewcommand{\a}{\alpha}
\renewcommand{\b}{\beta}
\newcommand{\eps}{\varepsilon}
\newcommand{\g}{\gamma}
\renewcommand{\k}{\kappa}

\newcommand{\s}{\sigma}
\renewcommand{\phi}{\varphi}

\newcommand{\Cnd}{\mathrm{Coind}}

\newcommand{\cplx}{\mathrm{cplx}}

\newcommand{\coker}{\mathrm{coker}}
\newcommand{\img}{\mathrm{img}}

\newcommand{\sgn}{\mathrm{sgn}}

\newcommand{\PDE}{{PD$^{\rm{ce}}$E}}

\newcommand{\Mod}{\mathsf{Mod}}

\newcommand{\PMod}{\mathsf{PMod}}
\newcommand{\SFMod}{\mathsf{SFMod}}
\newcommand{\SPMod}{\mathsf{SPMod}}

\newcommand{\incln}{{\hbox{incl$^{\ul{\mathrm{n}}}$}}}

\renewcommand{\sp}{\mathrm{sp}}

\renewcommand{\hat}[1]{\widehat{#1}}
\newcommand{\ol}[1]{\overline{#1}}
\newcommand{\ul}[1]{\underline{#1}}
\newcommand{\into}{\hookrightarrow}
\newcommand{\fin}{\nolinebreak\hspace{\stretch{1}}$\lhd$}

\renewcommand{\t}[1]{\widetilde{#1}}
\renewcommand{\to}{\longrightarrow}
\newcommand{\onto}{\twoheadrightarrow}
\newcommand{\actson}{\curvearrowright}
\newcommand{\lfl}{\lfloor}
\newcommand{\rfl}{\rfloor}

\newcommand{\uhr}{\!\upharpoonright}
\newcommand{\llc}{\llcorner}

\begin{document}

\title{Partial difference equations over compact Abelian groups, II: step-polynomial solutions}
\author{Tim Austin\footnote{Research supported by a fellowship from the Clay Mathematics Institute}\\ \\ \small{Courant Institute, New York University}\\ \small{New York, NY 10012, U.S.A.}\\ \small{\texttt{tim@cims.nyu.edu}}}
\date{}

\maketitle

\begin{abstract}
This paper continues an earlier work on the structure of solutions to two classes of functional equation.  Let $Z$ be a compact Abelian group and $U_1$, \ldots, $U_k \leq Z$ be closed subgroups.  Given $f:Z\to\bbT$ and $w \in Z$, one defines the differenced function
\[d_wf(z) := f(z+w) - f(z).\]
In this notation, we shall study solutions to the system of difference equations
\[d_{u_1}\cdots d_{u_k}f \equiv 0 \quad \forall (u_1,\ldots,u_k) \in \prod_{i\leq k}U_i,\]
and to the zero-sum problem
\[f_1 + \cdots + f_k = 0\]
for functions $f_i:Z\to \bbT$ that are $U_i$-invariant for each $i$.

Part I of this work showed that the $Z$-modules of solutions to these problems can be described using a general theory of `almost modest $\P$-modules'.  Much of the global structure of these solution $Z$-modules could then be extracted from results about the closure of this general class under certain natural operations, such as forming cohomology groups. The main result of the present paper is that solutions to either problem can always be decomposed into summands which either solve a simpler system of equations, or have some special `step polynomial' structure.  This will be proved by augmenting the definition of `almost modest $\P$-modules' further, to isolate a subclass in which elements can be represented by the desired `step polynomials'.  We will then find that this subclass is closed under the same operations.
\end{abstract}

\newpage

\parskip 0pt

\setcounter{tocdepth}{1}
\tableofcontents

\section{Introduction}

This paper continues the work of~\cite{Aus--PDceEI} (`Part I').  As in that paper, let $Z$ be a compact metrizable Abelian group, let $\bfU = (U_1,\ldots,U_k)$ be a tuple of closed subgroups of $Z$, and let $A$ be an Abelian Lie group (often $A$ will be $\bbT := \bbR/\bbZ$).  Given a measurable function $f:Z\to A$ and an element $w \in Z$, we define the associated \textbf{differenced function} to be
\begin{eqnarray}\label{eq:d}
d_wf(z) := f(z-w) - f(z).
\end{eqnarray}

Part I introduced two classes of functional equation for such functions.
\begin{itemize}
\item The \textbf{partial difference equation}, or {\PDE}, associated to $Z$ and $\bfU$ is the system
\begin{eqnarray}\label{eq:PDceE}
d_{u_1}\cdots d_{u_k}f = 0 \quad \forall u_1\in U_1,\,u_2 \in U_2,\,\ldots,\,u_k \in U_k
\end{eqnarray}
(since we quotient by functions that vanish a.e., this means formally that for \emph{strictly} every $u_1$, \ldots, $u_k$, the left-hand side is a function $Z\to A$ that vanishes at \emph{almost} every $z$).

\item The \textbf{zero-sum problem} associated to $Z$ and $\bfU$ is the problem of solving the equation
\begin{eqnarray}\label{eq:zero-sum}
f_1(z) + \cdots + f_k(z) = 0 \quad \quad \hbox{for $m_Z$-a.e.}\ z
\end{eqnarray}
among $k$-tuples of measurable functions $f_i:Z\to A$ such that each $f_i$ is $U_i$-invariant.  A tuple $(f_i)_{i=1}^k$ satisfying~(\ref{eq:zero-sum}) is a \textbf{zero-sum} tuple of functions.
\end{itemize}

For each of these problems, the set of solutions (that is, functions satisfying the {\PDE} in the first case, and zero-sum tuples in the second) forms a Polish $Z$-module.  Part I was concerned with the global structure of these Polish $Z$-modules, and established a basic description of them as two families of examples of `almost modest $\P$-modules'.

In this paper, our interest will be in the form of the individual functions that solve these equations.  Our main result is already suggested by the several examples analyzed in the Introduction and Section 12 of Part I.  Here it will suffice to recall one of those.  Let $\lfl\cdot\rfl$ denote the integer-part function on $\bbR$, and let $\{\cdot\}$ denote the obvious selector $\bbT\to [0,1)\subset \bbR$, which we will call the `fractional part' map.  Then Example 12.7 in Part I showed that the functions $\s,c:\bbT^2\times \bbT^2 \to \bbT$ defined by
\[\s(s,x) := \{s_1\}\{x_2\} - \lfl\{x_2\} + \{s_2\}\rfl\{x_1 + s_1\} \quad \mod 1\]
and
\[c(s,t) = \{s_1\}\{t_2\} - \{t_1\}\{s_2\} \quad \mod 1\]
satisfy the equation
\[\s(t,x) + \s(s,x+t) = \s(s,x) + \s(t,x+s) + c(s,t).\]
As explained there, this may be read as giving a zero-sum quintuple of functions on $\bbT^2\times \bbT^2\times \bbT^2$, with each function invariant under a different subgroup of this group.

The main result of this paper will be that, in a sense to be made precise below, the basic `building block' solutions to {\PDE}s or zero-sum problems may always be (chosen to be) functions assembled using $\lfl\cdot\rfl$ and $\{\cdot\}$, such as those above.

To formulate these theorems properly, we will introduce in Section~\ref{sec:QP} a precise notion of `step polynomial' functions on a compact Abelian group.  For now let us simply remark that a `step polynomial' on $Z$ is a function for which there is a partition of $Z$ into `geometrically-simple' pieces such that, on each of those pieces, the function is given by a polynomial in some fractional parts of characters of $Z$.

In the first place, step polynomials will appear as representative cocycles of the cohomology groups $\rmH^p_\m(Z,\bbT)$, which played a key r\^{o}le in the analysis in Part I.  A precursor to this fact can be found as~\cite[Proposition 9.4]{AusMoo--cohomcty}, and we give a complete proof in Subsection~\ref{subs:compar} below.  Our main results, Theorems A and B below, assert that step polynomials similarly appear as a complete list of solutions to {\PDE}s and zero-sum problems, modulo degenerate solutions.  Those theorems also give the following related fact: if a degenerate {\PDE}-solution (resp. zero-sum tuple) happens to be a step polynomial, then it may be decomposed into {\PDE}-solutions (resp. zero-sum tuples) corresponding to simpler equations so that the summands are still step polynomials.  Conclusions of this second kind will be an essential inclusion in some of the inductive proofs that will lead to Theorems A and B, as well as having some interest in their own right.

Theorems A and B below are best stated using the language of $\P$-modules from Part I.  Fix $Z$ and $\bfU = (U_1,\dots,U_k)$, and assume that $U_1 + \cdots + U_k = Z$.  (If this is not so, then the {\PDE} and zero-sum problems may simply be solved on each coset of $U_1 + \cdots + U_k$ independently, as discussed in Part I.)  Let $\scrM = (M_e)_{e\subseteq [k]}$ and $\scrN = (N_e)_{e \subseteq [k]}$ be the $\P$-modules of {\PDE} solutions and zero-sum tuples associated to this $Z$ and $\bfU$ (Subsection 4.6 in Part I).  Recall the main results of Part I: in this setup, the submodules $\partial^\scrM_k(M^{(k-1)})$ of degenerate {\PDE} solutions, and $\partial^\scrN_k(N^{(k-1)})$ of degenerate zero-sum tuples, are relatively open and co-countable in $M_{[k]}$ and $N_{[k]}$ respectively.

Also, before proceeding, we should recall from Corollaries A$''$ and B$''$ in Part I that if $A$ is a Euclidean space, then these $\P$-modules have vanishing structural homology except in the lowest nonvanishing position of each structure complex.  As a result, {\PDE}-solutions for $k\ge 2$ and zero-sum tuples for $k\ge 3$ could all be expressed quite easily in terms of degenerate solutions for Euclidean $A$.  On the other hand, some of the auxiliary theory to be developed below fails for some Euclidean modules, so we will simply exclude these now.

\vspace{7pt}

\noindent\textbf{Theorem A}\quad \emph{Suppose that $U_1 + \cdots + U_k = Z$ and that $A$ is a compact-by-discrete $Z$-module (this includes all compact and discrete modules).  The cosets of $\partial^\scrM_k(M^{(k-1)})$ in $M_{[k]}$ all contain representatives that are step polynomials.  Also, if an element of $\partial^\scrM_k(M^{(k-1)})$ is a step polynomial, then it is the image of a $k$-tuple in $M^{(k-1)}$ consisting of step polynomials.}

\vspace{7pt}

If $f:Z\to\bbT$ solves the {\PDE} associated to $\bfU$, then by an iterative appeal to Theorem A it can be decomposed as
\begin{eqnarray}\label{eq:PDceE-struct}
f = \sum_{e\subseteq [k]}f_e,
\end{eqnarray}
where:
\begin{itemize}
\item $f_e$ solves the {\PDE}-system associated to the sub-tuple $(U_i)_{i\in e}$;
\item each $f_e$ is a step polynomial on every coset of $\sum_{i\in e}U_i$.
\end{itemize}

\vspace{7pt}

\noindent\textbf{Theorem B}\quad \emph{Suppose that $U_1 + \cdots + U_k = Z$ and that $A$ is a compact-by-discrete $Z$-module. The cosets of $\partial^\scrN_k(N^{(k-1)})$ in $N_{[k]}$ all contain representatives that are step polynomials.  Also, if an element of $\partial^\scrN_k(N^{(k-1)})$ is a step polynomial, then it is the image of a $k$-tuple in $N^{(k-1)}$ consisting of step polynomials.}

\vspace{7pt}

Similarly to the above, this implies that any zero-sum tuple $(f_i)_{i=1}^k$ as in~(\ref{eq:zero-sum}) can be decomposed as
\begin{eqnarray}\label{eq:zero-sum-struct}
(f_i)_{i=1}^k = \sum_{e\subseteq [k],\,|e|\geq 2}(g_{e,i})_{i=1}^k
\end{eqnarray}
where:
\begin{itemize}
\item $(g_{e,i})_{i=1}^k$ is a zero-sum tuple for every $e$;
\item $g_{e,i} = 0$ if $i\not\in e$;
\item each $g_{e,i}$ is a step polynomial on every coset of $\sum_{j\in e}U_j$.
\end{itemize}

As they stand, Theorems A and B are vacuous in case $Z$ is a finite group, since every function on a finite Abelian group may be written as a step polynomial.  However, knowing these results for arbitrary compact Abelian groups, we will be able to extract some (highly ineffective) quantitative dependence via a compactness argument.  This will promise some control on the `complexity' of the functions involved, nontrivial even for finite groups.

To formulate this, we next invoke a notion of complexity for step polynomial functions.  That notion will not be defined until Section~\ref{sec:quantitative}, where the theorem will be proved, but intuitively it bounds the number of partition-cells, the degrees and the coefficients involved in specifying the step polynomial.

\vspace{7pt}

\noindent\textbf{Theorem C}\quad \emph{For every $k\geq 1$ there is an $\eps > 0$ such that for every $d\in \bbN$ there is a $D \in \bbN$ for which the following holds.  Let $Z$ be a compact metrizable Abelian group, let $\bfU = (U_i)_{i=1}^k$ be a tuple of subgroups of $Z$, and let $\scrM = (M_e)_e$ be the associated $\bbT$-valued {\PDE}-solution $\P$-module.  If $f \in M_{[k]}$ is such that
\[d_0(f,g) < \eps\]
for some step polynomial $g \in \F(Z)$ of complexity at most $d$, then $f \in f' + \partial_k(M^{(k-1)})$ for some step polynomial $f' \in M_{[k]}$ of complexity at most $D$.  That is, in the decomposition~(\ref{eq:PDceE-struct}) one may choose $f_{[k]}$ to have complexity at most $D$.}

\vspace{7pt}

Thus, if a function solves a {\PDE}, and can be well-approximated by a step polynomial of a certain complexity, then it agrees with a step polynomial {\PDE}-solution which itself has a bound on its complexity, up to a degenerate solution.  Since we will see that any element of $\F(Z)$ may be approximated in $d_0$ by step polynomials (this will follow easily from Lemma~\ref{lem:q-p-fine}), in principle this can be applied to a set of representatives for each class in $M_{[k]}/\partial_k(M^{(k-1)})$, and so gives a quantitative enhancement of the first conclusion of Theorem A.

It will be clear that the analog of Theorem C holds also for zero-sum tuples, and can be proved in the same way, but we will not give those details separately.  The same argument should also work for other compact-by-discrete target modules $A$, but we will also set this aside for the sake of brevity.

Interestingly, there do not seem to be analogous quantitative versions of the second parts of Theorems A and B.  In the setting of Theorem A, if $Z$ and each $U_i$ are fixed compact Abelian Lie groups, then given a step polynomial $f \in \partial_k(M^{(k-1)})$, one can bound the minimal complexity of pre-images $(f_1,\ldots,f_k) \in \partial_k^{-1}\{f\}$ in terms of the complexity of $f$.  However, it can happen that this bound must deteriorate as one considers increasingly `complicated' Lie groups $Z$ and $\bfU$, and such a bound is actually impossible for some infinite-dimensional $Z$ and $\bfU$.  This will be discussed further, and witnessed by examples, in Subsection~\ref{subs:cplx-discussion}.

Since our proof of Theorem C is by compactness and contradiction, it does not give explicit bounds.  One could presumably extract such bounds by making all steps of the proof suitably quantitative, but even then one would expect them to be extremely poor.  On the other hand, even if one ignores the topological aspects of the proof, these results have a similar flavour to more classical structure theorems in real semi-algebraic algebraic geometry, where quantitative bounds are well-known to grow extremely rapidly: see, for instance, the classical monograph~\cite{BenedettiRis90}, or~\cite{BasPolRoy06}.

\section{Background and basic definitions}

We shall refer to~\cite{Aus--PDceEI} as `Part I', and will freely use the definitions and results of that paper.  They will be cited by prepending `I' to their numbering in that paper: for instance, `Theorem I.$X$.$Y$' means `Theorem $X$.$Y$ in Part I'.

\subsection{Compact and Polish Abelian groups}

Like Part I, this work will focus on measurable functions defined on a compact metrizable Abelian group. Assuming metrizability incurs no loss of generality, as explained in Section I.1.2.

A function from such a group to a Polish space is `measurable' if it is measurable with respect to the Haar-measure completion of the Borel $\s$-algebra; it is `Borel' if it is Borel measurable.  We will freely use the Measurable Selector Theorem for the former notion of measurability.

Many of our functions will take values in an Abelian Lie group.  By this we understand a locally compact, second-countable Abelian group whose topology is locally Euclidean.  It need not be connected or compactly generated.

If $Z$ is a compact metrizable group and $M$ is any Polish Abelian group, then, as in Part I, the space of Haar-a.e. equivalence classes of measurable functions $Z\to M$ is denoted $\F(Z,M)$.  It is regarded as another Polish Abelian group with the topology of convergence in probability.  We also abbreviate $\F(Z,\bbT) =: \F(Z)$.

If $Z$ is a compact Abelian group, then an \textbf{affine function} on $Z$ is a function $f:Z\to \bbT$ of the form $\theta + \chi$ for some $\theta \in \bbT$ and $\chi \in \hat{Z}$.  Since affine functions are continuous, two of them can be equal Haar-a.e. only if they are strictly equal, so we will generally identify them with their Haar-a.e. equivalence classes. Having done so, they form a closed $Z$-submodule $\A(Z) \leq \F(Z)$.  More generally, if $Y$ is another compact Abelian group then an \textbf{affine map} $f:Z\to Y$ is of the form $f(z) = y_0 + \chi(z)$ for some $y_0 \in Y$ and continuous homomorphism $\chi:Z \to Y$.  It is an \textbf{affine isomorphism} if $\chi$ is an isomorphism.

A \textbf{torus} is a compact group isomorphic to $\bbT^d = \bbR^d/\bbZ^d$ for some $d$. For $\bbT^d$ itself there is a canonical choice of fundamental domain, $[0,1)^d \subset \bbR^d$.  For any $t \in \bbT^d$, we let $\{t\}$ denote its unique representative in $[0,1)^d$.  If $Z$ is a torus and $\chi:Z\to \bbT^d$ is a choice of affine isomorphism, then $\{\chi\}:Z\to [0,1)^d$ denotes the composition of $\chi$ with $\{\cdot\}$.

An Abelian Lie group will be called \textbf{compact-by-discrete} if it is isomorphic to $\bbT^d \times D$ for some $d\geq 0$ and some discrete Abelian group $D$.  For this group, the short exact sequence
\[\bbT^d \into \bbT^d \times D \onto D\]
is preserved by any continuous automorphism, because the copy of $\bbT^d$ inside the middle group equals the identity component. Thus, for any compact $Z$ and any $Z$-action on $\bbT^d\times D$, the resulting module is a discrete extension of a compact $Z$-module, hence the name.

\subsection{Functions, sets and partitions}

Let $X$, $Y$ and $Z$ be sets, and let $\chi:X\to Y$ and $\g:X\to Z$ be functions.  Then $\chi$ \textbf{factorizes through} $\g$ if $\chi = f\circ \g$ for some $f:Z\to Y$; equivalently, if the level-set partition of $\g$ refines that of $\chi$.

If $\frP$ and $\frQ$ are two partitions of any set, then $\frP\vee \frQ$ denotes their common refinement, and $\frP \preceq \frQ$ denotes that $\frQ$ is already a refinement of $\frP$.  If $\frP$ is a partition of a set $S$ and $T \subseteq S$, then
\[\frP\cap T:= \{C\cap T\,|\ C\in \frP\},\]
a partition of $T$.   If $\frP$ is a partition of an Abelian group $Z$ and $z \in Z$, then
\[\frP - z:= \{C - z\,|\ Z \in \frP\}.\]
Given also a subgroup $W \leq Z$, the partition $\frP$ is \textbf{$W$-invariant} if $\frP - w = \frP$ for all $w \in W$; of course, this does not require that the individual cells of $\frP$ be $W$-invariant.

If $\frP$ and $\frQ$ are Borel partitions of a compact Abelian group $Z$, then $\frP$ \textbf{almost refines} $\frQ$ if there is a Borel subset $Y \subseteq Z$ with $m_Z(Y) = 1$ and $\frP\cap Y\succeq \frQ\cap Y$; they are \textbf{almost equal} if $\frP\cap Y = \frQ\cap Y$ for such a $Y$.

Relatedly, if $Z$ is a set and $\cal{U}$ and $\cal{V}$ are any covers of it (not necessarily partitions), then $\cal{V}$ is \textbf{subordinate} to $\cal{U}$ if for every $U \in \cal{U}$ there is a $V \in \cal{V}$ such that $U \subseteq V$. This will also be denoted by $\cal{U}\succeq \cal{V}$, as it is obviously equivalent to refinement in case $\cal{U}$ and $\cal{V}$ are partitions.

Lastly, if $\frP$ is a partition of $S$ then $\sim_\frP$ denotes the corresponding equivalence relation on $S$:
\[s \sim_\frP t \quad \Longleftrightarrow \quad s,t\ \hbox{lie in same cell of}\ \frP.\]

\section{Step polynomials}\label{sec:QP}

The section introduces the `step polynomials' that appear in the formulations of Theorems A and B, and builds up some basic theory for them.

\subsection{Quasi-polytopal partitions and step-affine maps}

Our convention below is that a convex polytope in a Euclidean space may contain some of its facets and not others, and may lie in a proper affine subspace.

\begin{dfn}[Quasi-polytopal partitions and step functions]\label{dfn:qp}
If $Z$ is a compact Abelian group and $\frP$ is a partition of $Z$, then $\frP$ is \textbf{quasi-polytopal} (`\textbf{QP}') if there are
\begin{itemize}
\item an affine map $\chi:Z \stackrel{\cong}{\to} \bbT^d$,
\item and a partition $\frQ$ of $[0,1)^d$ into convex polytopes
\end{itemize}
such that $\frP \preceq \{\chi\}^{-1}(\frQ)$.

A subset $C \subseteq Z$ is \textbf{QP} if the partition $\{C,Z\setminus C\}$ is QP

For any set $S$, a function $f:Z\to S$ is \textbf{step} if its level-set partition $\{f^{-1}\{s\}\,|\ s\in S\}$ is QP
\end{dfn}

The following properties are immediate.

\begin{lem}\label{lem:qp-obvious}
Quasi-polytopal partitions have the following properties:
\begin{itemize}
\item If $\theta:Z\to Y$ is affine and $\frP$ is a QP partition of $Y$, then $\theta^{-1}(\frP)$ is a QP partition of $Z$.

\item If $\frP \preceq \frQ$ are partitions of $Z$ such that $\frQ$ is QP, then $\frP$ is QP \qed
\end{itemize}
\end{lem}

\begin{lem}\label{lem:qp-join}
If $\frP_1$ and $\frP_2$ are two QP partitions of $Z$, then $\frP_1\vee \frP_2$ is also QP.
\end{lem}

\begin{proof}
For $i=1,2$, let $\chi_i:Z\to \bbT^{d_i}$ be affine maps and $\frQ_i$ be partitions of $[0,1)^{d_i}$ into convex polytopes such that $\frP_i \preceq \{\chi_i\}^{-1}(\frQ_i)$.  Then $(\chi_1,\chi_2):Z\to \bbT^{d_1 + d_2}$ is also affine; the product partition $\frQ_1\otimes \frQ_2$ (whose cells are products of cells from $\frQ_1$ and $\frQ_2$) is a partition of $[0,1)^{d_1+ d_2}$ into convex polytopes; and $\frP_1\vee \frP_2 \preceq \{(\chi_1,\chi_2)\}^{-1}(\frQ_1\otimes \frQ_2)$.
\end{proof}

\begin{lem}\label{lem:qp-on-cpts}
If $\frP$ is a partition of a compact Abelian Lie group $Z$, and $Z_0 \leq Z$ is the identity component, then $\frP$ is QP if and only if $(\frP - z)\cap Z_0$ is a QP partition of $Z_0$ for every $z \in Z$.
\end{lem}

\begin{proof}
If $\frP$ is QP and $z \in Z$, then the map $\theta:Z_0\to Z:z_0\mapsto z_0 + z$ is affine, and $(\frP - z)\cap Z_0 = \theta^{-1}(\frP)$, so this direction follows from Lemma~\ref{lem:qp-obvious}.

On the other hand, let $\frP'$ be the partition of $Z$ into cosets of $Z_0$, and let $\chi:Z\to \bbT^d$ be a homomorphism whose kernel equals $Z_0$ (such exists, because characters separate points). Then $\chi(Z)$ is a finite subgroup of $\bbT^d$, and so if $\frQ'$ is a partition of $[0,1)^d$ into sufficiently small boxes, then $\frP' = \{\chi\}^{-1}(\frQ')$.

Let $z_1$, \ldots, $z_m$ be a cross-section of $Z_0$ in $Z$. If $(\frP - z_i)\cap Z_0$ is QP for each $i$, then there are affine maps $\theta_i:Z_0 \to \bbT^{D_i}$ and convex polytopal partitions $\frQ_i''$ of $[0,1)^{D_i}$ such that $(\frP - z_i)\cap Z_0\preceq \{\theta_i\}^{-1}(\frQ_i'')$.  Let $\theta_i':Z\to \bbT^{D_i}$ be an extension of $\theta_i$ to all of $Z$, and let $\theta_i''$ be the composition of $\theta_i'$ with rotation by $z_i$.  Then
\[\frP \preceq \frP' \vee \bigvee_{i\leq m}\{\theta_i''\}^{-1}(\frQ_i''),\]
and this latter is QP by Lemma~\ref{lem:qp-join}, so $\frP$ is also QP.
\end{proof}

\begin{lem}\label{lem:q-p-fine}
If $\cal{U}$ is an open cover of a compact metrizable Abelian group $Z$, then it is subordinate to some QP partition.
\end{lem}

\begin{proof}
Choosing a suitable metric on $Z$, the structure theory for compact Abelian groups (see, in particular,~\cite[Theorem 9.5]{HewRos79}) and Lebesgue's Number Lemma imply that there are a Lie-group quotient $q:Z\to Z_1$ and an open cover $\cal{V}$ of $Z_1$ such that $\cal{U}$ is subordinate to $q^{-1}(\cal{V})$.  This justifies assuming that $Z$ is a Lie group.  In that case a partition of $Z$ into connected components, and then of each connected component into sufficiently small boxes, has the desired property, by another appeal to Lebesgue's Number Lemma.
\end{proof}

We will next introduce a class of maps that respect QP partitions.  We will first define the corresponding maps on Euclidean spaces, before returning to compact groups.

\begin{dfn}\label{dfn:Euc-step-aff}
If $Q \subseteq \bbR^d$ is a convex polytope in a Euclidean space, then a map $f:Q \to \bbR^r$ is \textbf{step-affine} if there are a partition $\frQ$ of $Q$ into further convex polytopes, and, for each $C \in \frQ$, an affine map $\ell_C:\bbR^d\to \bbR^r$ such that $f|C = \ell_C|C$.
\end{dfn}

\begin{lem}
If $f:Q\to \bbR^r$ is step-affine and $A:\bbR^r\to \bbR^s$ is affine, then $A\circ f$ is step-affine. \qed
\end{lem}

\begin{lem}\label{lem:combine-Euc-step-aff}
If $f:Q\to \bbR^r$ and $g:Q\to \bbR^s$ are step-affine, then so is $(f,g):Q\to \bbR^{r+s}$ and, in case $r=s$, so is $f+g:Q\to \bbR^r$.
\end{lem}

\begin{proof}
If two partitions of $Q$ consist of convex polytopes, then so does their common refinement.
\end{proof}

\begin{lem}\label{lem:Euc-qp-preim}
If $f:Q\to \bbR^s$ is step-affine and $\frR$ is a pairwise-disjoint collection of convex polytopes in $\bbR^s$ which covers $f(Q)$, then $f^{-1}(\frR)$ is a partition of $Q$ which has a refinement into convex polytopes.
\end{lem}

\begin{proof}
Letting $\frQ$ be as in Definition~\ref{dfn:Euc-step-aff}, it is obvious that $C\cap f^{-1}(\frR)$ consists of convex polytopes for each $C\in \frQ$.
\end{proof}

\begin{cor}\label{cor:Euc-compos-step}
If $f:Q\to \bbR^s$ and $g:R\to \bbR^p$ are step-affine and $f(Q) \subseteq R\subseteq \bbR^s$, then $g\circ f$ is step-affine. \qed
\end{cor}

\begin{lem}\label{lem:Euc-qp-im}
If $f:Q\to \bbR^s$ is step-affine and $D \subseteq Q$ is a finite union of convex sub-polytopes of $Q$, then $f(D)$ is also  finite union of convex polytopes.
\end{lem}

\begin{proof}
Clearly we may assume that $D$ is a single convex polytope. Let $\frQ$ be the partition appearing in Definition~\ref{dfn:Euc-step-aff}.  Then $f(C\cap D)$ is an affine image of a convex polytope for each $C \in \frQ$, so their union has the required form.
\end{proof}

In the first place, the relevance of step-affine maps to the setting of compact Abelian Lie groups stems from the following.

\begin{lem}\label{lem:little-step-aff-coords}
If $\a:\bbT^d\to \bbT^s$ is an affine map, then there is a step-affine map $f:[0,1)^d\to [0,1)^s$ which makes the following diagram commute:
\begin{center}
$\phantom{i}$\xymatrix{
\bbT^d \ar_{\{\cdot\}}[d]\ar^\a[r] & \bbT^s \ar^{\{\cdot\}}[d] \\
[0,1)^d \ar_f[r] & [0,1)^s.
}
\end{center}
\end{lem}

\begin{proof}
By Lemma~\ref{lem:combine-Euc-step-aff}, we may argue coordinate-wise, and hence assume $s = 1$.  In this case there are $\theta \in \bbT$ and $n_1,\ldots,n_d \in \bbZ$ such that
\[\a(t_1,\ldots,t_d) = \theta + n_1t_1 + \ldots + n_dt_d,\]
using the obvious coordinates in $\bbT^d$, and hence
\[\{\a\}(t_1,\ldots,t_d) = \{\theta + n_1t_1 + \ldots + n_dt_d\}.\]
Let $\frQ$ be the partition of $[0,1)^d$ into the sets
\[C_m:= \{(x_1,\ldots,x_d) \in [0,1)^d\,|\ \{\theta\} + n_1x_1 + \ldots + n_dx_d \in [m,m+1)\}\]
for all integers $-|n_1| - \cdots - |n_d| -1 \leq m \leq |n_1| + \cdots + |n_d| + 1$. This is a partition into convex polytopes, and on the set $C_m$ the desired function $f$ agrees with the affine function
\[\ell_{C_m}(x_1,\ldots,x_d) = \{\theta\} + n_1x_1 + \ldots + n_dx_d - m.\]
\end{proof}

This has the some useful corollaries.

\begin{cor}[Freedom in choice of coordinates]
If $\frP$ is a QP partition of $\bbT^d$, then there is a partition $\frQ$ of $[0,1)^d$ into convex polytopes such that $\frP \preceq \{\cdot\}^{-1}(\frQ)$.
\end{cor}

\begin{proof}
Definition~\ref{dfn:qp} gives an affine map $\a:\bbT^d\to \bbT^s$ and a partition $\frR$ of $[0,1)^s$ into convex polytopes such that $\frP \preceq \{\a\}^{-1}(\frR)$.  Now Lemma~\ref{lem:little-step-aff-coords} gives a step-affine map $f:[0,1)^d\to [0,1)^s$ which makes the diagram in that lemma commute.  Having done so, Lemma~\ref{lem:Euc-qp-preim} gives a convex-polytopal partition $\frQ\succeq f^{-1}(\frR)$, and so the commutativity of that diagram gives $\frP \preceq \{\cdot\}^{-1}(f^{-1}(\frR)) \preceq \{\cdot\}^{-1}(\frQ)$.
\end{proof}

\begin{cor}\label{cor:conn-cpts-still-qp}
If $\frP$ is a QP partition of a compact Abelian Lie group $Z$, and $\frP_1$ is the refinement consisting of all connected components of cells of $\frP$, then $\frP_1$ is also QP
\end{cor}

\begin{proof}
First, if $Z_0 \leq Z$ is the identity component, and $\frR$ is the partition of $Z$ into $Z_0$-cosets, then clearly $\frP_1 \succsim \frP\vee \frR$, and $\frP\vee \frR$ is still QP  We may therefore consider $Z_0$-cosets separately, or, equivalently, assume that $Z = Z_0$. In this case, assuming $Z = \bbT^d$, if $\frQ$ is the partition provided by the previous corollary, then every cell of $\frQ$ is connected, and the map $\{\cdot\}^{-1}:[0,1)^d\to \bbT^d$ is just the restriction of the covering map $\bbR^d\to \bbT^d$, hence continuous, so $\{\cdot\}^{-1}(\frQ)$ also refines $\frP_1$.
\end{proof}

\begin{rmk}
The previous lemma and corollary are false for general compact Abelian groups. For example, for the totally disconnected group $Z: = (\bbZ/2\bbZ)^\bbN$, the trivial partition $\{Z\}$ is QP, but the partition into singletons is certainly not. \fin
\end{rmk}

\begin{cor}\label{cor:im-still-qp}
If $\a:Z\to Y$ is an affine map of compact Abelian groups with $Y$ a Lie group, and $C \subseteq Z$ is a QP set, then $q(C) \subseteq Y$ is a QP set.
\end{cor}

\begin{proof}
Since $C$ is QP, it is lifted from some Lie group quotient of $Z$.  Also, $\a$ must factorize through such a quotient of $Z$, since $Y$ is Lie.  Combining these two quotients, we may assume $Z$ itself is Lie.

Having done so, the result will be true for $C$ if it is true for $C\cap (z + Z_0)$ for every identity-component coset $z + Z_0 \leq Z$, so we may assume $C \subseteq Z_0$, and hence that $Z$ is a torus.  This implies that the image $\a(Z)$ lies in an identity-component coset of $Y$, so we may assume that $Y$ is also a torus.

Finally, letting $Z = \bbT^d$ and $Y = \bbT^s$, the result follows by Lemma~\ref{lem:Euc-qp-im} and the commutative diagram of Lemma~\ref{lem:little-step-aff-coords}.
\end{proof}

We can now proceed from Definition~\ref{dfn:Euc-step-aff} to the following.

\begin{dfn}[Step-affine map]\label{dfn:step-aff}
Let $Z$ be a compact Abelian group.

A map $f:Z\to \bbR^s$ is \textbf{step-affine} if it equals $f_0\circ \{\chi\}$ for some affine map $\chi:Z\to \bbT^d$ and step-affine map $f_0:[0,1)^d\to \bbR^s$.

A map $f:Z\to\bbT$ is \textbf{step-affine} if it equals $\psi\circ f_1$ for some homomorphism $\psi:\bbR\to\bbT$ and some step-affine $f_1:Z\to\bbR$.

Finally, if $Y$ is another compact Abelian group, then a map $f:Z\to Y$ is \textbf{step-affine} if $\chi\circ f$ is step-affine for every $\chi \in \A(Y)$.
\end{dfn}

Let $\F_{\rm{sa}}(Z,Y)$ denote the set of Haar-a.e. equivalence classes of step-affine functions $Z\to Y$, for $Y$ equal to either $\bbR^s$ or another compact Abelian group.

In case $f:Z\to \bbR$ is step-affine, the QP partition implicit in its definition will be said to \textbf{control} $f$.

The following lemma is immediate from the properties of Euclidean step-affine maps established above.

\begin{lem}\label{lem:step-aff-obvious}
Step-affine functions have the following properties.
\begin{itemize}
\item Sums and scalar multiples of step-affine functions $Z\to\bbR^s$ are still step-affine.
\item If $Y$ is compact Abelian, then sums of step-affine functions $Z\to Y$ are still step-affine, and so $\F_{\rm{sa}}(Z,Y)$ is a subgroup of $\F(Z,Y)$.
\item If $Y$ is compact Abelian and $S \subseteq \hat{Y}$ is a generating set for $\hat{Y}$, then $\psi:Z\to Y$ is step-affine if and only if $\chi\circ\psi$ is step-affine for every $\chi \in S$.
\item If $Y$ is a compact Abelian Lie group, then any step-affine map $Z\to Y$ factorizes through a Lie-group quotient of $Z$.
\item If $f:Z\to \bbR^s$ is step-affine, then so is $\{f\}$.
\item If $f:Z\to\bbT$, then $f$ is step-affine if and only if $\{f\}:Z\to [0,1)$ is step-affine.
\item If $f:Z\to \bbR$ is step-affine, and $\frR$ is a locally finite partition of $\bbR$ into intervals, then $f^{-1}(\frR)$ is a QP partition of $Z$. \qed
\end{itemize}
\end{lem}

\begin{lem}
If $Y$ and $Z$ are compact Abelian groups, then affine maps $Z\to Y$ are step-affine.
\end{lem}

\begin{proof}
Composing with a character, it suffices to prove this in case $Y = \bbT$.  However, if $\chi \in \A(Z)$, then
\[\chi(z) = \{\chi(z)\} + \bbZ,\]
and $\{\chi\}:Z\to\bbR$ is clearly step-affine.
\end{proof}

The next result is a natural extension of Lemma~\ref{lem:little-step-aff-coords}

\begin{lem}\label{lem:step-aff-coords}
If $Z$ is a compact Abelian Lie group and $f:Z\to \bbT^d$ is step-affine, then there are an affine map $\chi:Z\to \bbT^D$ and a step-affine map $f_0:[0,1)^D \to [0,1)^d$ such that the following diagram commutes:
\begin{center}
$\phantom{i}$\xymatrix{
Z \ar^f[r] \ar_{\{\chi\}}[d] & \bbT^d \ar^{\{\cdot\}}[d]\\
[0,1)^D \ar_{f_0}[r] & [0,1)^d.
}
\end{center}
\end{lem}

\begin{proof}
Arguing coordinate-wise, it suffices to prove this when $d = 1$, in which case $\{f\}:Z\to \bbR$ is step-affine.  The desired conclusion is now just the definition of step-affine maps to $\bbR$.
\end{proof}

\begin{lem}\label{lem:step-aff-compos}
If $X$, $Y$ and $Z$ are compact Abelian groups, and $\phi:X\to Y$ and $\psi:Y\to Z$ are step-affine, then so is $\psi\circ \phi:X\to Z$.
\end{lem}

\begin{proof}
We treat this in two steps.

\vspace{7pt}

\emph{Step 1.}\quad If $\psi$ is actually affine, and $\chi \in\A(Z)$, then $\chi\circ \psi \in \A(Y)$, so in this case the result follows at once from the definition.

\vspace{7pt}

\emph{Step 2.}\quad For the general case, after composing with a character we may assume $Z = \bbT$. Having done so, by definition there are an affine map $\chi:Y\to \bbT^d$ and a step-affine map $\psi_0:\bbT^d\to \bbT$ such that $\psi = \psi_0\circ \chi$.  Replacing $\phi$ with $\chi \circ \phi$ (justified by Step 1) and $\psi$ with $\psi_0$, we may therefore assume that $Y = \bbT^d$.

With this assumption, Definition~\ref{dfn:step-aff} gives a commutative diagram
\begin{center}
$\phantom{i}$\xymatrix{
X \ar^\phi[rr] && \bbT^d \ar^\psi[rr]\ar^{\{\theta\}}[dr] && \bbT \\
& && [0,1)^s \ar_{f_2}[rr] && \bbR \ar_{\mod 1}[ul]}
\end{center}
in which $\theta:\bbT^d\to \bbT^s$ is affine and $f_2$ is step affine.  Now Lemmas~\ref{lem:little-step-aff-coords} and~\ref{lem:step-aff-coords} enable one to enlarge this to a commutative diagram
\begin{center}
$\phantom{i}$\xymatrix{
&X \ar^\phi[rr]\ar_{\{\chi\}}[dl] && \bbT^d \ar^\psi[rr]\ar_{\{\cdot\}}[dl]\ar^{\{\theta\}}[dr] && \bbT \\
[0,1)^D \ar_{f_0}[rr] && [0,1)^d \ar_{f_1}[rr] && [0,1)^s \ar_{f_2}[rr] && \bbR \ar_{\mod 1}[ul]}
\end{center}
in which also $\chi:X\to \bbT^D$ is affine and $f_0$ and $f_1$ are step affine.

Reading around the bottom row of this diagram, the proof is completed by an appeal to Corollary~\ref{cor:Euc-compos-step}.
\end{proof}

\begin{lem}\label{lem:qp-partn-preim}
If $\frP$ is a QP partition of $Y$ and $\psi:Z\to Y$ is step-affine, then $\psi^{-1}(\frP)$ is a QP partition of $Z$.
\end{lem}

\begin{proof}
By definition we may assume that $\frP = \{\chi\}^{-1}(\frQ)$ for some affine map $\chi:Y\to \bbT^d$ and some partition $\frQ$ of $[0,1)^d$ into convex polytopes. Replacing $\psi$ with $\chi\circ \psi$, we may therefore assume $Y = \bbT^d$. Now the definitions and Lemma~\ref{lem:step-aff-coords} give a commutative diagram
\begin{center}
$\phantom{i}$\xymatrix{
Z \ar_{\{\theta\}}[d]\ar^\psi[r] & \bbT^d \ar^{\{\cdot\}}[d]\\
[0,1)^D \ar_{f_1}[r] & [0,1)^d
}
\end{center}
for some affine $\theta$ and step-affine $f_1$. Reading counterclockwise around this diagram and applying Lemma~\ref{lem:Euc-qp-preim} completes the proof.
\end{proof}

\subsection{Step-affine cross-sections}

One important way in which step-affine maps appear naturally is as cross-sections of quotient homomorphisms.

\begin{lem}\label{lem:fund-lift}
Suppose that $W \leq Z$ is an inclusion of tori, and let $q:Z \to Z/W$ be the quotient homomorphism.  Then $q$ has a step-affine cross-section $\s:Z/W\to Z$.
\end{lem}

\begin{proof}
We may identify $Z = \bbT^d$, and then the quotient $Z/W$ may be identified with another torus, say $\bbT^r$.

Let $\pi:\bbR^d\to \bbT^d$ be the universal cover, and let $Q:\bbR^d\to \bbR^r$ be the lift of $q$ to the universal covers.  Then $Q$ is a surjective linear map, so it has a linear cross-section $A:\bbR^r\to \bbR^d$.  Let $\frP$ be the partition of $\bbT^r$ that is the pullback under $A\circ \{\cdot\}$ of the partition of $\bbR^d$ into half-open unit cubes.  Then $\frP$ is QP, and for each $D \in \frP$ the map $(A\circ \{\cdot\})|D$ has image contained in a single fundamental domain of $\pi$.  Composing with $\pi$ identifies each of these restrictions of $A\circ \{\cdot\}$ with a map $D\to \bbT^d$, which gives a suitable step-affine cross-section.
\end{proof}

\begin{prop}\label{prop:fund-lift}
Suppose that $W \leq Z$ is an inclusion of compact Abelian groups, and let $q:Z\to Z/W$ be the quotient.  Then $q$ has a step-affine cross-section $Z/W\to Z$, and there is a step-affine $W$-equivariant map $Z\to W$.
\end{prop}

We will first prove this in a special case, and then use that case to prove the result in general.

\begin{proof}[Proof in special case]
Assume that $Z/W$ is finite-dimensional (that is, topologically isomorphic to a subgroup of a torus).  We will find a step-affine cross-section $\s:Z/W\to Z$ under this assumption.  The structure theory for compact Abelian groups (see~\cite[Theorem 9.5]{HewRos79}) gives a decreasing sequence of closed subgroups $U_1 \geq U_2 \geq \cdots$ in $Z$ whose intersection is $\{0\}$, and such that each $Z/U_i$ is finite-dimensional.  Intersecting each of them with $W$ if necessary, we may assume that $U_0 := W \geq U_1$.  Let $q_i:Z\to Z/U_i$ be the quotient homomorphism, and similarly $q_{ij}:Z/U_i\to Z/U_j$ whenever $i \geq j$.  By Lemma~\ref{lem:fund-lift}, there are step-affine cross-sections
\[Z/W \stackrel{\s_1}{\to} Z/U_1 \stackrel{\s_2}{\to} Z/U_2 \stackrel{\s_3}{\to} \cdots\]
of $q_{21}$, $q_{32}$, \ldots.  It is now routine to check that the partial compositions
\[\s_{1,j}:= \s_j\circ \s_{j-1}\circ \cdots\circ \s_1:Z/W\to Z/U_j\]
converge uniformly.  Their limit $\s:Z/W\to Z$ is a step-affine cross-section, because any $\chi \in \A(Z)$ is a lift of some $\chi' \in \A(Z/U_i)$, hence
\[\chi\circ \s = \chi'\circ \s_{1,i},\]
and $\s_{1,i}$ is step-affine by construction.
\end{proof}

\begin{lem}\label{lem:consistent-fund-lift}
Suppose that $Z$ is a compact Abelian group and that $U,W \leq Z$ are closed subgroups such that both $Z/U$ and $Z/W$ are finite-dimensional.  Suppose also that $\tau:Z/(U+W) \to Z/U$ is a step-affine cross-section for the relevant quotient.  Then there is a step-affine cross-section $\s:Z/W\to Z$ which makes the following diagram commute:
\begin{center}
$\phantom{i}$\xymatrix{
Z \ar@{->>}^-{q_1}[r] & Z/U\\
Z/W \ar@{->>}^-{q_2}[r]\ar^\s[u] & Z/(U+W) \ar_\tau[u],
}
\end{center}
where $q_1$ and $q_2$ are the relevant quotient maps.
\end{lem}

\begin{proof}
Let $\s_1:Z/W \to Z$ be any step-affine cross-section as given by the special case of Proposition~\ref{prop:fund-lift} proved above.  Then $q_1\circ \s_1 - \tau\circ q_2$ must take values in the subgroup $(W+U)/U \leq Z/U$, which is canonically isomorphic to $W/(W\cap U)$.  Let $\g:(W+U)/U \cong W/(W\cap U)\to W$ be a step-affine cross-section, again by the special case of Proposition~\ref{prop:fund-lift} proved above.  Now
\[\s := \s_1 - \g\circ (q_1\circ \s_1 - \tau\circ q_2)\]
is a step-affine cross-section $Z/W\to Z$ which gives the desired commutativity.
\end{proof}

\begin{proof}[Proof of Proposition~\ref{prop:fund-lift}]
Now consider general $Z$ and $W$.  Let $U_1 \geq U_2 \geq \cdots $ be as in the proof of the special case, and let $W_i := W+U_i$ for each $i$.  This is a decreasing sequence of closed subgroups of $Z$ whose intersection is $W$, and such that $Z/W_i$ is finite-dimensional for each $i$.

By recursively applying Lemma~\ref{lem:consistent-fund-lift}, we can choose step-affine cross-sections $\s_i:Z/W_i\to Z/U_i$ such that the following diagram commutes:
\begin{center}
$\phantom{i}$\xymatrix{
Z \ar@{->>}[r] & \cdots \ar@{->>}[r] & Z/U_2 \ar@{->>}[r] & Z/U_1 \\
Z/W \ar@{->>}[r] & \cdots \ar@{->>}[r] & Z/W_2 \ar@{->>}[r]\ar_{\s_2}[u] & Z/W_1. \ar_{\s_1}[u]
}
\end{center}

Letting $q_i:Z/W\to Z/W_i$ be the quotient maps, it now follows as in the proof of the special case that the maps $\s_i\circ q_i:Z/W\to Z/U_i$ converge uniformly to some cross-section $\s:Z/W\to Z$, and that $\s$ is still step-affine.

Finally, we obtain a step-affine $W$-equivariant map $\xi:Z\to W$ by setting
\[\xi(z) := z - \s(q(z)).\]
If $\s$ is step-affine, then so is this, by Lemmas~\ref{lem:step-aff-compos} and~\ref{lem:step-aff-obvious}.
\end{proof}

\begin{cor}\label{cor:fund-split}
Whenever $W\leq Z$ is an inclusion of compact Abelian groups, there is a $W$-equivariant bijection
\[\psi:(Z/W)\times W\to Z\]
such that both $\psi$ and $\psi^{-1}$ are step-affine.
\end{cor}

\begin{proof}
Letting $\s:Z/W\to Z$ be a step-affine cross-section, as produced by Proposition~\ref{prop:fund-lift}, a suitable bijection is given by
\[\psi(\ol{z},w) = \s(\ol{z}) + w \quad \hbox{for}\ \ol{z} \in Z/W,\ w \in W.\]
\end{proof}

\begin{ex}
For the two-fold covering homomorphism
\[\bbT\stackrel{\times 2}{\to} \bbT,\]
one has the obvious setp-affine selector
\[\theta \mapsto e(\{\theta\}/2):\bbT \to e([0,1/2)) \subset \bbT,\]
where $e:\bbR\to\bbT$ is the usual quotient homomorphism. \fin
\end{ex}

\begin{ex}
The restriction map arising from the inclusion $\bbZ\subset \bbQ$ defines a homomorphism
\[\hat{\bbQ}\to \hat{\bbZ}\cong \bbT,\]
where $\bbQ$ is given its discrete topology so that $\hat{\bbQ}$ is compact ($\hat{\bbQ}$ is sometimes called the \textbf{solenoid}). In this case one has also a `natural' selector, which to $\theta \in \bbT$ assigns the element of the solenoid defined by
\[\phi_\theta(p/q) := e(p\{\theta\}/q)\]
where $e$ is as in the previous example.  By Pontryagin duality, any $\chi \in \hat{\hat{\bbQ}}$ takes the form of evaluation at some $p/q \in \bbQ$, so the above formula shows that the composition $\theta\mapsto \phi_\theta\mapsto \chi(\phi_\theta)$ is step-affine.

Topologically, $\hat{\bbQ}$ is a bundle of copies of the Cantor set over $\rm{S}^1$, and the above selector embeds $[0,1)$ into this as a cross-section of the projection from the total space onto $\rm{S}^1$.  \fin
\end{ex}

\subsection{Step polynomials}\label{subs:step-poly}

Again let $Z$ be a compact metrizable Abelian group and $A$ an Abelian Lie group.  We next introduce step polynomials, which form a natural generalization of step affine maps.

\begin{dfn}\label{dfn:Euc-step-poly}
If $Q \subseteq \bbR^d$ is a convex polytope, then a function $f:Q\to \bbR^r$ is a \textbf{step polynomial} if there are a partition $\frP$ of $Q$ into convex sub-polytopes and, for each $C \in \frP$, a polynomial $p_C:\bbR^d\to \bbR^r$ such that $f|C = p_C|C$.  Such a choice of $\frP$ will be said to \textbf{control} $f$.

A step polynomial $f:Q\to \bbR^r$ is \textbf{basic} if $f = g\cdot 1_R$ for some polynomial $g:\bbR^d\to \bbR^r$ and convex sub-polytope $R \subseteq Q$.
\end{dfn}

\begin{dfn}[Step polynomial]\label{dfn:step-poly}
If $Z$ is a compact Abelian group and $A$ is an Abelian Lie group, then a map $f:Z\to A$ is a \textbf{step polynomial} if it is a composition
\[Z\stackrel{\{\chi\}}{\to} [0,1)^d \stackrel{f_0}{\to} \t{A} \stackrel{\psi}{\to} A,\]
where
\begin{itemize}
\item $\chi:Z\to \bbT^d$ is affine,
\item $\t{A}$ is a closed subgroup of $\bbR^r$ for some $r$,
\item $f_0:[0,1)^d\to \bbR^r$ is a step polynomial with image contained in $\t{A}$,
\item and $\psi:\t{A}\to A$ is a continuous homomorphism.
\end{itemize}

The step polynomial $f$ is \textbf{basic} if $f_0$ may be taken to be basic in the above definition.

The set of Haar-a.e. equivalence classes in $\F(Z,A)$ that contain step polynomials $Z\to A$ is denoted by $\F_\sp(Z,A)$.
\end{dfn}

Clearly any step polynomial decomposes as a finite sum of basic step polynomials.  Also, if $Z$ is an Abelian Lie group with identity component $Z_0 \leq Z$ and $f:Z\to \bbR^d$ is a step polynomial, then an easy exercise shows that $f\cdot 1_{z + Z_0}$ is also a step polynomial for every coset $z + Z_0$.  Combining these facts, if $Z$ is Lie then we may always decompose $f$ into basic step polynomials supported on single identity-component cosets.

\begin{rmk}
Having reached this definition, it is high time we drew attention to the overlap between this section and Bergelson and Leibman's work~\cite{BerLei07}. Their interest is in the study of bounded `generalized polynomials' from $\bbZ^d$ to $\bbR$.  These comprise the smallest class which contains the linear functions and is closed under addition, multiplication, and also taking integer parts.  Generalized polynomials arise naturally in various problems from equidistribution theory and additive combinatorics (see, for instance,~\cite{GreTaoZie12}, and the many further references in~\cite{BerLei07}).  The paper~\cite{BerLei07} develops a general structure theory for them.  It proves that every bounded generalized polynomial can be obtained by sampling along an orbit of a $\bbZ^d$-action by rotations on a compact nilmanifold, where the function sampled is essentially what our terminology would call a step polynomial on that nilmanifold.

Insofar as any compact connected nilmanifold can be described as a tower of topological circle extensions, and has a natural `coordinate system' which identifies it with some cube $[0,1)^d$, the study of such functions on nilmanifolds forms a natural generalization of our work on step polynomials on compact Abelian Lie groups. The concerns of~\cite{BerLei07} are fairly disjoint from ours, and so are the results, but it seems likely that most of the work of this section could be generalized to their setting.  It might even be worth looking for an abstract category of `step-polynomial spaces', whose objects are spaces that admit a coordinate system based on finitely many Euclidean cubes, and whose morphisms are an abstract characterization of `step polynomial mappings' between such spaces.  This would also bear comparison with the abstract study of `nilspaces' in~\cite{CamSze--nilspaces}. Such a theory would put us in the realm of quite general semi-algebraic geometry (see~\cite[Chapter 2]{BenedettiRis90} for a good introduction), but I do not know whether ideas from that theory could shed additional light on the kind of work that we will do below.  Our interest in {\PDE}s and zero-sum problems could also be generalized, by replacing the rotation-actions of subgroups of $Z$ with the actions of commuting nilpotent subgroups of a nilpotent Lie group $G$ on a compact nilmanifold $G/\G$ (a similar proposal was already discussed in Subsection I.10.4).  In this way, there might generalizations of Theorems A and B to that setting.  \fin
\end{rmk}

\begin{lem}
A sum of two $A$-valued step polynomials is a step polynomial, so $\F_\sp(Z,A)$ is a subgroup of $\F(Z,A)$.
\end{lem}

\begin{proof}
Let
\[Z\stackrel{\{\chi_i\}}{\to} [0,1)^{d_i} \stackrel{f_i}{\to} \t{A}_i \stackrel{\psi_i}{\to} A\]
for $i=1,2$ be factorizations of our two step polynomials as given by Definition~\ref{dfn:step-poly}.  Then their sum factorizes as
\[Z \ \ \stackrel{\{(\chi_1,\chi_2)\}}{\to}\ \ [0,1)^{d_1 + d_2} \ \ \stackrel{(f_1\pi_1,f_2\pi_2)}{\to} \ \ \t{A}_1\oplus \t{A}_2 \ \ \stackrel{\psi_1q_1 + \psi_2q_2}{\to} \ \ A,\]
where $\pi_i:[0,1)^{d_1 + d_2}\to [0,1)^{d_i}$ is the projection onto the first (resp. last) $d_i$ coordinates for $i=1$ (resp. $i=2$), and also $q_i:\t{A}_1\oplus \t{A}_2\to \t{A}_i$ are the coordinate projections.  This is clearly a factorization into ingredients of the required kind.
\end{proof}

\begin{lem}
If $f:Z\to A$ is a step polynomial, then there is a QP partition of $Z$ such that $f|C$ extends to a uniformly continuous function on $\ol{C}$ for every $C \in \frP$.
\end{lem}

\begin{proof}
Let $f = \psi\circ f_0\circ \{\chi\}$ be a factorization as in Definition~\ref{dfn:step-poly}, and let $\frQ_1$ be the convex polytopal partition of $[0,1)^d$ associated to $f_0$ as in Definition~\ref{dfn:Euc-step-poly}. Let $\frQ$ be a further convex polytopal refinement of $\frQ_1$ so that the map $\{\chi\}^{-1}$ restricts to a homeomorphism on each cell of $\frQ$. Then $\frP := \{\chi\}^{-1}(\frQ)$ is the desired QP partition of $Z$.
\end{proof}

\begin{cor}\label{cor:disc-step-poly-step-fn}
If $A$ is discrete, then a step polynomial $f:Z\to A$ is a step function.
\end{cor}

\begin{proof}
This follows by combining the preceding lemma and Corollary~\ref{cor:conn-cpts-still-qp}.
\end{proof}

\begin{lem}\label{lem:stepaffsteppoly}
If $\xi:Z\to Y$ is step-affine and $f:Y\to A$ is a step polynomial, then $f\circ\xi$ is a step polynomial.
\end{lem}

\begin{proof}
Clearly every step polynomial factorizes through an affine map to a torus, so we may assume $Y = \bbT^d$.

Let $f = \psi\circ f_0\circ \{\chi\}$ as in Definition~\ref{dfn:step-poly}.  Lemma~\ref{lem:step-aff-coords} provides a commutative diagram
\begin{center}
$\phantom{i}$\xymatrix{
Z \ar_{\{\theta\}}[d] \ar^\xi[r] & \bbT^d \ar_{\{\chi\}}[d]\ar^f[drr]\\
[0,1)^D \ar_{\xi_0}[r] & [0,1)^d \ar_-{f_0}[r] & \t{A} \ar_{\psi}[r] & A,
}
\end{center}
where $\theta$ is affine and $\xi_0$ is step-affine.  The proof is completed by observing that $f_0\circ \xi_0:[0,1)^D\to \t{A}$ is a step polynomial, which is clear from the definitions.
\end{proof}

%Using Proposition~\ref{prop:fund-lift}, the above lemma has the following useful converse.
%
%\begin{lem}\label{lem:strict-invar}
%If $q:Z \onto Y$ is a quotient homomorphism of compact metrizable Abelian groups, and $F:Z\to A$ is a step polynomial which is invariant under $\ker q$, then there is a step polynomial $f:Y\to A$ such that $F = f\circ q$.
%\end{lem}
%
%\begin{proof}
%Let $\s:Y\to Z$ be a step-affine cross-section as given by Proposition~\ref{prop:fund-lift} and let $f:= F\circ \s$.  This is still step-affine by Lemma~\ref{lem:stepaffsteppoly}.
%\end{proof}

Another fact we will use repeatedly is that step polynomials can be lifted through target-module homomorphisms.

\begin{lem}\label{lem:s-p-lift}
Let $q:A\onto B$ be a continuous epimorphism of Abelian Lie groups, and $Z$ a compact Abelian group.  For any step polynomial $f:Z\to B$ there is a step polynomial $F:Z\to A$ such that $f = q\circ F$.
\end{lem}

\begin{proof}
By the structure theory for locally compact Abelian groups (see~\cite[Section II.9]{HewRos79}), we may assume that
\[B = \bbR^r \oplus \bbT^d\oplus D,\]
with $D$ discrete. Correspondingly, one may decompose $f$ as $f_1 + f_2 + f_3$, where each summand takes values within one of the direct summands on the right above.  Each $f_i$ is the composition of $f$ with a coordinate-projection, so is still a step polynomial.  It therefore suffices to lift each $f_i$ separately.

Firstly, $f_3$ takes values in a discrete group.  It is therefore constant on each cell of some QP partition, so one may simply choose a lift of that constant value on each cell separately.

Second, $f_1$ takes values in $\bbR^r$.  Consider the analogous decomposition
\[A = \bbR^{r'} \oplus \bbT^{d'} \oplus D'.\]
Let $q_1:A\to \bbR^r$ be the composition of $q$ with the projection from $B$ onto its summand $\bbR^r$.  The image $q_1(\bbT^{d'})$ must be a compact subgroup of $\bbR^r$, hence it must equal $0$.  Also, the image $q_1(D')$ is countable, so since $q_1$ is onto, the image $q_1(\bbR^{r'})$ is a co-countable vector subspace of $\bbR^r$.  It is therefore equal to $\bbR^r$.  By linear algebra, this implies that there is a linear function $M:\bbR^r\to \bbR^{r'} \leq A$ which is a cross-section of $q_1$, and now $M\circ f_1$ is the desired lift of $f_1$.

Finally, let $q_2:A\to \bbT^d$ be the composition of $q$ with the projection from $B$ onto $\bbT^d$. The image $q_2(\bbT^{d'})$ is a closed, connected subgroup of $\bbT^d$.  Is is therefore a subtorus, and so we may split $\bbT^d$ further as $q(\bbT^{d'}) \oplus T$ for some complementary subtorus $T \leq \bbT^d$.  Correspondingly we may decompose $f_2 = f_{21} + f_{22}$ and lift each summand separately.  For $f_{21}$, Proposition~\ref{prop:fund-lift} gives a step-affine cross-section $\s:q_2(\bbT^{d'}) \to \bbT^{d'}$, so the composition $\s\circ f_{21}$ is a suitable lift of $f_{21}$.  For $f_{22}$, observe that the composition
\[\bbR^{r'} \leq A \stackrel{q}{\onto} B\onto \bbT^d \onto \bbT^d/q_2(\bbT^{d'}) \cong T\]
is surjective: its image is a $\s$-compact (hence Borel) and co-countable subgroup of $T$, so must equal $T$.  This gives a continuous epimorphism $\bbR^{r'}\to T$, which therefore lifts to a linear epimorphism of the universal covers $Q:\bbR^{r'}\to \t{T}$.  Composition $f_{22}$ with a step-affine fundamental domain $T\to \t{T}$ and then with a linear $Q$-cross-section $\t{T}\to \bbR^{r'}$ completes the proof.
\end{proof}

\begin{rmk}
This result has quietly made an appeal to our assumption of second-countability for Abelian Lie groups.  Without it, one could consider the uncountable group $\bbZ^{\oplus \bbT}$ with its discrete topology and the map
\[q:\bbZ^{\oplus \bbT}\to \bbT: (z_t)_{t \in \bbT}\mapsto \sum_{t\in\bbT}z_t\cdot t.\]
This $q$ is a continuous epimorphism, but the identity map $\bbT\to\bbT$ has no lift through $q$ to a step polynomial $\bbT\to \bbZ^{\oplus \bbT}$. \fin
\end{rmk}

\subsection{Slicing}

Suppose that $X$ and $Y$ are compact Abelian groups and that $A$ is an Abelian Lie group.  Given a Borel function $f:X\times Y\to A$, each of its slices $f(x,\,\cdot\,)$ is a Borel function $Y\to A$, so this defines a map $X\to \F(Y,A)$, which is easily seen to be measurable.  However, a step polynomial $f:X\times Y\to A$ cannot generally be viewed as a step function $X\to \F(Y,A)$, even if $A$ is discrete.

\begin{ex}\label{ex:bad-slicing}
If $f:\bbT^2\to \{0,1\}$ is the step function
\[(s,t)\mapsto \lfloor\{s\} + \{t\}\rfloor,\]
then regarded as a function $\bbT\to \F(\bbT,\bbZ)$ it is continuous and injective, so its level-set partition is not even finite, let alone quasi-polytopal. \fin
\end{ex}

However, it will be important later that this `slicing' operation gives rise to a function that is continuous on the cells of a QP partition.  This will be the main result of the present subsection.

\begin{lem}
Let $Q \subseteq \bbR^{d+r}$ be a bounded convex polytope, and suppose that $f = g\cdot 1_Q:\bbR^{d+r}\to \bbR$ is a basic step polynomial supported on $Q$.  Let $\Pi:\bbR^{d+r}\to \bbR^r$ be the projection onto the last $r$ coordinates.  Then the sliced function
\[\bbR^r\to \F(\bbR^d,\bbR):v \mapsto f(\cdot,v)\]
is identically zero on $\bbR^r\setminus \ol{\Pi(Q)}$ and is uniformly continuous on $\rm{int}\,\Pi(Q)$.  (Here, as usual, $\F(\bbR^d,\bbR)$ is the set of a.e. equivalence classes of maps $\bbR^d\to\bbR$ for Lebesgue measure, with the topology of convergence in probability on bounded sets.  The expression $f(\cdot,v)$ is to be understood as such an equivalence class.)
\end{lem}

\begin{proof}
The first assertion is obvious, so we focus on the second.

Since $g$ is uniformly continuous and uniformly bounded on the bounded set $\ol{Q}$, it suffices to prove this for $1_Q$ alone.

Let $B \supseteq Q$ be an open ball containing $\ol{Q}$.  By an easy calculation, the sliced function
\[v\mapsto 1_B(\cdot,v) \in \F(\bbR^d,\bbR)\]
is uniformly continuous on the whole of $\bbR^r$.

Now, $Q$ is defined by a finite intersection of half-spaces in $\bbR^{d+r}$, say $Q = H_1\cap \cdots \cap H_m$, where each $H_i$ may be open or closed.  This implies that $1_Q = (1_{B\cap H_1})\cdots (1_{B\cap H_m})$, and each function here has bounded support and is uniformly bounded by $1$, so the result will follow if we show that each of the sliced functions
\[v\mapsto 1_{B\cap H_i}(\cdot,v), \quad i=1,2,\ldots,m,\]
is uniformly continuous on $\rm{int}\,\Pi(Q)$.

Now, on the one hand, if the bounding hyperplane $\partial H_i$ is not of the form $\bbR^d\times V$ for any hyperplane $V \leq \bbR^r$, then this sliced function is actually continuous on the whole of $\bbR^r$.

On the other, if $\partial H_i$ equals $\bbR^d\times V$ for such a hyperplane $V \leq \bbR^r$, then $H_i$ itself equals $\bbR^d\times K_i$ for some half-space $K_i \subseteq \bbR^r$, and in this case the sliced function of $1_{B\cap H_i}$is uniformly continuous on $\rm{int}\,K_i$, since it agrees with the sliced function of $1_B$ on $\rm{int}\,K_i$.

Let $I \subseteq \{1,2,\ldots,m\}$ be the set of indices $i$ for which $H_i$ is as in the second of the possibilities above.  Then we have shown that the function in question is uniformly continuous on
\[\Pi(B)\cap \bigcap_{i \in I}\rm{int}\,K_i.\]
This clearly contains (and is often equal to) $\rm{int}\,\Pi(Q)$.
\end{proof}

\begin{cor}
If $A \subseteq \bbR^d$ and $B \subseteq \bbR^r$ are bounded convex polytopes of positive measure (equivalently, nonempty interior), and $f:A\times B\to\bbR^D$ is a step polynomial, then there is a partition $\frP$ of $B$ into convex sub-polytopes such that for each $C \in \frP$, the sliced function
\[C\to \F(A,\bbR^D):v\mapsto f(\cdot,v)\]
extends to a uniformly continuous function on $\ol{C}$.
\end{cor}

\begin{proof}
Arguing coordinate-wise, we may assume $D=1$.

If $f$ is a basic step polynomial supported on a convex polytope $Q \subseteq A\times B$, then the previous lemma gives that the sliced function is uniformly continuous on the relative interiors in $B$ of both $\Pi(Q)$ and $B \setminus \Pi(Q)$.

A general $f$ may be written as a finite sum of basic step polynomials, say $f = f_1 + \cdots + f_m$.  Let $Q_1$, \ldots, $Q_m$ be convex polytopal supports for $f_1$, \ldots, $f_m$. Let $\frP_0$ be the partition generated by the polytopes $\Pi(Q_1)$, \ldots, $\Pi(Q_m)$ \emph{together with all their facets}, regarded as separate polytopes also, and let $\frP_1$ be a refinement of $\frP_0$ into convex polytopes.

It follows that the sliced function of each $f_i$, and hence also of $f$, is uniformly continuous on the relative interior in $B$ of every cell of $\frP_1$. This implies the desired conclusion for the cells of $\frP_1$ that have nonempty relative interiors, but does not handle the cells that lie in $\partial \Pi(Q_1) \cup\cdots \cup \partial \Pi(Q_m)$.  However, if a cell of $\frP_1$ has no relative interior, then it lies in some co-dimension-$1$ affine subspace of $\bbR^r$.  For each of these lower-dimensional subspaces $V$, we may now simply repeat the previous construction for the restriction $f|(A\times (B\cap V))$. An induction on $r$ completes the proof.
\end{proof}

\begin{cor}\label{cor:connected-images}
Suppose that $A$ is an Abelian Lie group, $W$ and $Z$ are compact Abelian groups, and $f:W\times Z \to A$ is a step polynomial.  Then there is a QP partition $\frQ$ of $Z$ such that the map $z\mapsto f(\cdot,z)$ is uniformly continuous from $C$ to $\F(W,A)$ for every $C \in \frQ$.
\end{cor}

\begin{proof}
By Definition~\ref{dfn:step-poly}, there are affine maps $\a:W\to \bbT^d$ and $\chi:Z\to \bbT^r$, a step polynomial $f_0:\bbT^{d+r} \to \bbR^D$ with image contained in some closed subgroup $\t{A} \leq \bbR^D$, and a homomorphism $\psi:\t{A}\to A$, such that $f = \psi\circ f_0\circ (\a\times \chi)$.  It therefore suffices to find a suitable QP partition of $\bbT^r$ for $f_0$ and pull it back through $\chi$.

However, upon factorizing through the fundamental-domain map $\{\cdot\}:\bbT^{d+r} \to [0,1)^{d+r}$, this is precisely the output of the preceding corollary.
\end{proof}

\subsection{Integrating slices of step polynomials}

The following result will be of great importance in the next section, where we develop some group cohomology using step-polynomial cocycles.  It strikes me as something that is probably known, but I have not been able to find a suitable reference.

\begin{prop}
Suppose that $C \subseteq \bbR^n\times \bbR^m$ is a bounded convex polytope and that $p:C\to \bbR$ is a step polynomial, which we extend by $0$ outside $C$.  Then the integrated function
\[q(u) := \int_{\bbR^m}p(u,v)\,\d v\]
is a step polynomial of bounded support.
\end{prop}

\begin{proof}
It suffices to prove this result for $m=1$, since the general case may then be recovered by integrating out the last $m$ coordinates one-by-one.

Let $\Pi:\bbR^n\times \bbR^m \to \bbR^n$ be the coordinate projection. By decomposing $p$ into basic step polynomials, it suffices to assume that $p$ is the restriction to $C$ of a genuine polynomial.  Because $C$ is bounded, so is $\Pi(C)$, and clearly $q$ vanishes outside $\Pi(C)$.  It remains to prove that $q$ is a step polynomial on $\Pi(C)$.

We will complete this proof on $\rm{int}\,\Pi(C)$.  Having done so, the same proof may be repeated above the $\Pi$-pre-image of each lower-dimensional facet of $\Pi(C)$, so this implies the full result.

The behaviour of $q$ on $\Pi(C)$ is unchanged if we replace $C$ by its closure, so we may now assume that $C$ is closed. In this case it is defined by some intersection of closed linear inequalities, so we may write
\[C = \bigcap_{i=1}^rH_i \quad \hbox{with} \quad H_i = \{(u,v)\,|\ a_i\bullet u + b_iv_i \leq \a_i\}\]
for some $(a_i,b_i) \in (\bbR^n\times \bbR) \setminus \{(0,0)\}$ and $\a_i \in \bbR$, $i=1,2,\ldots,r$.  For each $j\leq r$, let also
\[C_j := \bigcap_{i \in [r]\setminus \{j\}}H_i.\]

We may assume that the above representation of $C$ is irredundant, meaning that for each $i$, the simpler intersection $C_i$ is strictly larger than $C$.  Since $C$ and $C_i$ are both closed, this is equivalent to the assertion that each of the hyperplanes $\partial H_i$ intersects the interior of the corresponding $C_i$.  Knowing this, it follows that each of these intersections $\partial H_i\cap C_i$ is a bounded convex polytope with non-empty interior relative to $\partial H_i$.

By convexity and boundedness, the slice
\[I_u := \{v \in \bbR\,|\ (u,v) \in C\}\]
is a closed bounded interval for every $u$, and the set
\[D := \{u \in \bbR^n\,|\ \rm{int}\,I_u \neq \emptyset\}\]
is also a bounded convex polytope.  It has the property that $q|(\bbR^n\setminus D) = 0$. Each $u \in D$ may be labeled by a pair $(i_1(u),i_2(u))$ of distinct elements of $\{1,2,\ldots,r\}$ with the property that $\partial H_{i_1(u)}$ intersects $\{u\}\times \bbR$ precisely in the lower end-point of $I_u$, and $\partial H_{i_2(u)}$ intersects $\{u\}\times \bbR$ precisely in the upper end-point.  It is easy to see that this choice of $(i_1(u),i_2(u))$ is unique for a.e. $u \in D$, but for those $u$ with more than one possibility, let us order the set of distinct pairs in $\{1,2,\ldots,r\}$, and always choose $(i_1(u),i_2(u))$ to be minimal among the possibilities.

With this convention, the choice of $(i_1(u),i_2(u))$ defines a partition of $D$, say $\frP$, into sets defined by intersections of linear inequalities: that is, into convex sub-polytopes.  The proof will be finished by showing that $q$ agrees with a polynomial on each of these sub-polytopes.

Indeed, suppose that $D' = \{u\,|\ (i_1(u),i_2(u)) = (i_1,i_2)\} \in \frP$.  Then we may write
\[I_u = [\psi_1(u),\psi_2(u)]\]
for some affine functions $\psi_1,\psi_2:D'\to \bbR$, where $\psi_s$ is the function whose graph is $\partial H_{i_s} \cap (D'\times \bbR)$ for $s = 1,2$.  For our integral, this now gives
\[q(u) = \int_{\psi_1(u)}^{\psi_2(u)}p(u,v)\,\d v \quad \forall u \in D'.\]

We will complete the proof by induction on $\deg p$.  If $p$ is a constant, then we simply obtain
\[q(u) = \psi_2(u) - \psi_1(u) \quad \hbox{on}\ D',\]
which is affine.  So now suppose that $\deg p \geq 1$.  We will use the inductive hypothesis to prove that the gradient $\nabla q$ is a polynomial function on $D'$.  Indeed, the rules for differentiation under an integral give
\[\nabla q(u) = \int_{\psi_1(u)}^{\psi_2(u)}\nabla_up(u,v)\,\d v + p(u,\psi_2(u))\cdot \nabla\psi_2 - p(u,\psi_1(u))\cdot\nabla\psi_1,\]
so we can apply the inductive hypothesis to the first term here, and observe directly that the second and third are polynomials on $D'$.
\end{proof}

\begin{cor}\label{cor:int-step-poly}
If $Y$ and $Z$ are compact Abelian groups and $f:Y\times Z \to \bbR$ is a step polynomial, then so is the integrated function
\[g(y) := \int_Zf(y,z)\,\d z.\]
\end{cor}

\begin{proof}
Since $f$ factorizes through a Lie-group quotient, and then $g$ does the same, we may assume that $Y$ and $Z$ are both Lie groups.  Since a function on an Abelian Lie group is a step polynomial if this is true of its restriction to every identity-component coset, we may assume further that $Y$ is a torus.  On the other hand, if $Z_0$ is the identity-component of $Z$, then $g$ is the sum of $[Z:Z_0]$-many integrals over cosets of $Z_0$, so it suffices to prove the result for each of these, and hence assume that $Z$ is also a torus.

Finally, if $Z = \bbT^m$ and $Y = \bbT^n$, then applying the map $\{\cdot\}:Y\times Z\to [0,1)^{m+n}$ turns the definition of $g$ into
\[g(y) = \int_{\bbR^m}p(\{y\},v)\,\d v\]
for some step polynomial $p:\bbR^{m+n}\to \bbR$ supported on $[0,1)^{m+n}$.  This integral is treated by the preceding proposition.
\end{proof}

\subsection{A special class of Lie modules}

In the sequel we will often consider step polynomials $Z\to A$ for an Abelian Lie group $A$ that is already equipped with an action of $Z$.  In this setting the following special class of Lie modules will become important.

\begin{dfn}\label{dfn:Lie-sp-mod}
Let $A$ be a Polish Abelian group, $Z$ a compact Abelian group and $T:Z\to \rm{Aut}\,A$ an action.  Then this action is \textbf{step-polynomial}, or `\textbf{SP}', if $A$ is a Lie group and,
for every compact Abelian group $Y$ and step polynomial $f:Y\to A$, the function
\[F:Z\times Y\to A:(z,y)\mapsto T^zf(y)\]
is also a step polynomial.  In this case, $A$ is a \textbf{SP $Z$-module}.
\end{dfn}

This includes the requirement that the orbit maps $z\mapsto T^za$ be SP for all $a \in A$.  This will turn out to be equivalent, but we will use the definition in the stronger form above.  Later we will introduce a larger class of `SP modules', of which Lie SP modules will be a special case.

\begin{ex}
The trivial action of $Z$ on $A$ is always SP  More generally, an action that factorizes through a quotient $q:Z \to Z_1$ to a finite group is always SP, because the partition of $Z$ into the cosets of $\ker q$ is QP. \fin
\end{ex}

\begin{ex}
The rotation action $\rm{rot}:\bbT \to \rm{O}(2) \leq \rm{Aut}\,\bbR^2$ given by
\[\rm{rot}(t) = \left(\begin{array}{cc}\cos 2\pi t & \sin 2\pi t\\ -\sin 2\pi t & \cos 2 \pi t\end{array}\right)\]
is not SP, because $\sin$ and $\cos$ are not SP functions $\bbT\to\bbR$.

More generally, a \textbf{full rotation action} of some $Z$ is an action isomorphic to $\rm{rot}\circ \chi:Z\actson \bbR^2$ for some $\chi \in \hat{Z}$ for which $\chi(Z) = \bbT$.  It is easy to see that any full rotation action is not SP. \fin
\end{ex}

In fact, it will turn out that full rotation actions are the unique obstructions to a Lie module being SP.

\begin{prop}\label{prop:full-rot-char}
A Lie $Z$-module $T:Z\actson A$ is SP unless it has a closed submodule which is a full rotation action.  In particular, any compact-by-discrete Lie $Z$-module is SP.
\end{prop}

The proof of this will require a few steps.

\begin{lem}\label{lem:spmod-from-finite-index}
If $T:Z \to \rm{Aut}\,A$ is the action of a Lie $Z$-module, and $W \leq Z$ is a finite-index subgroup such that $T|W$ is SP, then $T$ is SP.
\end{lem}

\begin{proof}
If $f:Y\to A$ is SP, then the function $F(z,y) := T^zf(y)$ is SP if this is so on every coset of $W\times Y$ separately, since the partition of $Z\times Y$ into these cosets is QP  However, if $(z_0 + W)\times Y$ is such a coset, then the map $y\mapsto T^{z_0}f(y)$ is a step polynomial on $Y$, and so
\[F(z_0 + \cdot,\cdot):W\times Y \to A:(w,y) \mapsto T^w(T^{z_0}f(y))\]
is a step polynomial because $T|W$ is SP.
\end{proof}

\begin{lem}\label{lem:cpt-or-disc-are-sp}
If $A$ is a discrete Abelian group or a compact Abelian Lie group, then any $Z$-action on $A$ is SP.
\end{lem}

\begin{proof}
First suppose that $A$ is discrete and $T:Z\actson A$ is a continuous action.  Then every $T$-orbit must be finite.  It follows that every finite-subset of $A$ is contained in a finitely-generated, $Z$-invariant submodule of $A$.  Since a step polynomial from a compact group always has pre-compact image, this means we may consider only finitely-generated discrete modules. In this case, $T$ factorizes through a finite quotient of $Z$, so it is certainly SP.

On the other hand, if $A$ is a compact Lie $Z$-module, then $A = \hat{\hat{A}}$ is the Pontryagin dual of a finitely-generated discrete $Z$-module, so again the action factorizes through a finite quotient of $Z$.
\end{proof}

\begin{proof}[Proof of Proposition~\ref{prop:full-rot-char}]
\emph{Step 1: Euclidean modules.}\quad Let $E$ be a Euclidean space.  All continuous group automorphisms of $E$ are linear transformations, so the action of $Z$ is actually a linear representation.  Since $Z$ is compact, this representation admits an invariant inner product, by the standard averaging trick.  It therefore decomposes as a finite direct sum $\rho_0\oplus \bigoplus_i (\rm{rot}\circ \chi_i)$, where $\rho_0$ factorizes through a finite quotient of $Z$ and each $\rm{rot}\circ \chi_i$ is a full rotation action.  By Lemma~\ref{lem:spmod-from-finite-index}, this is SP if and only if there are no summands of the second kind.

\vspace{7pt}

\emph{Step 2: connected modules.}\quad Now suppose that $A$ is a connected Lie $Z$-module.  Then it may be presented as
\[D \into \t{A} \stackrel{q}{\onto} A,\]
where $\t{A}$ is the universal cover of $A$, a Euclidean space, and $D$ is a discrete subgroup of $\t{A}$.  This presentation is canonical, so the action of $Z$ on $A$ lifts to a linear representation of $Z$ on $\t{A}$.  By Step 1, either that representation on $\t{A}$ factorizes through a finite quotient of $Z$, in which case so does the action on $A$, or it contains a full rotation action as a direct summand.  In the second case, letting $E_1$ be such a full rotation subrepresentation, it must hold that $E_1\cap D = 0$, since any non-zero element of $E_1$ has non-discrete $Z$-orbit.  Therefore $q|E_1$ is injective, and so its image $q(E_1)$ is a full rotation submodule of $A$.  Thus, for connected modules, any $Z$-action either factorizes through a finite quotient of $Z$ (so is SP) or contains a full rotation subaction (so is not SP).

\vspace{7pt}

\emph{Step 3: general case.}\quad Finally, consider a general Lie $Z$-module $A$, and let $A_0$ be its identity component.  We will show that if $A_0$ does not contain a full rotation subaction, then the whole module $A$ is SP.

If the action of $A$ on $A_0$ contains no full rotation subaction, then step 2 has shown that it must be trivial for some finite-index subgroup $W \leq Z$.  By Lemma~\ref{lem:spmod-from-finite-index}, it suffices to show that the $W$-action on $A$ is SP.

Let $B:= A/A_0$.  Standard structure theory for Abelian Lie groups (see, for instance,~\cite[Section II.9]{HewRos79}) shows that we have an isomorphism $A \cong A_0\times B$ of topological groups, though this may not respect the $Z$-actions.

Now suppose that $f:Y\to A_0\times B$ is a step polynomial, and write it in components as $(f_1,f_2)$.  Then $f_2:Y\to B$ is a step function, hence takes only finitely many values.  Therefore, as in the proof of Lemma~\ref{lem:cpt-or-disc-are-sp}, we may assume that $B$ is finitely generated.  However, having done so, the $W$-action on $B$ must trivialize on a further finite-index subgroup, so by shrinking $W$ again we may assume that the $W$-action on $B$ is trivial.

Finally, in this case a routine exercise shows that the action $T$ on the whole of $A_0 \times B$ must take the form
\[T^z(a,b) = (a + \s(b)(z),b),\]
where $\s\in \rm{Hom}(b,\rm{Hom}(Z,A_0))$. This now gives
\[T^zf(y) = (f_1(y) + \s(f_2(y))(z),f_2(y)).\]
Since $f_2$ takes only finitely many values, controlled by some QP partition of $Z$, this clearly defines a step polynomial on $Z\times Y$.
\end{proof}

\begin{cor}\label{cor:s-p-class-closed}
If $A \leq B \stackrel{q}{\onto} C$ is a short exact sequence in $\PMod(Z)$, then $B$ is Lie and SP if and only if both $A$ and $C$ are Lie and SP.
\end{cor}

\begin{proof}
It is classical that the property of being Lie is closed  in $\PMod(Z)$ for extensions, quotients and closed subgroups, so we may assume all of $A$, $B$ and $C$ are Lie modules.  Let $T$ denote any of these three actions.

First suppose that $B$ is SP.  Any step polynomial $Y\to A$ is also a step polynomial $Y\to B$, so $A$ is also SP.  On the other hand, if $f:Y\to C$ is a step polynomial, then Lemma~\ref{lem:s-p-lift} gives a SP lift of it $F:Y\to B$.  It follows that
\[T^zf(y) = q(T^zF(y))\]
is a homomorphic image of a step polynomial, hence is a step polynomial.

On the other hand, suppose that $B$ is not SP.  Then Proposition~\ref{prop:full-rot-char} gives a full rotation $Z$-submodule $B_1 \leq B$.  Since full rotation modules are irreducible (indeed, the orbit of any non-zero element of $B_1$ has $\bbZ$-span which is dense in $B_1$), it follows that either $B_1 \leq A$, in which case $A$ is not SP, or $q|B_1:B_1\to q(B_1) \leq C$ is an isomorphism, in which case $C$ is not SP.
\end{proof}

\section{Functional and semi-functional modules}\label{sec:fnl-and-semi-fnl}

This subsection will introduce a class of Polish $Z$-modules consisting of functions between locally compact groups, and a larger class of modules obtained as quotients of such.  The latter will be assembled into a category.  The next section will develop some basic homological algebra in this category in order to introduce two different notions of groups cohomology that we will need for its objects.

Let $Z$ denote a compact metrizable Abelian group throughout this section.

\subsection{Functional and semi-functional modules}

\begin{dfn}[Functional modules]\label{dfn:fnl}
Suppose that $X$ is another compact Abelian group with a distinguished homomorphism $\a:Z\to X$, and that $A$ is a Lie $Z$-module, say with action $T_A$.  Then a \textbf{functional $Z$-module with fibre $A$ and base $\a$} is a closed $Z$-submodule of $\F(X,A)$, equipped with the diagonal $Z$-action:
\[(z \cdot f)(x) := T_A^z(f(x - \a(z))).\]
Sometimes we will refer instead to $X$ as the `base', when the choice of $\a$ is clear.

Sometimes it will be important that $X = X_0\times Z$ for some $X_0$ with the obvious homomorphism $Z \into X$.  In this case we shall refer instead to a functional $Z$-module with fibre $A$ and \textbf{dummy} $X_0$.

It could happen that a $Z$-module $P$ is a submodule of both $\F(X,A)$ and $\F(X',A')$ for two different bases $X$ and $X'$ and two different fibres $A$ and $A'$.  Thus, the inclusion $P\subseteq \F(X,A)$ for a specific choice of base and fibre is part of the structure that defines a functional $Z$-module, even though it will usually be left implicit.
\end{dfn}

\begin{ex}
Whenever $q:Z\onto Y$ is a surjective homomorphism and $A$ is a Lie $Y$-module, the $Y$-module $\F(Y,A)$ pulls back to a functional $Z$-module
\[\F(Y,A) \circ q := \{f\circ q\,|\ f \in \F(Y,A)\} \leq \F(Z,A^q),\]
where $A^q$ denotes $A$ endowed with the action of $Z$ obtained through $q$.  The same construction may be applied to any other functional $Y$-module. \fin
\end{ex}

\begin{ex}\label{ex:main-fnl-egs}
Most other examples will arise by repeatedly forming images or kernels of suitable homomorphisms.

For instance, given $A$, $Z$ and a subgroup-tuple $\bfU = (U_1,\ldots,U_k)$ in $Z$, the module of associated {\PDE}-solutions,
\[M := \{f \in \F(Z,A)\,|\ d^{U_1}\cdots d^{U_k}f = 0\},\]
is functional, with fibre $A$ and base $\rm{id}_Z$.  It is the kernel of the $Z$-module homomorphism
\[\F(Z,A)\to \F(U_1\times \cdots \times U_k\times Z,A):f\mapsto d^{U_1}\cdots d^{U_k}f,\]
where the target module has dummy $U_1\times \cdots \times U_k$.

Within $M$, one also finds the submodules $M_i$, $i=1,2,\ldots,k$, of functions that satisfy one of the obvious simplified equations:
\[M_i := \{f \in \F(Z,A)\,|\ d^{U_1}\cdots d^{U_{i-1}}d^{U_{i+1}}\cdots d^{U_k}f = 0\}.\]
From these one forms the submodule $M_0 := M_1 + \cdots + M_k$ of `degenerate' solutions to the original {\PDE}.  Each $M_i$ is also obtained as the kernel of a homomorphism defined by repeated differencing, and then $M_0$ is the image of $M_1\oplus \cdots \oplus M_k$ under the sum-homomorphism. The results of Part I include that $M_0$ is closed, so it is another example of a functional $Z$-module. \fin
\end{ex}

Much of our later work will concern pairs of functional modules, and their quotients.

\begin{dfn}[Semi-functional modules]\label{dfn:semi-fnl}
A \textbf{semi-functional} $Z$-module with fibre $A$ and base $\a:Z\to X$ is an inclusion $P\leq Q$ consisting of a functional $Z$-module $Q$ with fibre $A$ and base $\a$ and a closed submodule $P$ of $Q$.  Such a semi-functional module will sometimes be written $P\stackrel{\incln}{\to} Q$. The \textbf{quotient} of this semi-functional $Z$-module is $Q/P$.

As for functional $Z$-modules, the inclusion of $Q$ in $\F(X,A)$ for a particular base and fibre is part of the structure of a semi-functional $Z$-module.
\end{dfn}

\begin{ex}
In the notation of Example~\ref{ex:main-fnl-egs}, since $M_0$ is closed, the inclusion $(M_0 \leq M)$ is a semi-functional $Z$-module.  The quotient $M/M_0$ was studied in several of the specific worked examples in Part I, where it could be computed explicitly in terms of some cohomology groups. \fin
\end{ex}

\begin{dfn}[Semi-functional morphism]\label{dfn:semi-fnl-mor}
Given semi-functional $Z$-modules $(P\leq Q)$ and $(P' \leq Q')$, possibly with different fibres and bases, a \textbf{semi-functional morphism} between them is a commutative diagram
\begin{center}
$\phantom{i}$\xymatrix{
P\ar_-\incln[d]\ar^{\psi_1}[r] & P'\ar^-\incln[d]\\
Q\ar^{\psi_2}[r] & Q'.
}
\end{center}
\end{dfn}

Note that in the above diagram one must have $\psi_1 = \psi_2|P$, so properties of $\psi_1$ often follow from those of $\psi_2$.  This diagram will often be abbreviated to
\[(P\leq Q) \stackrel{\psi_2}{\to} (P'\leq Q').\]

\subsection{Step-polynomial submodules}\label{subs:semi-sp-mods}

It is clear that semi-functional $Z$-modules and morphisms form a category.  This point of view will be important shortly, but we will need to restrict our attention to a special subclass of morphisms, defined in terms of their behaviour on step-polynomial elements.

The following non-standard terminology will be valuable in the sequel.

\begin{dfn}
If $Z$ is a compact metrizable Abelian group, then an \textbf{enlargement} of it is another compact metrizable Abelian group $Z'$ containing $Z$ as a subgroup.
\end{dfn}

\begin{dfn}
Let $A$ an Abelian Lie group.  If $P \leq \F(X,A)$ is any subgroup, then its \textbf{SP subgroup relative to $(X,A)$} is the subgroup $P\cap \F_{\rm{sp}}(X,A)$.  Explicitly, this is the subgroup of Haar-a.e. equivalence classes of functions in $P$ that contain a step-polynomial representative $X \to A$.

Usually, $X$ and $A$ will be left to the reader's understanding, and the SP subgroup will be denoted by $P_{\rm{sp}}$.  To be precise, however, this subgroup may depend on a particular choice of $X$ and $A$.
\end{dfn}

\begin{dfn}[Complexity-bounded homomorphisms]\label{dfn:fnl-mor}
Suppose that $A$ and $A'$ are Abelian Lie groups, that $X$ and $X'$ are compact Abelian groups, and that $P\leq \F(X,A)$ and $Q\leq \F(X',A')$ are closed subgroups.  Then a continuous group homomorphism $\phi:P\to Q$ is \textbf{complexity-bounded relative to $(X,A)$ and $(X',A')$} if
\[\phi(P_{\rm{sp}})\subseteq Q_{\rm{sp}}.\]
Once again, the relevant choice of $(X,A)$ and $(X',A')$ will usually be obvious, and will be suppressed from the notation.

Similarly, a semi-functional morphism as in Definition~\ref{dfn:semi-fnl-mor} is \textbf{complexity-bounded} if the horizontal arrows of its commutative diagram are complexity-bounded.
\end{dfn}

Thus, complexity-bounded homomorphisms are those that respect the property of being step-polynomial among the elements of the domain module $P$.  This requirement can be weak or strong, depending on whether $P$ contains many (equivalence classes represented by) step polynomials.  In the examples that will concern us, the modules will be replete with step polynomials (for instance, the step polynomials will be dense in probability), and so complexity-bounded homomorphisms will be rather special.  However, we should note in passing that I do now know the answer to the following basic question.

\begin{ques}
Are there an Abelian Lie group $A$ with trivial $\bbT$-action and a functional $\bbT$-submodule $Q \leq \C(\bbT,A)$ such that $Q_{\rm{sp}}$ is not dense in $Q$ for convergence in probability?
\end{ques}

\begin{ex}
Suppose that $P$ and $Q$ are as above, that $\phi:P\to Q$ is a continuous homomorphism, and that there are
\begin{itemize}
 \item a continuous homomorphism $\k:A\to A'$, and
\item a family of continuous epimorphisms $\zeta_1,\ldots,\zeta_k:X' \onto X$
\end{itemize}
such that
\begin{eqnarray}\label{eq:cplx-bdd-homo}
\phi(f)(x') = \sum_{i=1}^k \k(f(\zeta_i(x'))) \quad \forall f\in P.
\end{eqnarray}
Homomorphisms of this form are always complexity-bounded. \fin
\end{ex}

In Example~\ref{ex:main-fnl-egs}, several functional modules arising in the study of {\PDE}s were obtained using images and kernels of homomorphisms, and all of those homomorphisms took the form~(\ref{eq:cplx-bdd-homo}).  However, it would be interesting to know whether kernels of complexity-bounded homomorphisms suffice, without allowing images.

\begin{ques}
In the notation of Example~\ref{ex:main-fnl-egs}, is there a complexity-bounded $Z$-module homomorphism $\phi:\F(Z,A)\to \F(X,A')$ for some base $X$ and fibre $A'$ such that $M_0 = \ker \phi$? Can $\phi$ be of the form~(\ref{eq:cplx-bdd-homo})? \fin
\end{ques}

Clearly a composition of complexity-bounded homomorphisms is complexity-bounded, so we may make the following.

\begin{dfn}[Basic semi-functional category]
The \textbf{basic semi-functional category} of $Z$-modules is the category $\SFMod_0(Z)$ whose objects are semi-functional $Z$-modules and whose morphisms are complexity-bounded semi-functional $Z$-morphisms.
\end{dfn}

If $P \leq Q$ is a semi-functional $Z$-module, then clearly $P_\sp \leq Q_\sp$.  From the definition of complexity-boundedness, it follows that the map
\[\big(P \leq Q\big) \mapsto (P_\sp \leq Q_\sp)\]
defines a functor $\rm{StepPoly}$ from $\SFMod_0(Z)$ to the category $\Mod^\leq(Z)$ of inclusions of abstract $Z$-modules (this functor forgets the topologies, base and fibre of $P_\sp$ and $Q_\sp$).

In the sequel, we will sometimes wish to study step polynomials across a whole commutative diagram of semi-functional modules.

\begin{dfn}[Step-polynomial subdiagrams]\label{dfn:subdiag}
Given any commutative diagram in $\SFMod_0(Z)$ (formally, any covariant functor from a small category $\sfC$ to $\SFMod_0(Z)$), its \textbf{step-polynomial subdiagram} is the diagram of the same shape, where each semi-functional module that appears in the original diagram has been replaced by its step-polynomial submodule, and all morphisms are accordingly restricted to these SP subgroups (formally, this is the composition of the original functor $\sfC \to \SFMod_0(Z)$ with $\rm{StepPoly}:\SFMod_0(Z) \to \Mod^\leq(Z)$).
\end{dfn}

We will also need the following approximate reversal of Definition~\ref{dfn:fnl-mor}.

\begin{dfn}[Step-polynomial pre-images]\label{dfn:s-p-pre-ims}
Let $P\leq \F(X,A)$ and $Q \leq \F(X',A')$ be as in Definition~\ref{dfn:fnl-mor}, and let $\phi:P\to Q$ be a continuous homomorphism.  Then $\phi$ admits \textbf{SP pre-images relative to $(X,A)$ and $(X',A')$} if
\[(\phi(P))_{\rm{sp}} \subseteq \phi(P_{\rm{sp}}):\]
that is, whenever $f \in \phi(P)$ is a Haar-a.e. equivalence class of a step polynomial, there is some $g \in P_{\rm{sp}}$ with $\phi(g) = f$.
\end{dfn}

Clearly, if $\phi$ is an invertible homomorphism of functional $Z$-modules, then $\phi$ has SP pre-images if and only if $\phi^{-1}$ is complexity-bounded. In general it may be false that a composition of homomorphisms with SP pre-images has SP pre-images, so one cannot define a further subcategory of $\SFMod_0(Z)$ by requiring SP pre-images of any of the homomorphisms involved.

The last task of this subsection will be to interpret a direct sum of two functional $Z$-modules as another functional $Z$-module.

\begin{dfn}
If $P_i \leq \F(X_i,A_i)$ are functional $Z$-modules with bases $\a_i:Z\to X_i$ for $i=1,2$, then their \textbf{direct sum} in the category $\SFMod_0(Z)$ is the functional $Z$-module with base $Z \to X_1\times X_2:z\mapsto (\a_1(z),\a_2(z))$ and fibre $A_1\oplus A_2$ which consists of all functions of the form
\[(x_1,x_2) \mapsto (f_1(x_1),f_2(x_2))\]
for some $f_1 \in P_1$ and $f_2 \in P_2$.

If $(P_i \leq Q_i)$, $i=1,2$, are semi-functional $Z$-modules, then their \textbf{direct sum} is $(P_1\oplus P_2 \leq Q_1\oplus Q_2)$.
\end{dfn}

It is an easy check that the above definition gives all the usual categorial properties of direct sums, such as associativity: there is a natural isomorphism $(P_1\oplus P_2)\oplus P_3 \cong P_1 \oplus (P_2 \oplus P_3)$.  For this reason we will generally write such multiple direct sums without brackets.

%It will sometimes be important to exploit some flexibility in the choice of base for a semi-functional modules.
%
%\begin{lem}
%Let $A$ be an SP Lie $Z$-module, and let
%\begin{center}
%$\phantom{i}$\xymatrix{
%Z \ar^{\a'}[r]\ar_\a[dr] & X' \ar@{->>}^q[d] \\
%& X
%}
%\end{center}
%be a commutative diagram of compact metrizable Abelian groups, where $q$ is a quotient homomorphism.  Then the composition operator $f\mapsto f\circ q$ defines a Polish $Z$-module isomorphism
%\[\Phi:\F(X,A) \to \{F \in \F(X',A)\,|\ F(z'+y) = F(z')\ \forall y \in \ker q\}\]
%such that both $\Phi$ and $\Phi^{-1}$ are complexity-bounded, where $Z$ acts on $\F(X,A)$ (resp. $\F(X',A)$) through the base $\a$ (resp. $\a'$).
%\end{lem}
%
%\begin{proof}
%This is all immediate, except for the fact that $\Phi^{-1}$ is complexity-bounded, which follows from Lemma~\ref{lem:strict-invar}.
%\end{proof}

\subsection{Behaviour under co-induction}\label{subs:behav-coind}

Unfortunately, $\SFMod_0(Z)$ is still not quite adequate for our later homological-algebraic purposes.  The problem is that (as far as I know) the property of complexity-boundedness for morphisms may not be respected by co-induction.

In order to resolve this, we must decide how to interpret `step polynomials' in a module co-induced from a functional module.  Let $P\leq \F(X,A)$ be a functional $Z$-module with base $\a:Z\to X$.

First, if $\beta:Z\to V$ is another base, then we can regard $\F(V,P)$ as a submodule of $\F(V\times X,A)$, so that it also becomes a functional $Z$-module with base $z\mapsto (\b(z),\a(z))$.  We will often use this construction with $\b = 0$.

Now suppose that $Z'\geq Z$ is an enlargement.  Then we have
\[\Cnd_Z^{Z'}P := \F(Z',P)^Z \leq \F(Z'\times X,A)^Z.\]
The $Z$-action whose fixed-point module is $\F(Z'\times X,A)^Z$ is the diagonal action on $X$ (through the homomorphism $\a:Z\to X$), on $Z'$ (by rotation) and on $A$.  Having taken these fixed points, the new $Z'$-action which defines $\Cnd_Z^{Z'}P$ as a $Z'$-module is by rotation of the $Z'$-coordinate alone. Thus we have canonically identified $\Cnd_Z^{Z'}P$ as a functional $Z'$-module in $\F(Z'\times X,A)$, with fibre $A$ and dummy $X$, and its step-polynomial elements will always be understood relative to $(Z'\times X,A)$.

With this interpretation, it is an easy check that given two enlargements $Z \leq Z' \leq Z''$ and a functional $Z$-module $P$, the usual isomorphism
\[\Cnd_Z^{Z''}P \cong \Cnd_{Z'}^{Z''}\Cnd_Z^{Z'}P\]
is complexity-bounded in both directions. We will therefore generally identify these two functional modules whenever convenient.

However, a slightly trickier issue is the following.  If $Z' = Z$, then one has a canonical isomorphism $\Cnd_Z^{Z'}M \cong M$ for any $M \in \PMod(Z)$.  It will sometimes be important that this `trivial' isomorphism respect the notion of step polynomials.  We will find that this is the case only if one restricts attention to fibres $A$ that are SP modules: this is why SP modules were introduced.  In that case, one has the following.

\begin{prop}\label{prop:s-p-coind-consistent}
Let $A$ be a Lie $Z$-module, thought of as a functional module with base $0$ (the trivial group) and fibre $A$.  Then $A$ is SP if the obvious isomorphism $\Cnd_Z^ZA \cong A$ is complexity-bounded.  In this case, for any base $\a:Z\to X$, any functional $Z$-module $P\leq \F(X,A)$, and any enlargement $Z' \geq Z$, there is a $Z$-equivariant isomorphism
\[\Phi:\F(Z'/Z,P) \stackrel{\cong}{\to} \Cnd_Z^{Z'}P\]
(though possibly not $Z'$-equivariant) such that both $\Phi$ and $\Phi^{-1}$ are complexity-bounded.  Moreover, this isomorphism may be chosen naturally in $P$: that is, if $P\to Q$ is a homomorphism of functional $Z$-modules, then one obtains a commutative diagram of $Z$-module homomorphisms
\begin{center}
$\phantom{i}$\xymatrix{
\F(Z'/Z,P) \ar[r] \ar_\cong[d] & \F(Z'/Z,Q) \ar^\cong[d]\\
\Cnd_Z^{Z'} P \ar[r] & \Cnd_Z^{Z'}Q.
}
\end{center}
\end{prop}

The proof of this will make use of the following simple lemma.

\begin{lem}\label{lem:strict-equivar}
Let $Y\leq Z$ be an inclusion of compact metrizable Abelian groups and let $A$ be an SP Lie $Y$-module.  If $F:Z\to A$ is a Borel function whose Haar a.e. equivalence class in $\F(Z,A)$ is $Y$-equivariant, then there is a \emph{strictly} $Y$-equivariant Borel function $f:Z \to A$ which agrees with $F$ a.e.  If $F$ is step-polynomial, then $f$ may be chosen step-polynomial.
\end{lem}

\begin{proof}
Let $\xi:Z\to Y$ be a step-affine (hence Borel) $Y$-equivariant map, as provided by Proposition~\ref{prop:fund-lift}, and define $\s:Z\to Z$ by $z = \s(z) + \xi(z)$.

Since the a.e. equivalence class of $F$ lies in $\F(Z,A)^Y$, we have that
\[F(z) = F((\s(z) + y) + (\xi(z)-y)) = T^{\xi(z) - y}F(\s(z) + y)\]
for a.e. $(z,y) \in Z\times Y$.  By Fubini's Theorem, we may therefore fix a choice of $y$ so that the above holds for a.e. $z$.  Now the right-hand side defines a Borel function $f(z)$ which is strictly $Y$-equivariant.  If $F$ is a step polynomial, then so is $f$, since $\s$ and $\xi$ are step-affine and $A$ is an SP module.
\end{proof}

\begin{proof}[Proof of Proposition~\ref{prop:s-p-coind-consistent}]
Let $T:Z\actson A$ be the action. Then $\Cnd_Z^Z A = \F(Z,A)^Z$ consists of the Haar-a.e. equivalence classes of the orbit functions $z\mapsto T^za$ for $a \in A$.  If $A \stackrel{\cong}{\to} \Cnd_Z^ZA$ is complexity-bounded, then for every $a\in A_{\rm{sp}} = A$ the function $z\mapsto T^za$ must agree with a step polynomial outside a negligible set.  This precludes a full-rotation submodule, so $A$ is SP, by Proposition~\ref{prop:full-rot-char}.

Now consider $\a:Z\to X$, $Z'$ and $P \leq \F(X,A)$ as in the statement of the proposition.  We construct the desired isomorphism $\Phi$ in two steps.

First, let $W := Z'/Z$, and let
\[\psi:W \times Z \to Z'\]
be a $Z$-equivariant bijection as given by Corollary~\ref{cor:fund-split}.  Then composition of functions with $\psi^{-1}\times \rm{id}_X$ gives a $Z$-equivariant isomorphism
\[\F(W\times Z,P)^Z \to \F(Z',P)^Z\]
which is complexity-bounded in both directions.  Since $\psi$ is chosen independently of $P$, this is obviously natural in $P$.

Secondly, we will find a natural $Z$-equivariant isomorphism
\[\Psi:\F(W,P)\to \F(W\times Z,P)^Z\]
which is complexity-bounded in both directions.  This is given as follows.  Let $R_P$ denote the $Z$-action on $P$.  Then for $f \in \F(W,P)$, let $\Psi(f) \in \F(W\times Z,P)^Z$ be the function
\[\Psi(f)(w,z) := R_P^z(f(w)).\]
If we regard $f$ as a function $W\times X \to A$, then $\Psi(f)$ is the function $W\times Z\times X\to A$ given by
\[\Psi(f)(w,z,x) = T^z\big(f(w,x-\a(z))\big).\]
It is routine to verify that this $\Psi$ is an injective and $Z$-equivariant homomorphism of Polish Abelian groups (recall that the relevant $Z$-action on $\F(W\times Z,P)^Z$ is now rotation of the $Z$-coordinate alone).  It is complexity-bounded because $A$ is assumed to be an SP module. Surjectivity and SP pre-images for $\Psi$ follow using Lemma~\ref{lem:strict-equivar}, because if $F:W\times Z\times X\to A$ is \emph{strictly} $Z$-equivariant then it equals $\Psi(f)$ with $f(w,x):= F(w,0,x)$.

Finally, to see that $\Psi$ is natural in $P$, suppose that $\phi:P\to Q$ is a homomorphism of functional $Z$-modules.  Let $\Psi_P$ and $\Psi_Q$ denote the above isomorphisms as constructed for $P$ and $Q$.  Then for any $f\in \F(W,P)$ one has
\[\Psi_Q(\phi\circ f)(w,z) = R_Q^z(\phi(f(w))) = \phi(R_P^z(f(w))) = \phi \big(\Psi_P(f)(w,z)\big).\]
This verifies the commutativity of the relevant diagram.
\end{proof}

\begin{cor}\label{cor:Lie-homos-s-p-reps}
If $\phi:A\to B$ is a closed homomorphism of Lie SP $Z$-modules, then it has stable SP pre-images.
\end{cor}

\begin{proof}
Clearly we may assume $B = \phi(A)$.  Letting $Z'\geq Z$ be an arbitrary enlargement, this is the assertion that any $f \in \F_\sp(Z',B)^Z$ is the image of some element of $\F_\sp(Z',A)^Z$.  By the previous lemma, this follows from the existence of a step-polynomial lift $Z'/Z \to A$ for any step polynomial $Z'/Z \to B$, and this is given by Lemma~\ref{lem:s-p-lift}.
\end{proof}

We are now ready for the following definition.

\begin{dfn}[Stable complexity-boundedness and SP pre-images]
Suppose that $P$ and $P'$ are functional modules with respective bases $X$ and $X'$ and respective fibres $A$ and $A'$, both of them SP  Then a homomorphism $P \to P'$ is \textbf{stably complexity-bounded} (resp. has \textbf{stable SP pre-images}) if for every enlargement $Z' \geq Z$, the co-induced morphism
\[\Cnd_Z^{Z'}P \to \Cnd_Z^{Z'}P'\]
is complexity-bounded (resp. has SP pre-images) relative to $(Z'\times X,A)$ and $(Z'\times X',A')$.

For complexity-boundedness, one also makes the analogous definition for semi-functional morphisms.
\end{dfn}

\begin{dfn}[Semi-functional category]
The \textbf{semi-functional category} of $Z$-modules is the category $\SFMod(Z)$ whose objects are semi-functional $Z$-modules and whose morphisms are stably complexity-bounded semi-functional morphisms.
\end{dfn}

Thus, $\SFMod(Z)$ is a subcategory of $\SFMod_0(Z)$ having the same objects.  Definition~\ref{dfn:subdiag} applies just as well to diagrams in $\SFMod(Z)$.  Semi-functional categories have the built-in property that for any enlargement $Z' \geq Z$, $\Cnd_Z^{Z'}(-)$ defines a functor $\SFMod(Z)\to \SFMod(Z')$; this will be crucial later.

Given a functional $Z$-module $Q$, it can be identified with the semi-functional module $0 \leq Q$.  With this convention, functional  $Z$-module define a full subcategory $\mathsf{FMod}(Z)$ of $\SFMod(Z)$.

\subsection{Semi-functional exact sequences}\label{subs:stable-ses}

Most na\"ively, one could define a `short exact sequence' in $\SFMod(Z)$ simply to be a short sequence of stably complexity-bounded semi-functional morphisms whose forgetful-functor image in $\PMod(Z)$ is short exact.  However, we will need to restrict attention to exact sequences whose step-polyomial subdiagrams are also exact, so we need a more refined definition.  Once again, a further complication arises because the desired behaviour of step-polynomial subdiagrams can be disrupted by co-induction.

Let
\[(P\leq Q) \to (P' \leq Q') \to (P'' \leq Q'')\]
be a pair of morphisms in $\SFMod(Z)$.  The next definition is in the spirit of classical `relative' homological algebra: see, for instance,~\cite[Section VI.2]{Bro82}.

\begin{dfn}[Exact sequences]
The above sequence is \textbf{exact} in $\SFMod(Z)$ if both $Q/P \to Q'/P' \to Q''/P''$ is exact in $\PMod(Z)$, and $Q_\sp/P_\sp \to Q_\sp'/P_\sp' \to Q_\sp''/P_\sp''$ is exact in $\Mod(Z)$.
\end{dfn}

Let $\phi$ be the homomorphism $Q\to Q'$ above.  Then, assuming the exactness of $Q/P \to Q'/P' \to Q''/P''$, the second condition above is equivalent to asserting that the homomorphism
\begin{eqnarray}\label{eq:sum-homo}
Q\oplus P' \to Q':(q,p')\mapsto \phi(q) + p'
\end{eqnarray}
has SP pre-images.

\begin{dfn}[Stable exact sequences]\label{dfn:stable}
A sequence as above is \textbf{stably exact} if the resulting co-induced diagram
\begin{center}
$\phantom{i}$\xymatrix{ \Cnd_Z^{Z'}P \ar[r]\ar[d] & \Cnd_Z^{Z'}P' \ar[r]\ar[d] & \Cnd_Z^{Z'}P''\ar[d]\\
\Cnd_Z^{Z'}Q \ar[r] & \Cnd_Z^{Z'}Q' \ar[r] & \Cnd_Z^{Z'}Q''
}
\end{center}
is still exact in $\SFMod(Z)$ for every enlargement $Z' \geq Z$.

A \textbf{stable short exact sequence} in $\SFMod(Z)$ is a sequence $0 \to (P \leq Q)\to (P' \leq Q') \to (P'' \leq Q'') \to 0$ that is stably exact in all places.
\end{dfn}

This time, assuming again the exactness of $Q/P \to Q'/P' \to Q''/P''$, this definition is equivalent to asserting that the homomorphism in~(\ref{eq:sum-homo}) has stable SP pre-images.  Knowing this for $Z' = Z$ does not seem to imply it for other enlargements, so this definition isolates a distinguished subclass of short exact sequences in $\SPMod(Z)$.

\begin{ex}\label{ex:concat}
Many stable short exact sequences arise as follows.  Let a \textbf{concatenation} of semi-functional $Z$-modules be a triple $P \leq R \leq Q$ of functional $Z$-modules, all contained in some $\F(X,A)$.  Then one checks easily that the diagram
\begin{center}
$\phantom{i}$\xymatrix{
0 \ar[r] & P \ar[r]\ar[d] & P \ar[r]\ar[d] & R \ar[r]\ar[d] & 0\\
0 \ar[r] & R \ar[r] & Q \ar[r] & Q \ar[r] & 0,
}
\end{center}
where every map is either zero or the relevant inclusion, is always a short exact sequence in $\SFMod(Z)$.  Its quotient-sequence is simply $R/P \into Q/P \onto Q/R$.  Since the functor $\Cnd_Z^{Z'}(-)$ simply converts this into the diagram for the concatenation $\Cnd_Z^{Z'}P \leq \Cnd_Z^{Z'}Q \leq \Cnd_Z^{Z'}R$, it is actually stably exact.  \fin
\end{ex}

It is worth noting the following special case of stability: if $Z' := W\times Z$, regarded as an enlargement through the coordinate-injection, then Definition~\ref{dfn:stable} implies that for a stable exact sequence, the following sequence is also exact:
\begin{center}
$\phantom{i}$\xymatrix{ (\F(W,P))_\sp \ar[r]\ar[d] & (\F(W,P'))_\sp \ar[r]\ar[d] & (\F(W,P''))_\sp \ar[d]\\
(\F(W,Q))_\sp \ar[r] & (\F(W,Q'))_\sp \ar[r] & (\F(W,Q''))_\sp.
}
\end{center}

\subsection{Step-polynomial representatives for quotients}

Suppose now that $P \leq Q$ is a semi-functional $Z$-module with base $X$ and fibre $A$.  In many of the examples that appear later, the quotient $Q/P$ will be `small': specifically, an SP module.  We may therefore regard $Q/P$ as a new functional module with base the trivial group $\{0\}$ and with fibre $Q/P$.  With this interpretation, it will be important to study complexity-boundedness or SP pre-images for the quotient map $Q \onto Q/P$.

Most of the results we need in this direction will stem from the next proposition.  It is fairly simple, but it captures a principle that will be vital to many of the more intricate results later.

\begin{prop}\label{prop:complexity-above-and-below}
Suppose that $P$ is a functional $W$-module with base $\a:W\to X$ and SP fibre $A$, and that $\phi:P\to B$ is a continuous $W$-homomorphism to a discrete $W$-module.  Interpret every element of $B$ as a step polynomial, so $\phi$ is trivially complexity-bounded. Then:
\begin{enumerate}
\item[(1)] the homomorphism $\phi$ is stably complexity-bounded;
\item[(2)] if $\phi$ has SP pre-images, then it actually has stable SP pre-images.
\end{enumerate}
\end{prop}

\begin{proof}
Let $Z\geq W$ be an arbitrary enlargement, and let $V:= Z/W$. Proposition~\ref{prop:s-p-coind-consistent} gives a commutative diagram
\begin{center}
$\phantom{i}$\xymatrix{
\Cnd_W^Z P \ar^-{\Cnd_W^Z\phi}[rr]\ar_\cong[d] && \Cnd_W^Z D \ar^\cong[d] \\
\F(V,P) \ar_-{\F(V,\phi)}[rr] && \F(V,B)
}
\end{center}
in which both vertical isomorphisms are complexity-bounded in both directions.  It therefore suffices to prove parts (1) and (2) for the bottom homomorphism: equivalently, in case $Z = V\times W$.

\vspace{7pt}

\emph{Proof of part (1).}\quad Let $f:V \times X \to A$ be a step-polynomial whose Haar-a.e. equivalence class lies in $\F(V,P)$.

Corollary~\ref{cor:connected-images} gives a QP partition $\frP_1$ of $V$ such that for each $C \in \frP_1$, the map
\[\Phi_C:C\to P: v \mapsto f(v,\cdot)\]
is continuous.

Since $f$ is a step polynomial, there are Lie-group quotients $q:V \onto \ol{V}$ and $q_1:X\onto \ol{X}$ such that $f$ factorizes as $\ol{f}\circ (q\times q_1)$ for some step polynomial
\[\ol{f}:\ol{V}\times \ol{X} \to A,\]
and such that $\frP_1$ is lifted from a QP partition $\ol{\frP}_1$ of $\ol{V}$.  It follows that the above maps $\Phi_C$ are pulled back from the corresponding maps
\[\ol{\Phi}_D:D\to \F(\ol{X},A): \ol{v} \mapsto \ol{f}(\ol{v},\cdot), \quad D \in \ol{\frP}_1.\]

Now Lemma~\ref{cor:conn-cpts-still-qp} provides that the refinement $\ol{\frP}_2$ of $\ol{\frP}_1$ into its connected components is still QP  By restricting the above maps $\ol{\Phi}_D$ to these connected components, we may therefore assume that $\ol{\frP}_2 = \ol{\frP}_1$. Having done so, each $D \in \ol{\frP}_1$ is connected, so has connected image under the continuous map $\ol{\Phi}_D$.  Since $B$ is discrete, each of these connected $\ol{\Phi}_D$-images lies in a single fibre of $\phi$.  It follows that the map
\[V\to B: v \mapsto \phi(f(v,\cdot))\]
is a step function whose level-set partition is refined by $\frP_1$, as required.

\vspace{7pt}

\emph{Proof of part (2).}\quad Now suppose that $\phi$ has SP pre-images.  By replacing $B$ with its (necessarily discrete) submodule $\phi(P)$, we may assume that $\phi$ is surjective.

Let $\ol{f}:V\to B$ be a step polynomial: that is, a step function.  Let $\frP$ be its level-set partition, so this is a QP partition of $V$, and for each $C \in \frP$ let $f_C \in P$ be a step-polynomial $\phi$-pre-image of the single value that $\ol{f}$ takes on $C$.  Now define $f:V\to P$ by
\[f(v,x) = f_C(x) \quad \forall (v,x) \in C\times X,\ C \in \frP.\]
This is a step-polynomial pre-image of $\ol{f}$.
\end{proof}

In a few places, we will use parts (1) and (2) of Proposition~\ref{prop:complexity-above-and-below} in combination, via the following slightly fiddly corollary.

\begin{cor}\label{cor:complexity-above-and-below}
Suppose that $Q$ and $R$ are functional $W$-modules (not necessarily with the same base or fibre), that $B$ is a discrete $W$-module, and that one has a diagram of continuous $W$-homomorphisms
\begin{center}
$\phantom{i}$\xymatrix{
& R \ar^\psi[d]\\
Q \ar_\phi[r] & B.
}
\end{center}
Suppose further that this diagram has the following property:
\begin{quote}
If $r \in R_\sp$ and $\psi(r) \in \phi(Q)$, then $\psi(r) \in \phi(Q_\sp)$.
\end{quote}
Then this property also holds for the co-induction of the above diagram to any enlargement $Z\geq W$.
\end{cor}

\begin{proof}
Let $B_1 := \psi(R_\sp)\cap \phi(Q)$, a (necessarily closed) submodule of $B$, and let $Q_1 := \phi^{-1}(B_1)$, so this is a closed submodule of $Q$.  The assumed property asserts that $\phi|Q_1:Q_1 \to B_1$ has SP pre-images.  Therefore, since $B_1$ is discrete, it actually has stable SP pre-images, by part (2) of Proposition~\ref{prop:complexity-above-and-below}.

Now let $\Psi := \Cnd_W^Z\psi$ and $\Phi := \Cnd_W^Z \phi$.  Let $r \in (\Cnd_W^Z R)_\sp$ be such that $\Psi(r) \in \Phi(\Cnd_W^Z Q) = \Cnd_W^Z\phi(Q)$.  Regarded as an element of $\F(Z,R)$, the fact that $r$ is step-polynomial implies that $r(z,\cdot) \in R_\sp$ for every $z$, and therefore $\Psi(r) \in \Cnd_W^Z B_1$.  Applying part (1) of Proposition~\ref{prop:complexity-above-and-below}, it actually lies in $(\Cnd_W^Z B_1)_\sp$.  Since $\phi|Q_1$ has stable SP pre-images, this is equal to the image of $(\Cnd_W^Z Q_1)_\sp$ under $\Phi$, so this completes the proof.
\end{proof}

Now suppose that $P\leq Q$ is a semi-functional $Z$-module whose quotient $Q/P$ is SP  We will have to refer frequently to SP pre-images for the quotient map $Q \onto Q/P$, so these merit their own definition.

\begin{dfn}[Step-polynomial representatives]\label{dfn:step-fn-reps}
The inclusion $P \leq Q$ has \textbf{SP representatives} if the quotient $Q/P$ is a SP Lie $Z$-module and the quotient homomorphism $Q\onto Q/P$ admits SP pre-images, where we interpret $(Q/P)_{\rm{sp}} = Q/P$.  Similarly, if $Q$ is a functional $Z$-module, $A$ is a SP Lie $Z$-module, and $\phi:Q \onto A$ is a surjective $Z$-homomorphism, then $\phi$ has \textbf{SP representatives} if $(\ker\phi \leq Q)$ has SP representatives.

It has \textbf{stable SP representatives} if that quotient homomorphism is stably complexity-bounded and admits stable SP pre-images.
\end{dfn}

By part (2) of Proposition~\ref{prop:complexity-above-and-below}, if $Q/P$ is discrete and $P \leq Q$ has SP representatives, then it automatically has stable SP representatives.

Given our convention that all elements of a SP Lie module such as $Q/P$ are step polynomial, we see that if $P \leq Q$ has quotient which is an SP Lie module, then it has SP representatives if and only if $P + Q_{\rm{sp}} = Q$.

Proposition~\ref{prop:complexity-above-and-below} and the content of Definition~\ref{dfn:step-fn-reps} can be combined in the following useful way.  The proof is an immediate verification.

\begin{lem}\label{lem:stable-collapse}
Let $P \leq Q$ be a semi-functional $Z$-module whose quotient $Q/P$ is an SP Lie module.  The following are equivalent:
\begin{itemize}
\item[i)] the quotient homomorphism $Q\onto Q/P$ has (resp. stable) SP representatives;
\item[ii)] the diagram
\begin{center}
$\phantom{i}$\xymatrix{
0 \ar[r] & P \ar[r]\ar[d] & 0 \ar[r]\ar[d] & 0\\
0 \ar[r] & Q \ar[r] & Q/P \ar[r] & 0
}
\end{center}
defines a sequence of morphisms in $\SFMod(Z)$ which is exact in $\SFMod(Z)$ (resp. stably exact). \qed
\end{itemize}
\end{lem}

\begin{rmks}
\emph{(1)} \quad This lemma provides examples that illustrate the following feature of the category $\SFMod(Z)$: there can be stable short exact sequences
\[0 \to (P\leq Q) \to (P' \leq Q') \to 0\]
which are nevertheless not isomorphisms, in the sense that there is no well-defined inverse $(P' \leq Q')\to (P \leq Q)$.

\vspace{7pt}

\emph{(2)}\quad Although we will not need it in the sequel, the following fact about SP representatives may help build the reader's understanding.

\begin{lem}
If $P \leq Q \leq \F(Z,A)$ are closed subgroups and $P + Q_{\rm{sp}} = Q$, then $Q/P$ is locally compact.
\end{lem}

To prove this, let $\chi_1$, $\chi_2$, \ldots be an enumeration of all homomorphisms from $Y$ to finite-dimensional tori, and now for $f \in \F_{\rm{sp}}(Z,A)$, let $\rm{Cplx}(f)$ denote the sum of the least $i$ such that $f$ factorizes through $\chi_i$, and of the number of linear inequalities and of all the degrees and coefficients needed to specify the convex polytopes and polynomials that go into the definition of $f$, if one uses that $\chi_i$ in the definition.  (This is similar to the notion of complexity that will be set up in Definition~\ref{dfn:cplxty}, except that $\rm{Cplx}$ also keeps track of which quotient $\chi_i$ one uses to factorize $f$.)

Let $q:Q\to Q/P$ be the quotient homomorphism, and for each $n$ let
\[Q_{\leq n} := \{f \in Q_{\rm{sp}}\,|\ \rm{Cplx}(f) \leq n\}.\]
An easy argument shows that each $Q_{\leq n}$ compact (or see Corollary~\ref{cor:cptness}), and hence $Q/P = q(Q_{\rm{sp}}) = \bigcup_{n\geq 1}q(Q_{\leq n})$ is $\s$-compact.  By the Baire Category Theorem, some $q(Q_{\leq n})$ must be co-meager inside some nonempty open set $U$ in $Q/P$.  This now implies that $q(Q_{\leq n}) - q(Q_{\leq n})$ is a compact neighbourhood of the identity in $Q/P$, so $Q/P$ is locally compact, as required.

In view of this, the assumption that $Q/P$ is SP in Definition~\ref{dfn:step-fn-reps} is only slightly more restrictive than simply asserting directly that $P + Q_{\rm{sp}} = Q$. \fin
\end{rmks}

The next simple property of SP representatives will be used repeatedly.

\begin{lem}\label{lem:step-reps-of-common-enlargement}
Suppose that $P \leq Q \leq \F(X,A)$ is a semi-functional $Z$-module with $Q/P$ discrete and with step-function representatives, and that $D \leq \F(X,A)$ is another $Z$-submodule such that $P + D$ and $Q+D$ are closed.  Then $P + D \leq Q + D$ has step-function representatives.
\end{lem}

\begin{proof}
First, $(Q+D)/(P+D)$ is Polish by assumption, and is a quotient of $Q/P$, so is still an SP $Z$-module.

If $q + d \in Q + D$, then it equals $q' + p + d$ for some $q' \in Q_\sp$ and some $p \in P$, so it agrees with $q'$ modulo $P+D$.
\end{proof}

%\begin{cor}\label{cor:combining-SHC-quot}
%Suppose that $P\leq Q \leq \F(Z,A)$ has s.-.p. representatives, that $M := Q/P$ is discrete, and that $P' \leq Q' \leq \F(Z',M)$ has step-function representatives.  Then the combined inclusion
%\[\pi^{-1}P' \leq \pi^{-1}Q'\]
%has SP representatives in $\F(Z'\times Z,A)$, where $\pi:\F(Z',Q) \to \F(Z',M)$ is defined by
%\[\pi f(z') = f(z',\cdot) + P \in Q/P,\]
%and where we identify $\pi^{-1}P'$ and $\pi^{-1}Q'$ as closed subgroups of $\F(Z'\times Z,A)$.
%\end{cor}
%
%\begin{proof}
%Let $f \in \pi^{-1}Q'$.  Then $\pi f \in Q'$, so there is a step function $\ol{g} \in Q'$ such that $\pi f - \ol{g} \in P'$.
%
%Since $\ol{g}$ is a step function, its level-set partition $\frP$ is QP  Also, since $P \leq Q$ has step-polynomial representatives, for each $C \in \frP$ we may choose a step polynomial $g_C \in Q$ such that $g_C + P = \ol{g}(z') + P$ for all $z' \in C$.  Now define $g \in \F(Z'\times Z,A)$ by requiring that
%\[g(z',z) := g_C(z) \quad \forall (z',z) \in C\times Z.\]
%
%This function has the property that $\pi g(z') = g_C + P = \ol{g}(z')$ whenever $z' \in C$, so $g \in \pi^{-1}Q'$ and $\pi(f - g) \in P'$, hence $f \in g + \pi^{-1}P'$.  Since $g$ is clearly a step polynomial, this completes the proof.
%\end{proof}
%

\subsection{Semi-functional complexes}\label{subs:semi-fun-cplx}

As with the Polish modules studied in Part I, we will soon have a need for whole complexes of semi-functional modules.

\begin{dfn}[Semi-functional complex]\label{dfn:semi-fnl-cplx}
A \textbf{semi-functional $Z$-complex} is a diagram of the form
\begin{center}
$\phantom{i}$ \xymatrix{
\cdots \ar[r] & P_\ell \ar[r]\ar^\incln[d] & P_{\ell+1} \ar[r]\ar^\incln[d] & P_{\ell+2}\ar^\incln[d]\ar[r] & \cdots \\
\cdots \ar^{\a_\ell}[r] & Q_\ell \ar^{\a_{\ell+1}}[r] & Q_{\ell+1} \ar^{\a_{\ell+2}}[r] & Q_{\ell+2} \ar^{\a_{\ell+3}}[r] & \cdots \\
}
\end{center}
where every column is a semi-functional $Z$-module with an SP fibre (possibly different for different $i$), and every column of horizontal arrows defines a semi-functional morphism.

Such a diagram will often be abbreviated to $P_\bullet \leq Q_\bullet$, when no confusion can arise.

A semi-functional complex is \textbf{left-bounded} if all modules sufficiently far to the left in this diagram are $0$.
\end{dfn}

\begin{dfn}
Given a semi-functional complex as above, its \textbf{quotient complex} is the resulting Polish complex
\[\cdots \to Q_\ell/P_\ell \to Q_{\ell+1}/P_{\ell+1} \to Q_{\ell+2}/P_{\ell+2} \to \cdots.\]

The \textbf{homology} of the semi-functional complex $P_\bullet \leq Q_\bullet$ is defined to be the homology of its quotient complex: in position $\ell$ this is
\[\frac{\ker(Q_\ell/P_\ell \to Q_{\ell+1}/P_{\ell+1})}{\img(Q_{\ell-1}/P_{\ell-1}\to Q_\ell/P_\ell)} \cong \frac{\a_{\ell+1}^{-1}(P_{\ell+1})}{\a_\ell(Q_{\ell-1}) + P_\ell}\]
(where this is easily seen to be a topological isomorphism, even if these quotients are not Hausdorff: compare Lemma~\ref{lem:coker-top} below).
\end{dfn}

Let $P_\bullet \leq Q_\bullet$ be a semi-functional $Z$-complex.  We will need to refer to several properties concerning step-polynomial elements of the modules appearing there.  The simplest are the following.

\begin{dfn}[Complexity-bounded complex]
The semi-functional complex is (resp. \textbf{stably}) \textbf{complexity-bounded} if every morphism $\a_\ell$ is (resp. stably) complexity-bounded.
\end{dfn}

\begin{dfn}[Finite-complexity decompositions]\label{dfn:decomp-cplx}
A semi-functional complex has (resp. \textbf{stable}) \textbf{finite-complexity decompositions at position $\ell$} if the homomorphism
\[Q_{\ell-1} \oplus P_\ell \to Q_\ell:(q,p)\mapsto \a_\ell(q) + p\]
admits (resp. stable) SP pre-images.  The complex has (resp. \textbf{stable}) \textbf{finite-complexity decompositions} if this holds at all positions.
\end{dfn}

The next lemma sometimes provides a convenient way to understand finite-complexity decompositions.  Given $P_\bullet \leq Q_\bullet$ as above, it has an SP subdiagram: in this case, this is an inclusion of complexes of abstract $Z$-modules, which we may write $P_{\sp,\bullet} \leq Q_{\sp,\bullet}$.  Let $M_\bullet = Q_\bullet/P_\bullet$ be the quotient complex, and similarly let $\t{M}_\bullet := Q_{\sp,\bullet}/P_{\sp,\bullet}$. Let $H_\bullet$ be the homology of $M_\bullet$, and $\t{H}_\bullet$ that of $\t{M}_\bullet$.

The inclusions $Q_{\sp,\bullet} \subseteq Q_\bullet$ restrict to inclusions $P_{\sp,\bullet} \subseteq P_\bullet$, and therefore induce a commutative diagram of complexes
\begin{center}
$\phantom{i}$\xymatrix{
P_{\sp,\bullet} \ar_\incln[d] \ar[r] & P_\bullet \ar^\incln[d]
\\ Q_{\sp,\bullet} \ar@{->>}[d] \ar[r] & Q_\bullet \ar@{->>}[d]\\
\t{M}_\bullet \ar[r] & M_\bullet.
}
\end{center}
This, in turn, induces a sequence of homomorphisms on the homologies of the two bottom rows of this diagram:
\begin{eqnarray}\label{eq:homol-mors}
\t{H}_\bullet \to H_\bullet.
\end{eqnarray}

\begin{lem}\label{lem:criterion-for-finite-complex}
In the situation above, the original complex has finite-complexity decompositions at position $\ell$ if and only if the map $\t{H}_\ell \to H_\ell$ in~(\ref{eq:homol-mors}) is injective.
\end{lem}

\begin{proof}
This is just a matter of chasing the definitions: $P_\bullet \leq Q_\bullet$ has finite-complexity decompositions at position $\ell$ if and only if
\begin{eqnarray*}
&&(\a_\ell(Q_{\ell-1}) + P_\ell)_\sp = \a_\ell(Q_{\sp,\ell-1}) + P_{\sp,\ell}\\
&& \Longleftrightarrow \quad Q_{\sp,\ell} \cap \big(\a_\ell(Q_{\ell-1}) + P_\ell\big) = \a_\ell(Q_{\sp,\ell-1}) + P_{\sp,\ell}\\
&& \Longleftrightarrow \quad \t{H}_\ell \stackrel{\rm{dfn}}{=} \frac{Q_{\sp,\ell}\cap \a_{\ell+1}^{-1}(P_{\sp,\ell+1})}{\a_\ell(Q_{\sp,\ell-1}) + P_{\sp,\ell}} \to \frac{\a_{\ell+1}^{-1}(P_{\ell+1})}{\a_\ell(Q_{\ell-1}) + P_\ell} \stackrel{\rm{dfn}}{=} H_\ell \quad \hbox{is injective}.
\end{eqnarray*}
\end{proof}

Now consider a stable short exact sequence of semi-functional complexes:
\begin{center}
$\phantom{i}$\xymatrix{
0 \ar[r] & P_\bullet \ar[r]\ar_\incln[d] & P'_\bullet \ar[r]\ar^\incln[d] & P''_\bullet \ar[r]\ar^\incln[d] & 0\\
0 \ar[r] & Q_\bullet \ar[r] & Q'_\bullet \ar[r] & Q''_\bullet \ar[r] & 0.
}
\end{center}
This means, of course, that every column is a semi-functional complex, and that this diagram becomes a stable short exact sequence in $\SFMod(Z)$ for every particular position $\ell$ in those complexes.

Example~\ref{ex:concat} generalizes easily to this setting.

\begin{ex}\label{eq:concat-cplx}
A \textbf{concatenation} of semi-functional complexes is a sequence of concatenations
\[P_\ell \leq R_\ell \leq Q_\ell \leq \F(X_\ell,A_\ell)\]
together with homomorphisms $\a_\ell:Q_{\ell-1}\to Q_\ell$ such that $\a_\ell(R_{\ell-1}) \leq R_\ell$ and $\a_\ell(P_{\ell-1}) \leq P_\ell$ for each $\ell$. Given such a concatenation, one obtains a stable short exact sequence of semi-functional complexes as follows:
\begin{center}
$\phantom{i}$\xymatrix{
0 \ar[r] & P_\bullet \ar[r]\ar_\incln[d] & P_\bullet \ar[r]\ar^\incln[d] & R_\bullet \ar[r]\ar^\incln[d] & 0\\
0 \ar[r] & R_\bullet \ar[r] & Q_\bullet \ar[r] & Q_\bullet \ar[r] & 0.
}
\end{center}
$\phantom{i}$ \fin
\end{ex}

Now consider again a general short exact sequence of semi-functional complexes as above. Let $\a_\bullet$, $\a'_\bullet$ and $\a''_\bullet$ be the sequences of morphisms of $Q_\bullet$, $Q'_\bullet$ and $Q''_\bullet$, respectively, and let $M_\bullet$, $M'_\bullet$ and $M''_\bullet$ be the quotient complexes. Similarly to the previous subsection, let $\t{M}'_\bullet$, $\t{M}'_\bullet$ and $\t{M}''_\bullet$ be the complexes of quotients arising from the step-polynomial subdiagrams of the above semi-functional complexes.  As in~(\ref{eq:homol-mors}), we have natural maps of the resulting homology groups
\begin{eqnarray}\label{eq:maps-on-homol}
\t{H}_\ell\to H_\ell, \quad \t{H}'_\ell \to H'_\ell \quad \hbox{and} \quad \t{H}''_\ell \to H''_\ell
\end{eqnarray}

Since our short exact sequences are assumed to be exact in$\SFMod(Z)$, the SP subdiagram of the original short exact sequence of complexes defines a new short exact sequence of complexes.  It maps to the original short exact sequence via all the inclusion maps such as $P_{\sp,\ell} \into P_\ell$, so we now have the following.

\begin{lem}\label{lem:compar-diag}
In the above setting, the homomorphisms in~(\ref{eq:maps-on-homol}) fit together into a homomorphism of the resulting long exact sequences:
\begin{center}
$\phantom{i}$\xymatrix{
\cdots \ar[r] & \t{H}_\ell \ar[r]\ar[d] & \t{H}'_\ell \ar[r]\ar[d] & \t{H}''_\ell \ar[r]\ar[d] & \t{H}_{\ell+1} \ar[r]\ar[d] & \t{H}'_{\ell+1} \ar[r]\ar[d] & \cdots\\
\cdots \ar[r] & H_\ell \ar[r] & H'_\ell \ar[r] & H''_\ell \ar[r] & H_{\ell+1} \ar[r] & H'_{\ell+1} \ar[r] & \cdots.
}
\end{center}
Since both rows here are the long exact sequences arising from certain short exact sequences of complexes, both are exact.

We will call this the \textbf{comparison diagram}. \qed
\end{lem}

\subsection{Semi-functional complexes with fully SP homology}

We now focus on a more restricted class of semi-functional complexes.  In general, if the homology groups $H_\bullet$ of a semi-functional $Z$-complex $P_\bullet \leq Q_\bullet$ are Hausdorff, then they may be viewed as the quotients of the semi-functional modules
\[\a_\ell(Q_{\ell-1}) + P_\ell \leq \a_{\ell+1}^{-1}(P_{\ell+1}).\]
In some special cases, the homology groups turn out to be 'small':

\begin{dfn}
A semi-functional $Z$-complex $P_\bullet \leq Q_\bullet$ is \textbf{locally SP} if it is stably complexity-bounded, if its homology groups $H_\bullet$ are all SP Lie $Z$-modules, and if the quotient homomorphisms
\begin{eqnarray}\label{eq:homol-quot}
\a_\ell(Q_{\ell-1}) + P_\ell \leq \a_{\ell+1}^{-1}(P_{\ell+1}) \onto H_\ell
\end{eqnarray}
are also all stably complexity-bounded.
\end{dfn}

Recall that Subsection I.2.4 introduced Polish complexes that have $\ell_0$-almost discrete homology groups: this means that all modules to the left of $\ell_0$ are zero, the homology at $\ell_0$ is Lie, and the homology at all positions to the right of $\ell_0$ is discrete.  The examples of locally SP complexes that we will meet below all have almost discrete homology, so it is worth knowing how these notions interact.

\begin{lem}\label{lem:almost-disc-and-locally-SP}
Let $(P_\bullet \leq Q_\bullet)$ be a left-bounded semi-functional complex with homomorphisms $\a_\bullet$, quotient complex $M_\bullet$, and homology groups $H_\bullet$.  Suppose that $M_\bullet$ has $\ell_0$-almost discrete homology.  Then $(P_\bullet \leq Q_\bullet)$ is locally SP if and only if $H_{\ell_0}$ is an SP module and the quotient homomorphism
\[P_{\ell_0} \leq \a_{\ell_0+1}^{-1}(P_{\ell_0+1}) \onto H_{\ell_0}\]
is stably complexity-bounded.
\end{lem}

\begin{proof}
The point is simply that these conditions do not need to be checked at any positions to the right of $\ell_0$: since $H_\ell$ is discrete for every $\ell > \ell_0$, it is SP, and stable complexity-boundedness of the quotient homomorphisms in~(\ref{eq:homol-quot}) is given by part (1) of Proposition~\ref{prop:complexity-above-and-below}.
\end{proof}

Often we will need to handle complexes that are co-induced from locally SP complexes.  Now let $Y\leq Z$ be an inclusion of compact Abelian groups, and suppose that $P_\bullet^\circ \leq Q_\bullet^\circ$ is a locally SP semi-functional $Y$-complex.  Let $H^\circ_\bullet$ be its homology groups, and let
\begin{eqnarray}\label{eq:semi-good}
(P_\bullet \leq Q_\bullet) := \Cnd_Y^Z(P_\bullet^\circ \leq Q_\bullet^\circ).
\end{eqnarray}
The homology groups of $(P_\bullet \leq Q_\bullet)$ are simply $H_\bullet \cong \Cnd_Y^Z H^\circ_\bullet$.  Since each $H^\circ_\ell$ is an SP $Y$-module, each $H_\ell$ may therefore be interpreted as a functional $Z$-module with base $\rm{id}_Z$ and fibre $H^\circ_\ell$.  This gives a direct meaning to `step polynomial' elements in these homology groups.  Any further co-induction gives an analogous structure.  With this interpretation, we may make the following definition.

\begin{dfn}[SP representatives]\label{dfn:s-p-reps-cplx}
Let $P^\circ_\bullet \leq Q^\circ_\bullet$, $H^\circ_\bullet$, $P_\bullet \leq Q_\bullet$ and $H_\bullet$ be as above.  Then $P_\bullet \leq Q_\bullet$ has (resp. \textbf{stable}) \textbf{SP representatives at position $\ell$} if the quotient homomorphism
\[\a_\ell(Q_{\ell-1}) + P_\ell \leq \a_{\ell+1}^{-1}(P_{\ell+1}) \onto H_\ell\]
has (resp. stable) SP pre-images (where the `step polynomials' in the quotient module here are understood as described above).  It has (resp. \textbf{stable}) \textbf{SP representatives} if this holds at all positions.
\end{dfn}

Often we will work with complexes satisfying both of Definitions~\ref{dfn:s-p-reps-cplx} and~\ref{dfn:decomp-cplx}.

\begin{dfn}[Fully SP homology]
Again let $P_\bullet \leq Q_\bullet$ be the co-induction to $Z$  of a locally SP $Y$-complex.  It has (resp. \textbf{stably}) \textbf{fully SP homology} if it has both (resp. stable) finite-complexity decompositions and (resp. stable) SP representatives.
\end{dfn}

Letting $H_\bullet$ be the homology of the $Z$-complex in the above definition, it follows that it has (resp. stably) fully SP homology if and only if the sequence
\begin{center}
$\phantom{i}$\xymatrix{
P_{\ell-1} \ar[r]\ar[d] & P_\ell \ar[r]\ar[d] & 0 \ar[r]\ar[d] & 0\\
Q_{\ell-1} \ar[r] & \a_{\ell+1}^{-1}(P_{\ell+1}) \ar[r] & H_\ell \ar[r] & 0
}
\end{center}
is (resp. stably) exact in $\SFMod(Z)$ for every $\ell$.

Next consider the comparison homomorphisms~(\ref{eq:homol-mors}) in the present setting.  Since the quotient homomorphisms~(\ref{eq:homol-quot}) are stably complexity-bounded, the comparison homomorphisms must take values in $H_{\sp,\bullet}$.  This leads to the following simple counterpart of Lemma~\ref{lem:criterion-for-finite-complex}.

\begin{lem}\label{lem:criterion-for-meekness}
In the situation above, the complex has SP representatives if and only if the maps $\t{H}_\ell \to H_{\sp,\ell}$ are all surjective.  Hence, it has fully SP homology if and only if these maps are isomorphisms. \qed
\end{lem}

Consider again a stable short exact sequence
\[0 \to (P^\circ_\bullet \leq Q^\circ_\bullet) \to (P'^\circ_\bullet \leq Q'^\circ_\bullet) \to (P''^\circ_\bullet \leq Q''^\circ_\bullet) \to 0\]
of semi-functional $Y$-complexes.  Assume now that all three of them are locally SP.  Let
\begin{eqnarray}\label{eq:resulting-Z-cplxs}
0 \to (P_\bullet \leq Q_\bullet) \to (P'_\bullet \leq Q'_\bullet) \to (P''_\bullet \leq Q''_\bullet) \to 0
\end{eqnarray}
be the result of co-inducing to some $Z \geq Y$.  It will be important later that, in this setting, fully SP homology (and hence also stably fully SP homology) is preserved by such stable short exact sequences.

\begin{prop}\label{prop:complex-ses-good}
In the above setting, if two of the complexes in~(\ref{eq:resulting-Z-cplxs}) have fully SP homology, then so does the third.
\end{prop}

\begin{proof}
This involves three cases: we treat only one, since they are all much the same. Suppose that $P_\bullet \leq Q_\bullet$ and $P'_\bullet \leq Q'_\bullet$ both have fully SP homology.

Consider the comparison diagram of Lemma~\ref{lem:compar-diag}.  All the images from the top row give step polynomials in the relevant groups of the bottom row, as in Lemma~\ref{lem:criterion-for-meekness}.  We may therefore restrict the bottom row of the comparison diagram to its SP subdiagram, and so obtain the following:
\begin{center}
$\phantom{i}$\xymatrix{
\cdots \ar[r] & \t{H}_\ell \ar[r]\ar[d] & \t{H}'_\ell \ar[r]\ar[d] & \t{H}''_\ell \ar[r]\ar[d] & \t{H}_{\ell+1} \ar[r]\ar[d] & \t{H}'_{\ell+1} \ar[r]\ar[d] & \cdots\\
\cdots \ar[r] & H_{\sp,\ell} \ar[r] & H'_{\sp,\ell} \ar[r] & H''_{\sp,\ell} \ar[r] & H_{\sp,\ell+1} \ar[r] & H'_{\sp,\ell+1} \ar[r] & \cdots.
}
\end{center}
The new bottom row is still exact, because it is the SP subdiagram of an exact sequence of closed morphisms all of which are co-induced from morphisms of SP Lie $Y$-modules, so are stably complexity-bounded (obvious) and have stable SP pre-images (by Corollary~\ref{cor:Lie-homos-s-p-reps}).

Finally, our assumptions give that the first, second, fourth and fifth vertical arrows here are isomorphisms, by Lemmas~\ref{lem:criterion-for-finite-complex} and~\ref{lem:criterion-for-meekness}, and hence so is the middle vertical arrow, by the Five Lemma~\cite[Proposition 2.72]{Rotman09}.
\end{proof}

\section{Cohomology theories for semi-functional modules}

As in Part I, a central r\^ole in this paper will be played by the cohomology of our various $Z$-modules of interest.  In Part I, the appropriate cohomology theory was that constructed by Calvin Moore in terms of measurable cochains.  A more complete overview of this was given in Section I.3.  Here we will need this theory again, but also another, constructed in terms of step-polynomial cochains.  This section will derive some basic properties of the former for semi-functional modules, then introduce the latter, and finally prove some comparison results between them.

\subsection{Measurable cohomology}\label{subs:semi-fnl-cohom}

We will bring forward from Part I the notation for measurable cocycles, coboundaries, coboundary operators, cohomology groups and so on.

There is an obvious forgetful functor $\SFMod(Z)\to \PMod(Z)$ which acts on objects by
\[\big(P \leq Q\big) \mapsto Q/P\]
and on morphisms in the corresponding way. Measurable cohomology for objects of $\SFMod(Z)$ will simply be the composition of $\rmH^\ast_\m$ with this forgetful functor: if $P \leq Q$ is a semi-functional $Z$-module, then its measurable cohomology in degree $p$ will be defined to be $\rmH^p_\m(Z,Q/P)$.

It will be most important to us that these groups $\rmH^p_\m(Z,Q/P)$ also admit descriptions as quotients of pairs of functional $Z$-modules (although these quotients may not be Hausdorff).  Indeed, quite generally, when a Polish $W$-module $M$ is given as a quotient of two other Polish $Z$-modules, one can present the group $\rmH^p_\m(Z,M)$ in terms of those other modules.  The following definition is in much the same spirit as relative homology and cohomology in algebraic topology~\cite[Sections 2.1 and 3.1]{Hat02}.

\begin{dfn}[Relative cocycles]
Given a short exact sequence
\[P\into Q\onto M\]
in $\PMod(Z)$, a \textbf{relative cocycle from $Z$ to $(P,Q)$ in degree $p$} is an element of the module
\[\Z^p(Z,Q,P) := \{f \in \C^p(Z,Q)\,|\ df \in \C^{p+1}(Z,P)\},\]
and a \textbf{relative coboundary from $Z$ to $(P,Q)$ in degree $p$} is an element of the module
\[\B^p(Z,Q,P) := \C^p(Z,P) + \B^p(Z,Q).\]
\end{dfn}

Just as for the usual groups of cocycles and coboundaries, $\Z^p(Z,P,Q)$ is always a closed subgroup of $\C^p(Z,Q)$, and the subgroup $\B^p(Z,Q,P)$ is contained in $\Z^p(Z,Q,P)$ but may not be closed.

These modules can be used to give an alternative presentation of the cohomology groups $\rmH^p_\m(Z,Q/P)$.  This is based on the following auxiliary lemma.

\begin{lem}\label{lem:coker-top}
Suppose that $\psi_i:P_i \into Q_i$ are continuous injective homomorphisms of Polish groups for $i=1,2$ with quotient groups $Q_i/\psi_i(P_i) = M_i$ (which are not assumed Hausdorff).  Suppose also that $\phi^P:P_1\to P_2$ and $\phi^Q:Q_1\to Q_2$ are continuous homorphisms which give rise to a commutative diagram
\begin{center}
$\phantom{i}$\xymatrix{
P_1 \ar@{^{(}->}_{\psi_1}[d]\ar^{\phi^P}[r] & P_2 \ar@{^{(}->}^{\psi_2}[d]\\
Q_1 \ar[d]\ar^{\phi^Q}[r] & Q_2 \ar[d]\\
M_1 \ar[r] & M_2.
}
\end{center}
Then the algebraic isomorphism
\[\coker(M_1\to M_2)\cong \frac{Q_2}{\phi^Q(Q_1) + \psi_2(P_2)}\]
is also a topological isomorphism.
\end{lem}

\begin{proof}
The construction of this algebraic isomorphism is a standard diagram chase.  It remains to prove that it is continuous in each direction.

The continuity from right to left is obvious from the definition of the quotient topology, because the composition of quotient maps
\[Q_2 \to M_2 \to \coker(M_1\to M_2)\]
is continuous by definition and its kernel is $\phi^Q(Q_1) + \psi_2(P_2)$.  Similarly, from left to right we see that $\rm{img}(M_1\to M_2)$ is the kernel of the homomorphism
\[M_2 = Q_2/\psi_2(P_2) \to Q_2/(\phi^Q(Q_1) + \psi_2(P_2)),\]
so the map is continuous in this direction also, by the definition of the topology on the cokernel.
\end{proof}

\begin{lem}\label{lem:rel-cocycs}
In the setting of the above definition, the inclusion
\[\B^p(Z,Q,P) \leq \Z^p(Z,Q,P)\]
has quotient topologically isomorphic to $\rmH_\m^p(Z,Q/P)$.
\end{lem}

\begin{proof}
Quotienting both $\B^p(Z,Q,P)$ and $\Z^p(Z,Q,P)$ by the common subgroup $\C^p(Z,P)$ gives a commutative diagram
\begin{center}
$\phantom{i}$\xymatrix{
\B^p(Z,Q,P) \ar@{^{(}->}[r]\ar[d] & \Z^p(Z,Q,P) \ar[d]\\
\B^p(Z,Q/P) \ar@{^{(}->}[r] & \Z^p(Z,Q/P) \ar@{->>}[r] & \rmH^p_\m(Z,Q/P).
}
\end{center}
The first of these vertical maps is obviously surjective, and a simple application of the Measurable Selector Theorem shows that the second is also surjective.  Since $\B^p(Z,Q,P)$ contains
\[\ker\big(\Z^p(Z,Q,P) \to \C^p(Z,Q/P)\big) = \C^p(Z,P),\]
it follows that the sequence
\[0 \to \B^p(Z,Q,P)\to \Z^p(Z,Q,P) \to \rmH^p_\m(Z,Q/P) \to 0\]
is algebraically exact.  Topological isomorphism follows from Lemma~\ref{lem:coker-top}.
\end{proof}

In the sense of Subsection~\ref{subs:semi-fun-cplx}, Lemma~\ref{lem:rel-cocycs} has identified $\rmH^\ast_\m(Z,Q/P)$ with the homology of the semi-functional $Z$-complex
\begin{center}
$\phantom{i}$\xymatrix{
0 \ar[r] & P \ar^-d[r]\ar^\incln[d] & \C^1(Z,P) \ar^-d[r]\ar^\incln[d] & \C^2(Z,P)
\ar^-d[r]\ar^\incln[d] & \cdots\\
0 \ar[r] & Q \ar^-d[r] & \C^1(Z,Q) \ar^-d[r] & \C^2(Z,Q) \ar^-d[r] & \cdots.
}
\end{center}

\begin{dfn}\label{dfn:rel-cochain-cplx}
The above semi-functional $Z$-complex is the \textbf{relative cochain complex} of $(P\leq Q)$.
\end{dfn}

In addition to defining $\rmH^\ast_\m$ for individual semi-functional modules, it will be important to note that one still obtains a long exact sequence relating these groups from a short exact sequence in the category $\SFMod(Z)$.  This is because a short exact sequence in $\SFMod(Z)$ is mapped to a short exact sequence in $\PMod(Z)$ by the forgetful functor, and long exact sequences have been constructed there.

\subsection{Step-polynomial cohomology}\label{subs:semi-sp-cohom}

Our next step is to introduce cohomology with step-polynomial coefficients.  Let $Q \leq \F(X,A)$ be a functional $Z$-module with base $\a:Z\to X$ and SP fibre $A$.  Let $p \geq 0$. Then the \textbf{SP $p$-cochains} with values in $Q$ are the elements of the group
\[\C_\sp^p(Z,Q) := \F_\sp(Z^p,Q),\]
where the right-hand side is understood according to Subsection~\ref{subs:behav-coind}.

Now recall the usual differentials for the inhomogeneous bar resolution, as in Subsection I.3.1:
\begin{eqnarray}\label{eq:dfn-d}
df(z_1,\ldots,z_{p+1},x) &:=& (z_1\cdot f)(z_2,\ldots,z_{p+1},x)\nonumber \\ && + \sum_{i=1}^p(-1)^pf(z_1,\ldots,z_i + z_{i+1},\ldots,z_{p+1},x)\nonumber \\ 
&& + (-1)^{p+1}f(z_1,\ldots,z_p,x).
\end{eqnarray}
If $f$ is a step polynomial, then so is every term on the right. This is obvious for all but the first term, and that term is a step polynomial because $A$ was assumed to be SP.  Moreover, all the appropriate slices on the right are elements of $Q_\sp$, so these differentials define complexity-bounded maps $\C^p(Z,Q)\to \C^{p+1}(Z,Q)$.  Since there are obvious isomorphisms
\[\Cnd_Z^{Z'}\C^p(Z,Q) \cong \C^p(Z,\Cnd_Z^{Z'}Q)\]
which are complexity-bounded in both directions, applying the same reasoning to $\Cnd_Z^{Z'}Q$ shows that these differentials are in fact \emph{stably} complexity-bounded.

Now let $P \leq Q \leq \F(X,A)$ be a semi-functional $Z$-module with base $\a:Z\to X$ and SP fibre $A$.

\begin{dfn}
For each $p \geq 0$, the \textbf{SP relative $p$-cocycles} with values in $P \leq Q$ are the elements of the group
\[\Z_\sp^p(Z,Q,P) := \{f \in \C^p_\sp(Z,Q)\,|\ df \in \C^{p+1}_\sp(Z,P)\},\]
and the \textbf{SP relative $p$-coboundaries} are the elements of the sum
\[\B^p_\sp(Z,Q,P) := d(\C^{p-1}_\sp(Z,Q)) + \C^p_\sp(Z,P).\]
Clearly $\B^p_\sp \subseteq \Z^p_\sp$.
\end{dfn}

\begin{rmk}
Note that $\B^p_\sp(Z,Q,P)$ may not be the same as the set of (Haar-a.e. classes of) step polynomials in $\B^p(Z,Q,P)$.  When we need the latter, it will be denoted by $(\B^p(Z,Q,P))_\sp$. \fin
\end{rmk}

\begin{dfn}
With $P \leq Q$ as above, the \textbf{SP cohomology} of $P \leq Q$ in degree $p$ is the quotient group
\[\rmH^p_\sp(Z,P\leq Q) := \Z_\sp^p(Z,Q,P) \big/ \B_\sp^p(Z,Q,P).\]
\end{dfn}

This definition is easily extended to see that each $\rmH_\sp^p(Z,-)$ defines a functor from $\SPMod(Z)$ to Abelian groups.

Similarly to the case of $\rmH^\ast_\m$, we now recognize that $\rmH^\ast_\sp(Z,P\leq Q)$ is precisely the homology of the step-polynomial subdiagram of the relative cochain complex of $P\leq Q$ (Definition~\ref{dfn:rel-cochain-cplx}).  The reasoning about $d$ above gives that this semi-functional complex is stably complexity-bounded, so this step-polynomial subdiagram is well-defined.

Having made this observation, now consider a stable short exact sequence
\[ 0 \to (P \leq Q) \to (P' \leq Q') \to (P'' \leq Q'') \to 0\]
in $\SFMod(Z)$.  Applying $\C^p(Z,-)$ for each $p$, it gives rise to a stable short exact sequence of the corresponding relative cochain complexes. Note that we really need the stability of the original short exact sequence here, in order that each of the sequences of cochain-module-inclusions
\begin{multline*}
0 \to (\C^p(Z,P) \leq \C^p(Z,Q)) \to (\C^p(Z,P') \leq \C^p(Z,Q'))\\ \to (\C^p(Z,P'') \leq \C^p(Z,Q'')) \to 0
\end{multline*}
still be an exact sequence in $\SFMod(Z)$ for $p \geq 1$.

Applying the abstract construction of long exact sequences (see, for instance,~\cite[Theorem 6.10]{Rotman09}) to this step-polynomial sub-diagram, we now obtain a long exact sequence for $\rmH^\ast_\sp$ corresponding to a stable short exact sequence in $\SFMod(Z)$.

\begin{rmk}
At this point, we have more or less covered all of the abstract properties of $\rmH^\ast_\m$ and $\rmH^\ast_\sp$ that we will use below.  However, it seems clear that one could develop both theories for $\SFMod(Z)$ considerably further.

The basis for this development would be the abstract structure of the category $\SFMod(Z)$. Letting $\mathsf{E}(Z)$ be the collection of all stable short exact sequences in $\SFMod(Z)$, one can show that the pair $(\SFMod(Z),\mathsf{E}(Z))$ is actually an \emph{exact category} in the sense of Quillen~\cite{Qui73}.  This is a certain kind of enrichment of an additive category: see, for instance,~\cite{Buh10} for a thorough survey.  It enables one to do homological algebra even when the underlying additive category, such as our $\SFMod(Z)$, is not Abelian (in our case, this is because injective homomorphisms need not have closed image: in categorial terms, `monos' are not always `kernels').

Given this exact-category structure, one can then show that both $\rmH^\ast_\m$ and $\rmH^\ast_\sp$ are cohomological functors from $(\SFMod(Z),\mathsf{E}(Z))$ to sequences of Abelian groups: this just means they have long exact sequences, as we have seen.  Also, both are effaceable in a canonical way: for any $(P\leq Q) \in \SFMod(Z)$, the semi-functional morphism
\[(P\leq Q)\to \F(Z,P \leq Q)\]
given by the inclusion $Q \into \F(Z,Q)$ as the constant functions can be extended to a stable short exact sequence $(P\leq Q) \into (P'\leq Q')\onto (P'' \leq Q'')$ such that
\[\rmH^p_\m(Z,Q'/P') = \rmH^\ast_\sp(Z,P' \leq Q') = 0 \quad \forall p \geq 1.\]
From this fact, the universality of these cohomological functors given their respective first terms $\rmH^0_\m$ and $\rmH^0_\sp$ follows by the usual argument.  This, in turn, should lead to analogs of various other standard facts from classical group cohomology, such as the Shapiro Lemma.

We will not use any of these more sophisticated results below, but I am grateful to Theo B\"uhler and Igor Minevich for helping me find my way around these ideas. \fin
\end{rmk}

\subsection{Comparing the cohomology theories}\label{subs:compar}

Having defines $\rmH^\ast_\m$ and $\rmH^\ast_\sp$ on the category $\SFMod(Z)$, we next observe that there are obvious comparison homomorphisms between them.  Indeed, for each $p$ one has inclusions $\C^p_\sp(Z,Q) \leq \C^p(Z,Q)$ which respect the coboundary morphisms, so these inclusions quotient to a sequence of comparison homomorphisms
\begin{eqnarray}\label{eq:compar-cohom}
\rmH^p_\sp(Z,P \leq Q) \to \rmH^p_\m(Z,Q/P).
\end{eqnarray}
In fact, these are nothing but a special case of the comparison homomorphisms in~(\ref{eq:homol-mors}) for the relative chain complex of $P\leq Q$ (Definition~\ref{dfn:rel-cochain-cplx}).

Much of our work later will concern whether these homomorphisms are injective or surjective.  At this point, it will be convenient to consider the action of a closed subgroup $W \leq Z$.

\begin{dfn}
Let $P \leq Q$ be a semi-functional $Z$-module, and let $W \leq Z$.  Then it has \textbf{SP relative-coboundary-solutions over $W$} if the comparison homomorphism $\rmH^p_\sp(W,P\leq Q) \to \rmH^p_\m(W,Q/P)$ is injective for all $p$.

On the other hand, it has \textbf{SP-represented cohomology over $W$} if those comparison maps are all surjective.

Finally, for either of these properties, $P \leq Q$ has that property \textbf{stably} if that property holds for $\Cnd_Z^{Z'}(P \leq Q)$ for every enlargement $Z' \geq Z$ (for the same fixed $W$).
\end{dfn}

More concretely, $P\leq Q$ has:
\begin{itemize}
\item SP relative-coboundary-solutions over $W$ if whenever $f \in \C^p(W,Q)$ is such that $df \in \C_\sp^{p+1}(W,Q) + \C^{p+1}(W,P)$, there is some $f' \in \C^p_\sp(W,Q)$ such that $df = df'$ modulo $\C^{p+1}(W,P)$;
\item SP-represented cohomology over $W$ if every $f \in \Z^p(W,Q,P)$ is relatively cohomologous to an element of $\Z^p_\sp(W,Q,P)$.
\end{itemize}
In terms of the notions of Subsection~\ref{subs:semi-fun-cplx}, we see, for instance, that $P\leq Q$ admits SP relative-coboundary solutions if and only if its relative cochain complex admits finite-complexity decompositions.

Now a special case of Lemma~\ref{lem:compar-diag} gives the following.

\begin{lem}\label{lem:compare-les}
Given a stable short exact sequence in $\SFMod(Z)$ as above, there is a commutative diagram relating the resulting long exact sequences,
\begin{scriptsize}
\begin{center}
$\phantom{i}$\xymatrix{
\cdots \ar[r] &\rmH^p_\sp(W,P\leq Q) \ar[r]\ar[d] & \rmH^p_\sp(W,P'\leq Q') \ar[r]\ar[d] & \rmH^p_\sp(W,P''\leq Q'') \ar[r]\ar[d] & \rmH^{p+1}_\sp(W,P\leq Q) \ar[r]\ar[d] & \cdots\\
\cdots \ar[r] & \rmH^p_\sp(W,Q/P) \ar[r] & \rmH^p_\m(W,Q'/P') \ar[r] & \rmH^p_\m(W,Q''/P'') \ar[r] & \rmH^{p+1}_\m(W,Q/P) \ar[r] & \cdots,}
\end{center}
\end{scriptsize}
in which the vertical homomorphisms are those given by~(\ref{eq:compar-cohom}). \qed
\end{lem}

Various key steps later will be proofs that particular semi-functional modules have SP relative-coboundary-solutions or SP-represented cohomology.  The seed from which all of those proofs will grow is the basic fact that if $A$ is a discrete $Z$-module, then the groups $\rmH^p_\m(Z,A)$ are all discrete (by Proposition I.3.3), and the comparison homomorphisms
\[\rmH^p_\sp(Z,0 \leq A) \to \rmH^p_\m(Z,A)\]
are isomorphisms.

In fact, we will need a slightly fiddly generalization of this result to allow for co-inductions and restrictions.  The quickest approach seems to be to prove the more general version directly.  To formulate it, let $Z$ be a compact metrizable Abelian group and $W,Y \leq Z$ two closed subgroups.  Let $A_0$ be an SP Lie $Y$-module, and let $Z_1 := Y+W$, let $A_1:= \Cnd_Y^{Z_1}A_0$ and let
\[A := \Cnd_Y^Z A_0 \cong \Cnd_{Z_1}^ZA_1\]
(using relation (9) in Part I).

Then Corollary I.3.4 gave that $\rmH^p_\m(W,\Cnd_Y^{Z_1}A)$ is an s.p. Lie $Z_1$-module for all $p\geq 0$, and Lemma I.3.5 gave that
\[\rmH^p_\m(W,\Cnd_Y^ZA) \cong \Cnd_{Z_1}^Z\rmH^p_\m(W,\Cnd_Y^{Z_1}A)\]
(an isomorphism of Polish $Z$-modules). This right-hand side is now a functional $Z$-module with base $\rm{id}_Z$ and fibre equal to $\rmH^p_\m(W,\Cnd_Y^{Z_1}A)$.  As such, it has step-polynomial elements of its own.

\begin{prop}\label{prop:disc-sp-and-m-cohom}
In the above setting, if $A_0$ is discrete, then the comparison homomorphism
\[\rmH^p_\sp(W,0\leq A) \to \rmH^p_\m(W,A) \cong \Cnd_{Z_1}^Z\rmH^p_\m(W,A_1)\]
is injective for every $p\geq 0$, and has image equal to
\[\big(\Cnd_{Z_1}^Z\rmH^p_\m(W,A_1)\big)_\sp.\]
\end{prop}

This is similar to~\cite[Proposition 94]{AusMoo--cohomcty}, but is rather more precise.  Its proof is also similar to various arguments from~\cite{AusMoo--cohomcty}, but requires a digression from the developments of the current section.  It is therefore deferred to Appendix~\ref{app:sp-and-m-cohom}.

From the case of discrete $A_0$ treated by Proposition~\ref{prop:disc-sp-and-m-cohom}, we can now easily prove the following enhancement.

\begin{prop}\label{prop:sp-and-m-cohom}
In the above setting, for any SP Lie $Y$-module $A_0$, the comparison homomorphism
\begin{eqnarray}\label{eq:sp-and-m-cohom}
\rmH^p_\sp(W,0\leq A) \to \rmH^p_\m(W,A) \cong \Cnd_{Z_1}^Z\rmH^p_\m(W,A_1)
\end{eqnarray}
is injective for every $p\geq 0$, and has image equal to
\[\big(\Cnd_{Z_1}^Z\rmH^p_\m(W,A_1)\big)_\sp.\]

In particular, $A$ has stable SP relative-coboundary-solutions over $W$.
\end{prop}

\begin{proof}
\emph{Step 1.}\quad Proposition~\ref{prop:disc-sp-and-m-cohom} already proves this in case $A_0$ is discrete.

\vspace{7pt}

\emph{Step 2.}\quad We next prove it for Euclidean $A_0$. When $p=0$, one obviously has
\[\rmH^0_\sp(W,A) = \big(\Z^0(W,\Cnd_{Z_1}^Z A_1)\big)_\sp = (\Cnd_{Z_1}^Z A_1^W)_\sp \subseteq A^W = \rmH^0_\m(W,A).\]

On the other hand, the Shapiro Isomorphism for the theory $\rmH^\ast_\m$ (see Theorem I.3.2) gives $\rmH^p_\m(W,A_1) \cong \rmH^p_\m(W\cap Y,A_0)$, and by~\cite[Theorem A]{AusMoo--cohomcty}, this equals $0$ for all $p\geq 1$, so the right-hand side of~(\ref{eq:sp-and-m-cohom}) is zero.  We therefore need to prove that also $\rmH^p_{\rm{sp}}(W,A) = 0$ for all $p \geq 1$.

Suppose that $p\geq 1$ and $f \in \Z^p_{\rm{sp}}(W,A)$, so $f$ is a step polynomial $W^p\times Z\to A_0$.  Since step polynomials are bounded, we may efface this cocycle as in the classical cohomology of finite groups (or as in the proof of~\cite[Theorem A]{AusMoo--cohomcty}): one has $f = dg$, where $g:W^{p-1}\times Z\to A_0$ is defined by
\[g(w_1,\ldots,w_{p-1},z) := (-1)^p\int_W f(w_1,\ldots,w_{p-1},w,z)\, \d w.\]
This $g$ is also a step polynomial, by Corollary~\ref{cor:int-step-poly}, so $\rmH^p_\sp(W,A) = 0$.

\vspace{7pt}

\emph{Step 3.}\quad Now consider a general SP module $A$, and consider its subgroups
\[T \leq A_0' \leq A_0,\]
where $A_0'$ is the identity component of $A_0$ and $T$ is the maximal compact subgroup of $A_0'$.  These are preserved by any topological automorphism of $A_0$, and standard structure theory for locally compact Abelian groups (see~\cite[Section II.9]{HewRos79}) gives that $A_0/A_0'$ is discrete, $A_0'/T$ is Euclidean, and $T$ is toral.  Also, $T$ is a quotient of its universal cover $\t{T}$, a Euclidean space, by a discrete subgroup, and any $Z$-action on $T$ lifts to a $Z$-action on $\t{T}$.

To finish the proof, it therefore suffices to prove that the desired conclusion is closed under forming quotients and extensions of SP Lie modules.  However, letting $A_0' \into A_0 \onto A_0''$ be a short exact sequence of SP Lie modules, this will follow by applying Proposition~\ref{prop:complex-ses-good} to the short exact sequence of relative chain complexes arising from the stable short exact sequence
\[0 \to (0 \leq A) \to (0 \leq A') \to (0 \leq A'') \to 0.\]
To see that all the conditions required for that proposition are satisfied here, we must check that that all three of the relevant cochain complexes are locally SP.  Firstly, the homology of the first complex in position $0$ is $\rmH^0_\m(W,A) = A^W$, which is a submodule of $A$ and hence still SP, and similarly for $A'$ and $A''$.  On the other hand, $\rmH^p_\m(W,A)$ is discrete, and hence SP, for all $p\geq 1$, by Corollary I.3.4, and similarly for $A'$ and $A''$.  Finally, all of the quotient homomorphisms
\[\B^p(W,A'')\leq \Z^p(W,A'') \onto \rmH_\m^p(W,A'')\]
are stably complexity-bounded: when $p=0$, this is an isomorphism, and for $p\geq 1$ the target $\rmH^p_\m(W,A'')$ is co-induced from a discrete module, so stable complexity-boundedness is given by part (1) of Proposition~\ref{prop:complexity-above-and-below}.
\end{proof}

Using the long exact sequence, Proposition~\ref{prop:sp-and-m-cohom} generalizes directly to semi-functional modules with SP quotients and stable SP representatives.

\begin{cor}\label{cor:sp-and-m-cohom}
Let $Y,W \leq Z$ be as above, and let $P_0 \leq Q_0$ be a semi-functional $Y$-module whose quotient $A_0 := Q_0/P_0$ is SP and which has stable SP representatives.  Let $(P \leq Q) := \Cnd_Y^Z(P_0 \leq Q_0)$ and $A := \Cnd_Y^Z A_0$.  Then the comparison homomorphisms
\[\rmH^p_\sp(W,P \leq Q) \to \rmH^p_\m(W,Q/P) \cong \Cnd_{Y+W}^Z\rmH^p_\m(W,\Cnd_Y^{Y+W}A_0)\]
are injective (so $P\leq Q$ has SP relative-coboundary-solutions), and have images equal to
\[\big(\Cnd_{Y+W}^Z\rmH^p_\m(W,\Cnd_Y^{Y+W}A_0)\big)_\sp\]
\end{cor}

Put another way, the second conclusion here implies that the semi-functional module
\[\B^p(W,\Cnd_Y^{Y+W}Q_0,\Cnd_Y^{Y+W}P_0) \leq \Z^p(W,\Cnd_Y^{Y+W}Q_0,\Cnd_Y^{Y+W}P_0)\]
has SP quotient module (isomorphic to $\rmH^p_\m(W,\Cnd_Y^{Y+W}A_0)$) and has stable SP pre-images.

\begin{proof}
According to Lemma~\ref{lem:stable-collapse}, the diagram
\begin{center}
$\phantom{i}$\xymatrix{
0 \ar[r] & 0 \ar[r]\ar[d] & P \ar[r]\ar[d] & 0 \ar[r]\ar[d] & 0\\
0 \ar[r] & 0 \ar[r] & Q \ar[r] & A \ar[r] & 0
}
\end{center}
is a stable short exact sequence in $\SFMod(Z)$.  Now consider the long exact sequences for $\rmH^\ast_\sp$ and $\rmH^\ast_\m$ that result from this stable short exact sequence, and the comparison diagram between them given by Lemma~\ref{lem:compare-les}.  Since the first column above is zero, this diagram disconnects into a sequence of commutative squares
\begin{center}
$\phantom{i}$\xymatrix{
\rmH^p_\sp(W,P\leq Q) \ar[r]\ar[d] & \rmH^p_\m(W,0\leq A) \ar[d]\\
 \rmH^p_\m(W,A) \ar@{=}[r] & \rmH^p_\m(W,A).
}
\end{center}
By exactness of the long exact sequence, the first horizontal arrow here is an isomorphism.  The second vertical arrow is injective and has image equal to
\[\big(\Cnd_{Y+W}^Z\rmH^p_\m\big(W,\Cnd_Y^{Y+W}A_0\big)\big)_\sp,\]
by Proposition~\ref{prop:sp-and-m-cohom}.  Therefore the same is true of the first vertical arrow.
\end{proof}

%CAN DESCRIBE ONE DIRECTION OF THE SHAPIRO ISOMORPHISM EASILY JUST AT THE LEVEL OF COCHAINS: TO DEFINE SUITABLE
%\[\Psi_p:\C^p(W,M)\to \C^p(Z,\Cnd_Z^W M)\]
%make a choice of a $W$-equivariant measurable map $\s:Z\to W$, which may be obtained as the map $\s(z) = z - s(z + W)$ for some measurable selector $s:Z/W\to Z$, and now set
%\begin{multline*}
%\Psi_p(f)(z_1,\ldots,z_p)(z)\\ = \s(-z)\cdot \big(f\big(\s(z_1-z) - \s(-z),\s(z_2+z_1-z) - \s(z_1 - z),\\ \ldots,\s(z_1 + \ldots + z_p-z)- \s(z_1 + \ldots + z_{p-1}-z)\big)\big).
%\end{multline*}
%One can now verify mechanically that $\Psi_p$ maps $\Z^p(W,M)$ into $\Z^p(Z,\Cnd_W^ZM)$, and that this implements the isomorphism constructed above.  Note that $\Psi_p(f)$ is well-defined, even though $f$ is only fixed up to a.e.-equality, because the map
%\begin{multline*}
%Z^{p+1}\to W^{p+1}:\\ (z_1,\ldots,z_p)(z) \mapsto \big(\s(z_1-z) - \s(-z), \ldots,\s(z_1 + \ldots + z_p-z)- \s(z_1 + \ldots + z_{p-1}-z)\big)
%\end{multline*}
%is easily seen to map Haar measure to Haar measure, and hence pull back negligible sets to negligible sets.  But for the reverse direction, you want to just restrict your cochains to the relevant subgroup, but you can have trouble in the setting of $\rmH^\ast_\m$, because it involves slices of a function on a set of measure zero -- hence need `universality' approach instead.

\subsection{Cohomology applied to complexes}

Now suppose that $Z_1$ is a compact Abelian group and $Y,W \leq Z_1$ are closed subgroups such that $Z_1 = Y+W$.  Suppose further that
\[0 = (P'_0 \leq Q'_0) \to (P'_1 \leq Q'_1) \to \cdots \to (P'_k \leq Q'_k) \to (P'_{k+1} \leq Q'_{k+1}) = 0\]
is a bounded semi-functional $Y$-complex, let $M'_\bullet$ be its quotient complex, and let
\begin{eqnarray}\label{eq:pre-init-cplx}
\big(P^\circ_\bullet \leq Q^\circ_\bullet \onto M^\circ_\bullet\big) := \Cnd_Y^{Z_1}\big(P'_\bullet \leq Q'_\bullet \onto M'_\bullet\big).
\end{eqnarray}
Let $\a'_\bullet$ be the homomorphisms of $Q'_\bullet$, $\a^{M'}_\bullet$ those of $M'_\bullet$, and similarly $\a^\circ_\bullet$ and $\a^{M^\circ}_\bullet$.

Now let $p\geq 0$, and assume a priori that $\rmH^p_\m(W,M^\circ_\ell)$ is Hausdorff for every $\ell$ (in the cases of interest this will follow from results of Part I).  Under this assumption, this subsection will study the resulting semi-functional complex
\begin{eqnarray}\label{eq:semi-fnl-cohom-cplx}
\B^p(W,Q^\circ_\bullet,P^\circ_\bullet) \leq \Z^p(W,Q^\circ_\bullet,P^\circ_\bullet),
\end{eqnarray}
whose quotient complex is equal to $\rmH^p_\m(W,M^\circ_\bullet)$.

To lighten notation, we abbreviate $\rmH^p_\sp(W,-) =: \rmH^p_\sp(-)$, $\rmH^p_\m(W,-) =: \rmH_\m^p(-)$, $\Z^p(W,-,-) =: \Z^p(-,-)$, $\B^p(W,-,-) =: \B^p(-,-)$ and $\C^p(W,-) =: \C^p(-)$ for the rest of this subsection.

The next proposition is an analog of Propositions I.3.6 and I.3.7 for the present setting.  Its proof will require several steps, involving many different ideas from earlier in the paper.

\begin{prop}\label{prop:Hp-of-cplx}
Assume that the $Y$-complex $M'_\bullet$ has locally SP and $\ell_0$-almost discrete homology, that $P'_\bullet \leq Q'_\bullet$ has stably fully SP homology, and that the following hold for all $i \in \{\ell_0,\dots,k-1\}$:
\begin{itemize}
\item[a)$_i$] the semi-functional module $P^\circ_i \leq Q^\circ_i$ admits stable SP relative-coboundary-solutions over $W$ for all $p\geq 1$;
\item[b)$_i$] the semi-functional complex~(\ref{eq:semi-fnl-cohom-cplx}) has stable SP representatives at position $i$ for all $p\geq 1$;
\item[c)$_i$] the semi-functional complex~(\ref{eq:semi-fnl-cohom-cplx}) has stable finite-complexity decompositions at position $i + 1$ for all $p\geq 1$.
\end{itemize}
Then property (a)$_k$ also holds; property (b)$_k$ also holds; and the following hold for all $i \in \{\ell_0,\dots,k\}$:
\begin{itemize}
\item[d)$_i$] the semi-functional complex~(\ref{eq:semi-fnl-cohom-cplx}) has stable SP representatives at position $i$ when $p=0$;
\item[e)$_i$] the semi-functional complex~(\ref{eq:semi-fnl-cohom-cplx}) has stable finite-complexity decompositions at position $i$ when $p=0$.
\end{itemize}
\end{prop}

By re-numbering if necessary, we may of course assume that $\ell_0 = 1$.

Proposition~\ref{prop:Hp-of-cplx} will be proved by induction on the length $k$ of these complexes.  The inductive step will involve the following construction: given the semi-functional $Y$-complex $P'_\bullet \leq Q'_\bullet$ as above, its \textbf{shortened version} is the length-$(k-1)$ semi-functional $Y$-complex
\begin{center}
$\phantom{i}$\xymatrix{
0 \ar[r] & P'_1 \ar[r]\ar^\incln[d] & \cdots \ar[r] & P'_{k-2} \ar[r]\ar^\incln[d] & P'_{k-1} \ar[r]\ar^\incln[d] & 0\\
0 \ar[r] & Q'_1 \ar[r] & \cdots \ar[r] & Q'_{k-2} \ar[r] & R' \ar[r] & 0,
}
\end{center}
where $R' := (\a'_k)^{-1}(P_k)$.  The quotient of this complex is
\[0 \to M'_1 \to \cdots \to M'_{k-2} \to K' \to 0,\]
where $K' := \ker \a^{M'}_k$.

\begin{lem}\label{lem:pre-Hp-of-cplx}
This shortened version satisfies the analogs of all the assumptions of Proposition~\ref{prop:Hp-of-cplx} for $i\leq k-2$.
\end{lem}

\begin{proof}
This is immediate for the counterparts of assumptions (a)$_i$ for all  $i\leq k-2$, and for the counterparts of assumptions (b)$_i$ and (c)$_i$ for $i\leq k-3$.

Now let $Z \geq Z_1 = Y+W$ be any enlargement, and let $P_i$ be the result of co-inducing $P_i'$ to $Z$, and similarly for the all the other modules involved here.

The counterpart of assumption (b)$_{k-2}$ now follows by observing that
\begin{multline*}
\img\big(\rmH^p_\m(M_{k-3}) \to \rmH^p_\m(M_{k-2})\big) \leq \ker\big (\rmH^p_\m(M_{k-2}) \to \rmH^p_\m(K)\big)\\ \leq \ker\big(\rmH^p_\m(M_{k-2})\to \rmH^p_\m(M_{k-1})\big).
\end{multline*}

It remains to verify the counterpart of assumption (c)$_{k-2}$: that is, that the relevant homomorphism
\[\Z^p(Q_{k-2},P_{k-2}) \oplus \B^p(R,P_{k-1}) \to \Z^p(R,P_{k-1})\]
admits SP pre-images.  Thus, suppose that $\s \in \Z^p(Q_{k-2},P_{k-2})$ and $\b \in \B^p(R,P_{k-1})$ are such that $\tau = \a_{k-1}\s + \b$ is a step polynomial.  Then, by assumption (c)$_{k-2}$ for the original length-$k$ complex, we know that $\a_{k-1}\s = \a_{k-1}\s'$ modulo $\B^p(Q_{k-1},P_{k-1})$ for some $\s' \in \Z^p_\sp(Q_{k-2},P_{k-2})$.  We need the same conclusion, but modulo $\B^p(R,P_{k-1})$.

At this point, we may replace $\s$ by $\s - \s'$ without disrupting our desired conclusion, and hence assume that $\a_{k-1}\s$ is itself a $(Q_{k-1},P_{k-1})$-relative coboundary.  It follows that the class of $\a_{k-1}\s$ in $\rmH^p_\m(K)$ lies in
\[A:= \ker(\rmH^p_\m(K) \to \rmH_\m^p(M_{k-1})).\]

Now consider the diagram
\begin{center}
$\phantom{i}$\xymatrix{ & \Z^p(R,P_{k-1})\cap \B^p(Q_{k-1},P_{k-1}) \ar^\psi[d]\\
\Z^p(Q_{k-2},P_{k-2}) \cap \a_{k-1}^{-1}\big(\B^p(Q_{k-1},P_{k-1})\big) \ar^-\phi[r]& A,
}
\end{center}
where $\psi$ is the quotient by $\B^p(R,P_{k-1})$ and
\[\phi(\s) = \a_{k-1}\s + \B^p(Q_{k-1},P_{k-1}) \in \rmH^p_\m(K).\]
Our assumptions above give that $\phi(\s)$ is equal to $\psi(\tau)$ for the step polynomial $\tau$.

In case $Z = Z_1$, $A$ is discrete, by Claim 1 in the proof of Proposition I.3.6--7.  The general case $Z \geq Z_1$ is co-induced from that one.  Also in case $Z=Z_1$, $\phi$ factorizes as
\begin{multline*}
\Z^p(Q_{k-2},P_{k-2}) \cap \a_{k-1}^{-1}\big(\B^p(Q_{k-1},P_{k-1})\big)\\ \onto \frac{\ker(\rmH^p(M_{k-2})\to \rmH^p(M_{k-1}))}{\img(\rmH^p(M_{k-3}) \to \rmH^p(M_{k-2}))} \to A.
\end{multline*}
The middle module here is also discrete, and the first of these factorizing homomorphisms has SP representatives by assumption (b)$_{k-2}$, so their composition has SP pre-images.  Therefore in this case there is certainly some $\s'' \in \Z_\sp^p(Q_{k-2},P_{k-2}) \cap \a_{k-1}^{-1}(\B^p(Q_{k-1},P_{k-1}))$ such that $\phi(\s) = \phi(\s'')$, and hence $\a_{k-1}\s = \a_{k-1}\s''$ modulo $\B^p(R,P_{k-1})$.  However, that last conclusion now extends to the case of a general enlargement $Z \geq Z_1$, by Corollary~\ref{cor:complexity-above-and-below}.
\end{proof}

To complete the proof of Proposition~\ref{prop:Hp-of-cplx}, we will also call on a slight variant of a standard truncated version of the Five Lemma: compare~\cite[Proposition 2.72(ii)]{Rotman09}.  Its proof is a routine diagram-chase.

\begin{lem}\label{lem:four}
Let
\begin{center}
$\phantom{i}$\xymatrix{
A_1 \ar_{\g_1}[d]\ar^{\a_1}[r] & A_2 \ar_{\g_2}[d]\ar^{\a_2}[r] & A_3 \ar_{\g_3}[d]\ar^{\a_3}[r] & A_4 \ar_{\g_4}[d]\\
B_1 \ar_{\b_1}[r] & B_2 \ar_{\b_2}[r] & B_3 \ar_{\b_3}[r] & B_4
}
\end{center}
be a commutative diagram of Abelian groups in which both rows are exact, $\g_2$ and $\g_4$ are injective, and one has
\begin{eqnarray}\label{eq:four}
\img\,\g_2\cap \img\,\b_1 = \img(\b_1\circ \g_1).
\end{eqnarray}
Then $\g_3$ is injective. \qed
\end{lem}

\begin{proof}[Proof of Proposition~\ref{prop:Hp-of-cplx}]
This is proved by induction on the length $k$ of the complex~(\ref{eq:pre-init-cplx}).

Let $Z \geq Z_1$ be an arbitrary enlargement, and let
\begin{eqnarray}\label{eq:init-cplx}
\big(P_\bullet \leq Q_\bullet \onto M_\bullet\big) := \Cnd_{Y+W}^Z\big(P^\circ_\bullet \leq Q^\circ_\bullet \onto M^\circ_\bullet\big).
\end{eqnarray}
Let $\a_\bullet$ be the morphisms of the complex $Q_\bullet$, and $\a^M_\bullet$ the morphisms of the quotient complex $M_\bullet$.  They are co-induced from the morphisms of $Q^\circ_\bullet$ and $M^\circ_\bullet$, respectively.

\vspace{7pt}

\emph{Base clause: $k=1$.}\quad In this case, our semi-functional complex is simply $\Cnd_Y^Z(P'_1 \leq Q'_1)$, and the quotient $M'_1$ is an SP $Y$-module.  Of the desired conclusion, (e)$_1$ is trivial, and (a)$_1$, (b)$_1$ and (d)$_1$ are all given by Corollary~\ref{cor:sp-and-m-cohom}.

\vspace{7pt}

\emph{Recursion clause.}\quad Now assume that $k\geq 2$. Consider the shortened version of $P'_\bullet \leq Q'_\bullet$, as in Lemma~\ref{lem:pre-Hp-of-cplx}, and its co-inductions from $Y$ to $Z_1$ and to $Z$.  Let $R'$, $R^\circ$, $R$, $K'$, $K^\circ$ and $K$ be the final modules in the shortened versions and their quotients, as in the proof of Lemma~\ref{lem:pre-Hp-of-cplx}.

In light of that lemma, we may apply the inductive hypothesis to this shortened complex.  It immediately gives the desired conclusions (d)$_i$ and (e)$_i$ for all $i\leq k-1$, and also the following:
\begin{itemize}
\item[i)] $(P_{k-1} \leq R)$ admits SP relative-coboundary-solutions over $W$ for all $p \geq 1$;
\item[ii)] the semi-functional module
\[\a_{k-1}(\Z^p(Q_{k-2},P_{k-2})) + \B^p(R,P_{k-1}) \leq \Z^p(R,P_{k-1})\]
has quotient $\coker(\rmH_\m^p(M_{k-2})\to \rmH_\m^p(K))$ that is co-induced from a discrete $Z_1$-module (as in the proof of Propositions I.3.6--7), and has SP representatives, for all $p\geq 1$.
\end{itemize}

It remains to prove (a)$_k$, (b)$_k$, (d)$_k$ and (e)$_k$.  The remainder of the proof will be presented as a sequence of claims which together cover all of these conclusions.

\vspace{7pt}

\noindent\emph{Claim 1.}\quad Let $I := \img \,\a^M_k$, and let $S := \a_k(Q_{k-1}) + P_k$.  The semi-functional $Z$-module $P_k \leq S$ admits SP relative-coboundary-solutions over $W$ in every degree $p\geq 1$.

\vspace{7pt}

\noindent\emph{Proof of claim.}\quad The short sequence
\[0 \to (P_{k-1} \leq R) \to (P_{k-1} \leq Q_{k-1}) \to (P_k \leq S) \to 0\]
is stably exact in $\SFMod(Z)$: this follows from our initial assumption of stable finite-complexity decompositions for the complex~(\ref{eq:pre-init-cplx}).  Consider the following portion of the comparison between the resulting long exact sequences for $\rmH^\ast_\sp$ and $\rmH^\ast_\m$ (from Lemma~\ref{lem:compare-les}):
\begin{small}
\begin{center}
$\phantom{i}$\xymatrix{
\rmH^p_\sp(P_{k-1}\leq R) \ar[r]\ar[d] & \rmH^p_\sp(P_{k-1}\leq Q_{k-1}) \ar[r]\ar[d] & \rmH^p_\sp(P_k\leq S) \ar[r]\ar[d] & \rmH^{p+1}_\sp(P_{k-1}\leq R) \ar[d]\\
\rmH^p_\m(K) \ar[r] & \rmH^p_\m(M_{k-1}) \ar[r] & \rmH^p_\m(I) \ar[r] & \rmH^{p+1}_\m(K).}
\end{center}
\end{small}
We will apply Lemma~\ref{lem:four} to show that the third vertical arrow here is injective, as required.

Inductive conclusion (i) gives that $(P_{k-1} \leq R)$ admits SP relative-coboundary-solutions, and assumption (a)$_{k-1}$ gives this for $(P_{k-1} \leq Q_{k-1})$. It remains to check the hypothesis~(\ref{eq:four}) in this setting.

To do this, next let
\[A := \frac{\ker(\rmH^p_\m(M_{k-1})\to \rmH^p_\m(M_k))}{\img(\rmH^p_\m(M_{k-2})\to \rmH^p_\m(M_{k-1}))},\]
and consider the diagram
\begin{center}
$\phantom{i}$\xymatrix{
& \Z^p(Q_{k-1},P_{k-1})\cap \a_k^{-1}(\B^p(S,P_k)) \ar^\psi[d]\\
\Z^p(R,P_{k-1}) \ar^-\phi[r] & A,
}
\end{center}
where
\[\phi(\s) = \s + \big(\B^p(Q_{k-1},P_{k-1}) + \a_{k-1}\big(\Z^p(Q_{k-2},P_{k-2})\big)\big),\]
and similarly for $\psi(\tau)$.  The image $\phi(\s)$ always lies in $A$ because ${\img(\rmH^p_\m(K) \to \rmH^p_\m(M_{k-1}))}$ is contained in $\ker(\rmH^p_\m(M_{k-1}) \to \rmH^p_\m(M_k))$, since $K = \ker \a_k^M$.

Now observe that $\phi$ has a natural factorization
\[\Z^p(R,P_{k-1})\to \coker\big(\rmH^p_\m(M_{k-2}) \to \rmH^p_\m(K)\big)\to A.\]
In case $Z = Z_1$, inductive conclusion (ii) gives that the second module here is discrete and that the first homomorphism has SP representatives.  This implies that $\phi$ has SP pre-images when $Z = Z_1$.  Thus, if $\s \in \Z^p(R,P_{k-1})$ and $\tau \in \Z^p_\sp(Q_{k-1},P_{k-1})$ are such that they represent the same cohomology class in $\rmH^p_\m(M_{k-1})$, then in particular they have the same image in $A$, and since $\phi$ has SP pre-images there is some $\s' \in \Z^p_\sp(R,P_{k-1})$ which also has that image in $A$.

Having proved this conclusion when $Z = Z_1$, Corollary~\ref{cor:complexity-above-and-below} extends it to an arbitrary enlargement $Z \geq Z_1$.  We therefore obtain
\[\tau = \s = \s' + \a_{k-1}\k \quad \mod \B^p(Q_{k-1},P_{k-1})\]
for some $\k \in \Z^p(Q_{k-2},P_{k-2})$.  Finally, this asserts that $\a_{k-1}\k$ agrees with the step polynomial $\tau - \s'$ modulo $\B^p(Q_{k-1},P_{k-1})$.  By assumption (c)$_{k-2}$, this gives some $\k' \in \Z^p_{\rm{sp}}(Q_{k-2},P_{k-2})$ such that $\a_{k-1}\k = \a_{k-1}\k'$ modulo $\B^p(Q_{k-1},P_{k-1})$.  Hence
\[\s = \s' + \a_{k-1}\k' \quad \mod \B^p(Q_{k-1},P_{k-1}),\]
where the right-hand side is the required element of $\Z^p_\sp(R,P_{k-1})$. \qed$\phantom{i}_{\rm{Claim}}$

\vspace{7pt}

\noindent\emph{Claim 2.}\quad The semi-functional $Z$-module $P_k \leq Q_k$ admits SP relative-coboundary-solutions over $W$ in every degree $p\geq 1$. (This is conclusion (a)$_k$.)

\vspace{7pt}

\noindent\emph{Proof of claim.}\quad The short exact sequence
\[0\to (P_k \leq S) \to (P_k \leq Q_k) \to (S \leq Q_k) \to 0\]
is stable, since it arises from a concatenation (see Example~\ref{ex:concat}).  Consider the following portion of the resulting comparison between the long exact sequences for $\rmH^\ast_\sp$ and $\rmH^\ast_\m$:
\begin{small}
\begin{center}
$\phantom{i}$\xymatrix{
\rmH^{p-1}_\sp(S \leq Q_k) \ar[r]\ar[d] & \rmH^p_\sp(P_k \leq S) \ar[r]\ar[d] & \rmH^p_\sp(P_k\leq Q_k) \ar[r]\ar[d] & \rmH^p_\sp(S\leq Q_k) \ar[d]\\
\rmH^{p-1}_\m(M_k/I) \ar[r] & \rmH^p_\m(I) \ar[r] & \rmH^p_\m(M_k) \ar[r] & \rmH^p_\m(M_k/I).}
\end{center}
\end{small}
The result will follow by applying Lemma~\ref{lem:four} to this picture.

Claim 1 has given that $(P_k \leq S)$ has SP relative-coboundary-solutions in all degrees, and this holds for $(S \leq Q_k)$ by Corollary~\ref{cor:sp-and-m-cohom} (since $M'_k/I'$ is a SP $Y$-module and $P'_\bullet \leq Q'_\bullet$ has stable SP representatives, by assumption).

It remains to check the hypothesis~(\ref{eq:four}).  To do this, let
\[B := \ker(\rmH^p_\m(I) \to \rmH^p_\m(M_k)),\]
and consider the diagram
\begin{center}
$\phantom{i}$\xymatrix{
& \Z^p(S,P_k) \cap \B^p(Q_k,P_k) \ar^\psi[d]\\
\Z^{p-1}(Q_k,S) \ar^-\phi[r] & B,
}
\end{center}
where $\psi$ is the quotient by $\B^p(S,P_k)$ and $\phi$ is the composition of the homomorphisms
\[\Z^{p-1}(Q_k,S) \ \stackrel{\rm{quotient}}{\to} \ \rmH^{p-1}_\m(M_k/I) \ \stackrel{\rm{switchback}}{\to} \ \rmH^p_\m(I).\]

Observe that $\phi$ is surjective, by the exactness of the bottom row of the comparison diagram above.  We need to show that if $\s \in \Z^p_\sp(S,P_k)\cap \B^p(Q_k,P_k)$, then $\psi(\s)$ is equal to $\phi(\tau)$ for some $\tau \in \Z^{p-1}_\sp(Q_k,S)$.

However, in the special case $Z = Z_1$, Corollary~\ref{cor:sp-and-m-cohom} gives that the comparison map
\[\rmH^{p-1}_\sp(S\leq Q_k)\to \rmH^{p-1}_\m(Q_k/S) = \rmH^{p-1}_\m(M_k/I)\]
is an isomorphism, so in fact $\phi(Z^{p-1}(Q_k,S)) = \phi(\Z^{p-1}_\sp(Q_k,S))$, from which the desired conclusion follows trivially.  Also, in this special case, $B$ is discrete, since it is equal to the image under the switchback homomorphism of $\rmH^p_\m(M_k/I)$, which is discrete in that case (since $M_k/I$ is co-induced from a discrete $Y$-module: see Corollary I.3.4).

We may therefore apply Corollary~\ref{cor:complexity-above-and-below} to the above diagram, and deduce that the desired $\tau \in \Z^{p-1}_\sp(Q_k,S)$ exists for any enlargement $Z \geq Z_1$ as well. \qed$\phantom{i}_{\rm{Claim}}$

\vspace{7pt}

\noindent\emph{Claim 3.}\quad The quotient homomorphism
\[\Z^p(S,P_k) \onto \coker\big(\rmH_\m^p(M_{k-1})\to \rmH^p_\m(I)\big)\]
has target that is co-induced from a discrete $Z_1$-module, and has stable SP representatives for all $p\geq 0$.

\vspace{7pt}

\noindent\emph{Proof of claim.}\quad The co-induced-of-discrete structure follows from Claim 2 in the proof of Proposition I.3.6--7.  Given this, by part (2) of Proposition~\ref{prop:complexity-above-and-below}, stable SP representatives will follow in general if we prove it in the special case $Z = Z_1$, so we now make that assumption.

Consider again the stable short exact sequence
\[0 \to (P_{k-1} \leq R) \to (P_{k-1} \leq Q_{k-1}) \to (P_k \leq S)\to 0,\]
as in the proof of Claim 1, and the following portion of the resulting comparison diagram:
\begin{small}
\begin{center}
$\phantom{i}$\xymatrix{\rmH^p_\sp(P_{k-1} \leq Q_{k-1}) \ar[r]\ar[d] & \rmH^p_\sp(P_k \leq S) \ar[r]\ar[d] & \rmH^{p+1}_\sp(P_{k-1} \leq R) \ar[r]\ar[d] & \rmH^{p+1}_\sp(P_{k-1} \leq Q_{k-1}) \ar[d] \\
\rmH^p_\m(M_{k-1}) \ar[r] & \rmH^p_\m(I) \ar[r] & \rmH^{p+1}_\m(K) \ar[r] & \rmH^{p+1}_\m(M_{k-1}).
}
\end{center}
\end{small}
Introducing cokernels and kernels, this collapses to
\begin{small}
\begin{center}
$\phantom{i}$\xymatrix{\coker\big(\rmH^p_\sp(P_{k-1} \leq Q_{k-1}) \to \rmH^p_\sp(P_k \leq S)\big) \ar^-\cong[r]\ar[d] & \ker\big(\rmH^{p+1}_\sp(P_{k-1} \leq R) \to \rmH^{p+1}_\sp(P_{k-1} \leq Q_{k-1})\big) \ar[d] \\
\coker\big(\rmH^p_\m(M_{k-1}) \to \rmH^p_\m(I)\big) \ar^-\cong[r] & \ker\big(\rmH^{p+1}_\m(K) \to \rmH^{p+1}_\m(M_{k-1})\big).
}
\end{center}
\end{small}

Our present task to prove that the first column here is surjective, so we may instead prove surjectivity for the second column.

Thus, suppose that $\s \in \Z^{p+1}(R,P_{k-1})\cap \B^{p+1}(Q_{k-1},P_{k-1})$.  Then the inductive conclusion (ii) gives
\[\s = \tau + \a_{k-1}\g \quad \mod \B^{p+1}(R,P_{k-1})\]
for some $\tau \in \Z_{\rm{sp}}^{p+1}(R,P_{k-1})$, $\g \in \Z^{p+1}(Q_{k-2},P_{k-2})$.

Re-arranging, this implies that
\[\a_{k-1}\g = -\tau \quad \mod \B^{p+1}(Q_{k-1},P_{k-1}),\]
where the right-hand side is a step polynomial.  Therefore property (c)$_{k-2}$ gives some $\g_1 \in \Z^p_\sp(Q_{k-2},P_{k-2})$ such that
\[\a_{k-1}\g = \a_{k-1}\g_1 \quad \mod \B^{p+1}(Q_{k-1},P_{k-1}).\]
This, in turn, implies that $\g - \g_1$ represents a class in
\[\ker\big(\rmH^{p+1}_\m(M_{k-2})\to \rmH^{p+1}_\m(M_{k-1})\big).\]
Since we are currently assuming $Z = Y+W$, this kernel is co-discrete over
\[\img\big(\rmH^{p+1}_\m(M_{k-3})\to \rmH^{p+1}_\m(M_{k-2})\big),\]
and property (b)$_{k-2}$ gives SP representatives for that kernel over this image.  Therefore we may write
\[\g = \g_1 + \g_2 \quad \mod\ \big(\a_{k-2}(\Z^{p+1}(Q_{k-3},P_{k-3})) + \B^{p+1}(Q_{k-2},P_{k-2})\big)\]
for some $\g_2 \in \Z^{p+1}_\sp(Q_{k-2},P_{k-2})$.

Applying $\a_{k-1}$ and substituting back into our expression for $\s$, we find that
\[\s = \tau + \a_{k-1}(\g_1 + \g_2) \quad \mod \B^{p+1}(R,P_{k-1}).\]

Letting $\s_1 := \tau + \a_{k-1}(\g_1 + \g_2)$, it follows that $\s_1$ is an element of $\Z^{p+1}_\sp(R,P_{k-1}) \cap \B^{p+1}(Q_{k-1},P_{k-1})$.  Since the assumed property (a)$_{k-1}$ implies that
\[\rmH^{p+1}_\sp(P_{k-1}\leq Q_{k-1}) \to \rmH^{p+1}_\m(M_{k-1})\]
is injective, it follows that in fact $\s_1$ lies in $\Z^{p+1}_\sp(R,P_{k-1}) \cap \B_\sp^{p+1}(Q_{k-1},P_{k-1})$, and hence represents a class in
\[\ker\big(\rmH^{p+1}_\sp(P_{k-1}\leq R) \to \rmH^{p+1}_\sp(P_{k-1} \leq Q_{k-1})\big).\]
\qed$\phantom{i}_{\rm{Claim}}$

\vspace{7pt}

\noindent\emph{Claim 4.}\quad The quotient homomorphism
\[\Z^p(Q_k,P_k) \onto \coker\big(\rmH^p_\m(I)\to \rmH^p_\m(M_k)\big)\]
has target that is co-induced from a discrete $Z_1$-module, and has stable SP representatives for all $p\geq 0$.

\vspace{7pt}

\noindent\emph{Proof of claim.}\quad Since $k\geq 2$, the co-induced-of-discrete structure of the cokernel here holds because the long exact sequence identifies it with a subgroup of $\rmH^p_\m(M_k/I)$ via a homomorphism that is co-induced over $Z_1$.

Given this co-induced-of-discrete structure, we may now assume that $Z = Y+W$, similarly to the proof of Claim 3.  In this case our assertion is that the given quotient homomorphism above is surjective.

Consider the stable short exact sequence
\[0 \to (P_k \leq S) \to (P_k \leq Q_k) \to (S \leq Q_k) \to 0\]
as in the proof of Claim 2, and the following portion of the resulting comparison diagram of long exact sequences:
\begin{center}
$\phantom{i}$\xymatrix{\rmH^p_\sp(P_k \leq S) \ar[r]\ar[d] & \rmH^p_\sp(P_k \leq Q_k) \ar[r]\ar[d] & \rmH^p_\sp(S \leq Q_k) \ar[r]\ar[d] & \rmH^{p+1}_\sp(P_k \leq S) \ar[d] \\
\rmH^p_\m(I) \ar[r] & \rmH^p_\m(M_k) \ar[r] & \rmH^p_\m(M_k/I) \ar[r] & \rmH^{p+1}_\m(I).
}
\end{center}
Introducing cokernels, we may collapse the left-hand square of this diagram as follows:
\begin{small}
\begin{center}
$\phantom{i}$\xymatrix{0 \ar[r] & \coker\big(\rmH^p_\sp(P_k\leq S) \to \rmH^p_\sp(P_k \leq Q_k)\big) \ar[r]\ar[d] & \rmH^p_\sp(S \leq Q_k) \ar[r]\ar[d] & \rmH^{p+1}_\sp(P_k \leq S) \ar[d] \\
0 \ar[r] & \coker\big(\rmH^p_\m(I) \to \rmH^p_\m(M_k)\big) \ar[r] & \rmH^p_\m(M_k/I) \ar[r] & \rmH^{p+1}_\m(I),
}
\end{center}
\end{small}
where both rows here are still exact.  In terms of this diagram, we need to show that the first vertical arrow here is surjective.  However, by a routine diagram chase (in fact, another truncated version of the Five Lemma: see~\cite[Proposition 2.72(i)]{Rotman09}), this follows because
\begin{itemize}
\item the map
\[\rmH^p_\sp(S \leq Q_k)\to \rmH^p_\m(M_k/I)\]
is an isomorphism, hence surjective (by Corollary~\ref{cor:sp-and-m-cohom}, since $Z = Y+W$ and $k\geq 2$ so $\rmH^p_\m(M_k/I)$ is discrete, and $S \leq Q_k\onto M_k/I$ has stable SP representatives by our assumption that $P'_\bullet \leq Q'_\bullet$ has stably fully SP homology), and
\item the map
\[\rmH^{p+1}_\sp(P_k \leq S) \to \rmH^{p+1}_\m(I)\]
is injective, because $(P_k \leq S)$ admits SP relative-coboundary-solutions (Claim 1). \qed$\phantom{i}_{\rm{Claim}}$
\end{itemize}

\vspace{7pt}

\noindent\emph{Claim 5.}\quad The quotient homomorphism
\[\Z^p(Q_k,P_k) \onto \coker\big(\rmH_\m^p(M_{k-1})\to \rmH_\m^p(M_k)\big)\]
has target that is co-induced from a discrete $Z_1$-module, and has stable SP representatives for all $p\geq 0$. (This gives properties (b)$_k$ and (d)$_k$.)

\vspace{7pt}

\noindent\emph{Proof of claim.}\quad Since $k\geq 2$, the co-induced-of-discrete structure of the cokernel again follows from Propositions I.3.6 and I.3.7.  As for Claim 3, we may therefore assume that $Z = Y+W$.

Let $D := \B^p(Q_k,P_k)$.  Clearly it suffices to find step-polynomial representatives in every coset for each of the two inclusions
\[\a_k\big(\Z^p(Q_{k-1},P_{k-1})\big) + D \leq \Z^p(S,P_k) + D \leq \Z^p(Q_k,P_k).\]
SP representatives for the first of these inclusions follow by Lemma~\ref{lem:step-reps-of-common-enlargement} and Claim 3, and for the second they follow by Claim 4.  \qed$\phantom{i}_{\rm{Claim}}$

\vspace{7pt}

\noindent\emph{Claim 6.}\quad Property (e)$_k$ holds.

\vspace{7pt}

\noindent\emph{Proof of claim.}\quad Suppose $q \in \Z^0(Q_{k-1},P_{k-1})$ is such that $\a_k(q) \in Q_{k,\sp} + P_k$.  Since the original complex itself has finite-complexity decompositions, there is some $q_1 \in Q_{k-1,\sp}$ such that $q - q_1 \in (\a_k)^{-1}(P_k) = R$.  Since $q$ is relatively $W$-invariant over $P_{k-1}$, it follows that
\[d^W q_1 = d^W(q_1 - q) \quad \mod \C^1(P_{k-1}),\]
so this is a $(R,P_{k-1})$-valued relative cocycle (by the right-hand side) which is a step polynomial (by the left-hand side).  By inductive conclusion (i), we may now find some $q_2 \in R$ such that $d^Wq_1 = d^W q_2$ modulo $\C^1(P_{k-1})$, so now $q_1 - q_2$ is a step-polynomial member of $\Z^0(Q_{k-1},P_{k-1})$ whose $\a_k$-image equals $\a_k (q)$. \qed$\phantom{i}_{\rm{Claim}}$

\vspace{7pt}

This completes the whole proof.
\end{proof}

\section{Semi-functional $\P$-modules}\label{sec:semi-fun-delta}

Starting in this section, we make free use of all the definitions from Section I.4.

\subsection{Definitions}

We first need the following simple generalization of Definition~\ref{dfn:fnl-mor}.  Recall the notion of derivation-actions from Definition I.4.2.

\begin{dfn}[Complexity-bounded derivation-actions]
Suppose that $P\leq \F(X,A)$ and $Q \leq \F(X',A')$ are functional $Z$-modules, that $U\leq Z$, and that
\[\t{\nabla}:U \to \rm{Hom}_Z(P,Q)\]
is a derivation-action.  Then $\t{\nabla}$ is \textbf{complexity-bounded} if the map
\[(u,x') \mapsto (\t{\nabla}_uf)(x')\]
is step-polynomial (hence, is an element of $\Z^1_\sp(U,Q)$) for every $f \in P_\sp$.

It is \textbf{stably complexity-bounded} if
\[\Cnd_Z^{Z'}\t{\nabla}:U \to \rm{Hom}_{Z'}(\Cnd_Z^{Z'}P,\Cnd_Z^{Z'}Q) : u \mapsto \Cnd_Z^{Z'}(\t{\nabla}_u)\]
is complexity bounded for every enlargement $Z' \geq Z$.
\end{dfn}

Note that, in principle, the above property is stronger than simply asserting that $\t{\nabla}_u$ is complexity-bounded for each $u$ separately.

We are now ready for the next major definition of this paper.

\begin{dfn}[Functional $\P$-modules]\label{dfn:fnl-delta-mod}
A $(Z,Y,\bfU)$-$\P$-module $(P_e)_e$ is \textbf{functional} if its sub-constituent modules are functional $(Y+U_e)$-modules, and its sub-constituent morphisms and sub-constituent derivation-actions are all stably complexity-bounded.

(Note that, as in Definition~\ref{dfn:fnl}, this implies a particular choice of base and fibre for all of those sub-constituent functional modules. They may all be different.)
\end{dfn}

Henceforth we will be concerned only with morphisms that are stably complexity-bounded, which is why this assumption is now built into the definition.

\begin{ex}
The principle examples from Part I, $\P$-modules of {\PDE}-solutions and zero-sum tuples, are clearly of this kind. \fin
\end{ex}

\begin{dfn}[Semi-functional $\P$-modules]\label{dfn:semi-fnl-delta-mod}
A \textbf{semi-functional $(Z,Y,\bfU)$-$\P$-module} is an inclusion $\scrP \leq \scrQ$ of functional $(Z,Y,\bfU)$-$\P$-modules, where for each $e$ the inclusion $P_e \leq Q_e$ is co-induced from an inclusion of the sub-constituent modules, and that latter inclusion is a semi-functional $(Y+U_e)$-module.

This semi-functional may also be denoted by $\scrP \stackrel{\incln}{\to} \scrQ$.
\end{dfn}

Alternatively, one may regard a semi-functional $\P$-module as a `$\P$-object' in the category $\SFMod(Z)$.

Often, we will view a semi-functional $\P$-module $\scrP \leq \scrQ$ as an enhanced way of looking at the quotient $\P$-module $\scrQ/\scrP$.  For this reason, given a property \textbf{P} of $\P$-modules, such as modesty or almost modesty, we write that a semi-functional $\P$-module $\scrP \leq \scrQ$ has \textbf{P} if this is true of $\scrQ/\scrP$, but not necessarily of $\scrP$ or $\scrQ$ separately.

Now suppose that $(\scrP \leq \scrQ)$ is a semi-functional $(Z,Y,\bfU)$-$\P$-module.  If $e \subseteq [k]$, then the structure complexes of $\scrP$ and $\scrQ$ at $e$ fit together into a semi-functional complex, which is co-induced over $Y + U_e$.

\begin{dfn}
We will refer to this semi-functional complex as the \textbf{structure complex} of $(\scrP \leq \scrQ)$ at $e$.
\end{dfn}

We will carry over many of our notational practices from Part I when working with structure complexes.  For example, we shall generally write $\partial_\ell$ or $\partial_{e,\ell}$ for the boundary homomorphisms of the structure complex (at $[k]$ or $e$, respectively), without notating the $\P$-module they belong to, since this will be clear from the context.

\subsection{Recap of some basic constructions}

Various auxiliary constructions and definitions for $\P$-modules occupied I.4.2: direct sums (Definition I.4.11), co-induction (Definition I.4.12), leanness (Definition I.4.13) and lean versions (Definition I.4.14). These mostly have obvious extensions to semi-functional $\P$-modules, and we will not describe them all again.  Simply note that one now makes sense of direct sums and co-induction using the definitions of these for individual semi-functional modules from Subsections~\ref{subs:semi-sp-mods} and~\ref{subs:behav-coind}.

Now suppose that $c \subseteq e$ is an inclusion of finite sets, that $Y\leq Z$ is an inclusion of compact Abelian groups, and that $\bfU = (U_i)_{i\in e}$ is a family of closed subgroups of $Z$.  Let $\bfU\uhr_c$ be the subfamily $(U_i)_{i\in c}$, and let $\scrP$ be a functional $(Z,Y,\bfU)$-$\P$-module.  Then the restriction $\scrP\uhr_c$ is obviously also functional.  Similarly, if $\scrP \leq \scrQ$ is a semi-functional $\P$-module, then so is $(\scrP\leq \scrQ)\uhr_c := (\scrP\uhr_c \leq \scrQ\uhr_c)$.

We next turn to aggrandizement, so now suppose instead that $\scrP = (P_e)_e$ is a functional ${(Z,Y,\bfU\uhr_c)}$-$\P$-module.  Then Definition I.5.1 gives immediately that $\rm{Ag}_c^e\scrP$ is a functional $(Z,Y,\bfU)$-$\P$-module with the same fibre and dummy.

Now let $\scrP \leq \scrQ$ be a semi-functional $(Z,Y,\bfU\uhr_c)$-$\P$-module.  Then both of $\scrP$ and $\scrQ$ have aggrandizements to $\bfU$, and these fit into a new semi-functional $\P$-module
\[\rm{Ag}_c^e\scrP \leq \rm{Ag}_c^e\scrQ.\]
We will sometimes abbreviate this to $\rm{Ag}_c^e(\scrP \leq \scrQ)$.  It is easy to check that the quotient of $\rm{Ag}_c^e(\scrP \leq \scrQ)$ is equal to $\rm{Ag}_c^e(\scrQ/\scrP)$.

One of the key results about aggrandizements was Corollary I.5.3, which related the structure complexes of $\rm{Ag}_c^{[k]}\scrM$ to those of $\scrM$ itself: if $e \subseteq c$, then the structure complexes at $e$ are the same; and if $e \not\subseteq c$, then the structure complex of $\rm{Ag}_c^{[k]}\scrM$ at $e$ is split.

We will need to re-use this splitting in the present paper, and in doing so we will need the following additional information.

\begin{lem}\label{lem:cplx-bdd-splitting}
Let $c$ and $\scrP \leq \scrQ$ be as above, and let $e \subseteq [k]$ with $e\not\subseteq c$.  The structure complex of $\rm{Ag}_c^{[k]}\scrQ$ at $e$ has a sequence of splitting homomorphisms that are stably complexity-bounded, and that restrict to the corresponding splitting homomorphisms for $\rm{Ag}_c^{[k]}\scrP$.
\end{lem}

\begin{proof}
This follows immediately from the explicit construction of splitting homomorphisms in the proof of Lemma I.5.2 (the `Homotopical Lemma').  Suppose for simplicity that $e = [k]$, and pick $s \in [k]\setminus c$.  Now for $\ell = 0,1,\ldots,k$ define the homomorphism
\[\xi_\ell:\bigoplus_{|b| = \ell+1}Q_{e\cap c} \to \bigoplus_{|a| = \ell}Q_{a\cap c}\]
by
\[\big(\xi_\ell((q_b)_{|b| = \ell+1})\big)_a := \left\{\begin{array}{ll}0 &\quad \hbox{if}\ s \in a\\
\sgn(a\cup s:a)q_{a \cup s} &\quad \hbox{if}\ s \not\in a.\end{array}\right.\]
These manifestly have the required properties.
\end{proof}

Restriction and aggrandizement were combined to define reduction: Definition I.5.4.  In the semi-functional context, the reduction of $\scrP \leq \scrQ$ at $c$ will be
\[(\scrP\leq \scrQ) \llcorner_c := \rm{Ag}_c^{[k]}((\scrP \leq \scrQ)\uhr_c),\]
which is again clearly semi-functional, and has quotient equal to $(\scrQ/\scrP) \llcorner_c$.

%\subsubsection*{[[ STILL NEED??? ]] Co-induction}
%
%Another simple property to study for $\P$-modules is co-induction.
%
%\begin{dfn}
%Suppose that $Z \leq Z'$ and that $\scrM$ is a $(Z,Y,\bfU)$-$\P$-module. Then the \textbf{co-induction} of $\scrM$ to $Z'$ is the $\P$-module $\Cnd_Z^{Z'}\scrM$ with modules, structure morphisms and derivation lifts all given by simply co-inducting the corresponding data for $\scrM$.  One checks immediately that $\Cnd_Z^{Z'}\scrM$ is a $(Z',Y,\bfU)$-$\P$-module [[ WRITE IT OUT MORE!? ]].
%\end{dfn}
%
%In case $Z' = W\times Z$ for some auxiliary compact Abelian group $W$, we may write $\Cnd_Z^{Z'}\scrM$ as $\C(W,\scrM)$, based on the obvious identification for the individual modules.
%
%\begin{lem}\label{lem:coind-preserves}
%If $\scrM$ is modest (resp. meek) then so is $\Cnd_Z^{Z'}\scrM$ WITH THE SAME PARAMS?
%\end{lem}
%
%\begin{proof}
%...
%\end{proof}

\subsection{Meekness}

We next introduce an enhancement of the notion of almost modesty that keeps track of SP representatives.  This definition will be given in two parts.  The first is only a slight modification of the definition of almost modesty itself (Definition I.4.19).

\begin{dfn}
Let $(\scrP \leq \scrQ)$ be a semi-functional $(Z,Y,\bfU)$-$\P$-module.

If $(Z,Y,\bfU)$ is lean, then $(\scrP \leq \scrQ)$ is \textbf{$\ell_0$-almost top-SP-modest} if the top structure complex of $\scrP \leq \scrQ$ both has $\ell_0$-almost discrete homology and is locally SP. (Note that the extra assumption of an SP homology group is nontrivial only in position in $\ell_0$, since $\ell_0$-almost discreteness implies that all other homology groups here are discrete: see Lemma~\ref{lem:almost-disc-and-locally-SP}.)

In general, $(\scrP \leq \scrQ)$ is \textbf{$\ell_0$-almost top-SP-modest} if its lean version has this property, and it is \textbf{$\ell_0$-almost SP-modest} if every restriction of it is $\ell_0$-almost SP modest.
\end{dfn}

It follows that strict modesty for $\scrQ/\scrP$ implies almost SP-modesty.

\begin{dfn}[Meek and almost meek $\P$-modules]\label{dfn:meek-delta-mod}
Suppose that $0 \leq \ell_0 \leq k$.  If $(\scrP \leq \scrQ)$ is a semi-functional $\P$-module, then it is \textbf{$\ell_0$-meek} (resp. \textbf{$\ell_0$-almost meek}) if
\begin{itemize}
\item $\scrQ/\scrP$ is $\ell_0$-modest (resp. $\ell_0$-almost SP-modest), and
\item all of its structure complexes have stably fully SP homology.
\end{itemize}
\end{dfn}

Almost meek $\P$-modules should be thought of as $\P$-modules which are `explicit' (in that they are semi-functional), and for which the structure complexes are `nearly exact' (in the sense of almost modesty) and have homology represented by step polynomials.

Lemma I.5.5 identified the structure complexes of $\scrM \llc_c$ in terms of those of $\scrM$.  In Corollary I.5.6 this implied that the properties of almost or strict modesty are inherited by this operation.  Exactly the same reasoning gives this for meekness.

\begin{lem}\label{cor:restriction-still-meek}
If $\scrP \leq \scrQ$ is a semi-functional $(Z,Y,\bfU)$-$\P$-module and $c\subseteq [k]$, then $(\scrP\leq \scrQ)\llcorner_c$ is $\ell_0$-almost (resp. strictly) meek if and only if $(\scrP \leq \scrQ)\uhr_c$ is $\ell_0$-almost (resp. strictly) meek, and both are implied if $\scrP \leq \scrQ$ itself is $\ell_0$-almost (resp. strictly) meek.
\end{lem}

\begin{proof}
Mostly this is clear from the identification of the structure complexes of the restricted and reduced $\P$-modules.  The only remaining observation one needs is that if $e\not\subseteq c$, then the fact that the structure complexes of $\scrP\llc_c$ and $\scrQ\llc_c$ at $e$ have a consistent sequence of stably complexity-bounded splitting homomorphisms (Lemma~\ref{lem:cplx-bdd-splitting}) implies the following:
\begin{itemize}
\item the property of SP representatives is vacuously satisfied for the homology in these structure complexes;
\item one can obtain bounded-complexity representatives in the structure complex of $(\scrP \leq \scrQ)\llc_c$ at $e$ simply by composition with those splitting homomorphisms.
\end{itemize}
\end{proof}

%\begin{lem}
%Suppose that $\Phi:\scrM\to \scrM'$ is a $\P$-isomorphism of $(Z,Y,\bfU)$-$\P$-modules, and that it has the presentation as in Definition~\ref{dfn:semi-fnl-pres-mor} which is complexity-bounded and has strong SP pre-images.  Then $\scrP \leq \scrQ \onto \scrM$ is $\ell_0$-almost meek if and only if $\scrP' \leq \scrQ'\onto \scrM'$ is $\ell_0$-almost meek.
%\end{lem}
%
%\begin{proof}
%This follows directly from Lemma~\ref{lem:cplx-semi-fnl-iso}.
%\end{proof}

\subsection{Semi-functional $\P$-morphisms and short exact sequences}\label{sec:delta-ses}

The next two definitions are inevitable.

\begin{dfn}[Semi-functional $\P$-morphism]
Let $(\scrP \leq \scrQ)$ and $(\scrP' \leq \scrQ')$ be semi-functional $(Z,Y,\bfU)$-$\P$-module.  Then a \textbf{$\P$-morphism} from $(\scrP \leq \scrQ)$ to $(\scrP' \leq \scrQ')$ is a commutative diagram
\begin{center}
$\phantom{i}$\xymatrix{
\scrP \ar[r]\ar_\incln[d] & \scrP' \ar^\incln[d] \\
\scrQ \ar[r] & \scrQ'
}
\end{center}
in which both horizontal arrows are $\P$-morphisms (Definition I.4.16).  This situation will usually be abbreviated to $(\scrP \leq \scrQ) \to (\scrP' \leq \scrQ')$.

The semi-functional $\P$-morphism is (resp. \textbf{stably}) \textbf{complexity-bounded} if this holds for the resulting semi-functional morphism of the constituent semi-functional modules for each fixed $e \subseteq [k]$.
\end{dfn}

\begin{dfn}[Short exact sequences]
A \textbf{short exact sequence} of semi-functional $(Z,Y,\bfU)$-$\P$-modules is a sequence of semi-functional $\P$-morphisms of the form
\[0 \to (\scrP \leq \scrQ) \to (\scrP' \leq \scrQ') \to (\scrP'' \leq \scrQ'') \to 0\]
such that at each $e \subseteq [k]$ the resulting diagram
\[0 \to (P_e \leq Q_e) \to(P'_e \leq Q'_e) \to (P''_e\leq Q''_e) \to 0\]
is a short exact sequence in $\SFMod(Z)$.

This short exact sequence is \textbf{stable} if each of the resulting short exact sequences in $\SFMod(Z)$ is stable.
\end{dfn}

Given a semi-functional short exact sequence as above, it clearly induces a short exact sequence of the resulting structure complexes of $(\scrP\leq \scrQ)$, $(\scrP' \leq \scrQ')$ and $(\scrP''\leq \scrQ'')$.

\begin{prop}\label{prop:snakerep}
Let
\[0 \to (\scrP \leq \scrQ) \to (\scrP' \leq \scrQ') \to (\scrP'' \leq \scrQ'')\to 0\]
be a stable short exact sequence of semi-functional $(Z,Y,\bfU)$-$\P$-modules, all of them almost SP-modest. Then the following implications hold.
\begin{itemize}
\item[(1)] If $(\scrP \leq \scrQ)$ and $(\scrP' \leq \scrQ')$ are $\ell_0$-almost (resp. strictly) meek, then so is $(\scrP'' \leq \scrQ'')$.
\item[(2)] If $(\scrP \leq \scrQ)$ and $(\scrP'' \leq \scrQ'')$ are $\ell_0$-almost (resp. strictly) meek, then so is $(\scrP' \leq \scrQ')$.
\item[(3)] If $(\scrP' \leq \scrQ')$ and $(\scrP'' \leq \scrQ'')$ are $\ell_0$-almost (resp. strictly) meek, and $M_e = 0$ whenever $|e| \leq \ell_0$, then $(\scrP \leq \scrQ)$ is $(\ell_0+1)$-almost (resp. strictly) meek.
\end{itemize}
\end{prop}

\begin{proof}
The analog of this for the property of almost modesty was Proposition I.4.25.  The present result now follows by applying Proposition~\ref{prop:complex-ses-good} to the structure complexes of these semi-functional $\P$-modules.
\end{proof}

\section{Cohomology $\P$-modules}

Let $\scrM$ be a Polish $(Z,Y,\bfU)$-$\P$-module, $W \leq Z$ and $p \geq 0$. Then one may form a new pre-$\P$-module $\rmH^p_\m(W,\scrM)$ over $(Z,\bfU)$ by applying the cohomology functor $\rmH^p_\m(W,-)$ to all the modules and morphisms of $\scrM$.  This construction was the subject of Section I.6.  In case $\rmH^p_\m(W,\scrM)$ is still Polish, it may be given the structure of a $(Z,Y+W,\bfU)$-$\P$-module (Lemma I.6.2).  In particular, if $\scrM$ is (almost) modest, then this Polish property obtains, and $\rmH^p_\m(W,\scrM)$ also inherits the property of being (almost) modest (Theorem I.6.7).

Now suppose that $\scrP \leq \scrQ$ is a semi-functional $(Z,Y,\bfU)$-$\P$-module, and let $\scrM$ be its quotient $\P$-module.   Assume that $\rmH^p_\m(W,\scrM)$ is Polish (this will be given by results from Part I in all the cases of interest later).  Then it can also be realized as the quotient of a semi-functional $\P$-module, suggested by Lemma~\ref{lem:rel-cocycs}:
\begin{eqnarray}\label{eq:consts-for-Hp}
\B^p(W,Q_e,P_e) \leq \Z^p(W,Q_e,P_e) \onto \rmH^p_\m(W,M_e) \quad \hbox{for each}\ e \subseteq [k].
\end{eqnarray}
Since we assume that each $\rmH^p_\m(W,M_e)$ is Hausdorff, part of Lemma~\ref{lem:rel-cocycs} gives that each $\B^p(W,Q_e,P_e)$ is closed.

It is easy to verify that the inclusions in~(\ref{eq:consts-for-Hp}) fit together into a new semi-functional $\P$-module. The structure modules and derivation-actions are inherited from $\scrQ$ as was explained before Definition I.6.1.

We will usually abbreviate the above semi-functional $(Z,Y+W,\bfU)$-$\P$-module whose quotient is $\rmH^p_\m(W,\scrM)$ to
\begin{eqnarray}\label{eq:cohom-presentation}
\rmH^p_\m(W,\scrP \leq \scrQ) := \big(\B^p(W,\scrQ,\scrP) \leq \Z^p(W,\scrQ,\scrP)\big).
\end{eqnarray}

\begin{thm}\label{thm:HpgoodSHC}
If $(\scrP \leq \scrQ)$ is $\ell_0$-almost meek, then $\rmH^p_\m(W,\scrP \leq \scrQ)$ is $\ell_0$-almost meek.  It is strictly meek in case either $(\scrP\leq \scrQ)$ is strictly meek or $p\geq 1$.
\end{thm}

The proof will require that we also establish the following, which is of interest in its own right.

\begin{thm}[Step-polynomial coboundary solutions]\label{thm:QP-prims}
Let $(\scrP \leq \scrQ)$ be an $\ell_0$-almost meek $(Z,Y,\bfU)$-$\P$-module, and let $W \leq Z$ be a closed subgroup.  Then the semi-functional module $(P_e \leq Q_e)$ admits stable SP relative-coboundary-solutions over $W$ for every $e\subseteq [k]$.
\end{thm}

This provides an extension of the first part of Corollary~\ref{cor:sp-and-m-cohom} to a larger class of semi-functional modules.

Theorems~\ref{thm:HpgoodSHC} and~\ref{thm:QP-prims} will be proved together by an induction on the data $(Z,Y,\bfU)$.

In this section, we will adopt essentially the same notation as in Section I.6.  In addition to those conventions, let
\[P^{(\ell)} := \bigoplus_{|e|=\ell} P_e \quad \hbox{and} \quad Q^{(\ell)} := \bigoplus_{|e|=\ell} Q_e,\]
giving presentations $P^{(\ell)}\into Q^{(\ell)}\onto M^{(\ell)}$.

\subsection{Complexity and pure semi-functional $\P$-modules}

Some of the more substantial proofs in Section I.6 were by induction on the subgroup-data $(Z,Y,\bfU)$ directing a $\P$-module of interest.  This induction was organized using a partial order on those data.  We will meet more proofs like this below, so we quickly recall that partial order now.

Given two tuples of subgroup-data $(Z,Y,\bfU = (U_i)_{i=1}^k)$ and $(Z',Y',\bfU' = (U'_i)_{i = 1}^{k'})$, we have $(Z,Y,\bfU) \prec (Z',Y',\bfU')$ if
\begin{itemize}
\item either $k < k'$,
\item or $k=k'$, but $|\{i\leq k\,|\ Y \geq U_i\}| > |\{i\leq k'\,|\ Y' \geq U'_i\}|$.
\end{itemize}
This is easily seen to be a well-ordering, although not a total ordering.

The minimal tuples for this ordering are those of the form $(Z,Y,\ast)$, where $\ast$ denotes the empty tuple.  However, in most of our proofs by $\prec$-induction, the downwards movement in the order $\prec$ stops before one reaches such a minimal tuple. Rather, one must argue directly, without the $\prec$-inductive hypothesis, for any tuple for which $Y \geq U_i$ for all $i$.  Such a tuple is called `pure', and a $(Z,Y,\bfU)$-$\P$-module is `pure' if $(Z,Y,\bfU)$ is pure.

Proposition I.6.6 showed that pure strictly modest $\P$-modules have quite a simple structure, which made those parts of the inductive proofs very simple.  The same happens for strictly meek semi-functional $\P$-modules, as we shall show now.

\begin{prop}\label{prop:pure-struct}
If $(Y,Y,\bfU)$ is pure (thus, $Z = Y$), and $(\scrP \leq \scrQ)$ is a strictly $\ell_0$-meek semi-functional $(Y,Y,\bfU)$-$\P$-module, then the semi-functional $Y$-module $(P_e \leq Q_e)$ has discrete quotient and stable SP representatives for every $e \subseteq [k]$.

If $(Z,Y,\bfU)$ is pure, then a strictly $\ell_0$-meek semi-functional $(Z,Y,\bfU)$-$\P$-module is co-induced from a strictly $\ell_0$-meek semi-functional $(Y,Y,\bfU)$-$\P$-module.
\end{prop}

\begin{proof}
Consider first a $(Y,Y,\bfU)$-$\P$-module $(\scrP \leq \scrQ)$.

If $|e| = \ell_0$, then all the asserted properties of $P_e \leq Q_e$ are directly contained in the definition of strict meekness.

The remaining cases are handled by induction on $|e|$.  Thus, suppose $|e| \geq \ell_0 + 1$.  The fact that $Q_e/P_e$ is discrete comes from Proposition I.6.6.  Therefore it suffices to prove SP representatives, from which the stable version follows by part (2) of Proposition~\ref{prop:complexity-above-and-below}.

So now suppose that $m = q + P_e\in Q_e/P_e$.  Then strict meekness gives that $q = q_1 + \partial_{|e|}(q_2) + p$ for some $q_1 \in (Q_e)_{\rm{sp}}$, $q_2 \in \bigoplus_{a \in \binom{e}{|e|-1}}Q_a$ and $p \in P_e$.  Also, the hypothesis of the induction on $|e|$ gives that we may choose $q_2$ to be a step polynomial modulo $\bigoplus_{a \in \binom{e}{|e|-1}}P_a$.  This represents $m$ by the step polynomial $q + \partial_{|e|}(q_2)$, as required.

For general $Z \geq Y$, the assertion for a pure semi-functional $(Z,Y,\bfU)$-$\P$-module is immediate from the definitions.
\end{proof}

\subsection{An auxiliary $\prec$-induction}

The following is an analog of Proposition I.6.8, and it will play a similar r\^ole in the ensuing proofs.

\begin{prop}\label{prop:internal-stuff}
Fix $(Z',Y',\bfU')$, and assume that Theorems~\ref{thm:HpgoodSHC} and~\ref{thm:QP-prims} are known for any strictly meek $(Z_1,Y_1,\bfU_1)$-$\P$-module for which $(Z_1,Y_1,\bfU_1) \prec (Z',Y',\bfU')$.  (This assumption may be dropped in case $(Z',Y',U')$ is pure.)

Let $(\scrR \leq \scrS)$ be a strictly modest semi-functional $(Z',Y',\bfU')$-$\P$-module all of whose nontrivial restrictions are strictly meek, and let $\scrN:= \scrS/\scrR$.  Let $\partial_\ell:S^{(\ell-1)}\to S^{(\ell)}$, $\ell = 1,\ldots,k$, be the boundary morphisms of the top structure complex of $\scrS$, and let $\partial^\scrN_\ell$ be their counterparts for $\scrN$.
\begin{enumerate}
\item[(1)] The homomorphism
\[S^{(\ell)} \oplus R^{(\ell+1)} \to S^{(\ell+1)}:(f,r)\mapsto \partial_{\ell+1}f + r\]
admits finite-complexity decompositions for every $\ell$.
\item[(2)] The homomorphism
\[\partial_{\ell+1}^{-1}(R^{(\ell+1)}) \to \frac{\partial_{\ell+1}^{-1}(R^{(\ell+1)})}{\partial_\ell(S^{(\ell-1)}) + R^{(\ell)}} \cong \frac{\ker \partial^\scrN_{\ell+1}}{\img\,\partial^\scrN_\ell}\]
has SP representatives for every $\ell \leq k-1$ (recall that the target module here is co-induced from a discrete $(Y'+U'_{[k]})$-module, by the assumption of strict modest).
\end{enumerate}
\end{prop}

\begin{rmks}
\emph{(1)}\quad Since $Z'$ may be arbitrarily larger than $Y' + U'_{[k]}$ in the above statement, it actually proves \emph{stable} SP representatives and finite-complexity decompositions.

\vspace{7pt}

\emph{(2)}\quad The proof will make essential use of the discreteness of the structural homology here, and so we cannot weaken `strictly meek' to `almost meek' in the assumptions.  This discreteness will provide one of the hypotheses for some appeals to Proposition~\ref{prop:complexity-above-and-below}. \fin
\end{rmks}

This proposition will be proved in several steps.

\begin{proof}[Proof of Proposition~\ref{prop:internal-stuff} in pure case]
This is our main application of Proposition~\ref{prop:pure-struct}.  Firstly, applying that proposition to each nontrivial restriction of $(\scrR \leq \scrS)$, we see that it is co-induced from a strictly modest semi-functional $(Y,Y,\bfU)$-$\P$-module, say $(\scrR^\circ \leq \scrS^\circ)$, all of whose restrictions are meek.

For part (1), suppose that $f \in S^{(\ell)}$ and $r \in R^{(\ell+1)}$ are such that $q := \partial_{\ell+1} f + r$ is a step polynomial. Let $\ol{f}$ be the image of $f$ in $N^{(\ell)}$. By Proposition~\ref{prop:pure-struct} and~\ref{prop:complexity-above-and-below}, the fact that $q$ is a step polynomial implies that its image
\[q + R^{(\ell+1)} = \partial^\scrN_{\ell+1}\ol{f} \in N^{(\ell+1)} = \Cnd_Y^Z N^{\circ (\ell+1)}\]
is a step-function element of $\F(Z,N^{\circ (\ell+1)})$.  The homomorphism $\partial^\scrN_{\ell+1}:N^{(\ell)}\to N^{(\ell+1)}$ is co-induced from a closed homomorphism of discrete $Y$-modules, so Corollary~\ref{cor:Lie-homos-s-p-reps} gives some $\ol{g} \in N_\sp^{(\ell)}$ such that $\partial^\scrN_{\ell+1}\ol{g} = \partial^\scrN_{\ell+1}\ol{f}$.

Now, $S^{\circ(\ell)}\onto N^{\circ (\ell)}$ is a direct sum of quotient maps that have stable SP pre-images, by Proposition~\ref{prop:pure-struct}. We may therefore lift $\ol{g}$ to some $g \in S^{(\ell)}_\sp$.  This satisfies $\partial_{\ell+1}(f-g) \in R^{(\ell+1)}$, hence $q = \partial_{\ell+1}g \mod R^{(\ell+1)}$, as required for part (1).

Finally, for part (2), it suffices to complete the proof under the additional assumption that $Y' + U'_{[k]} = Y' = Z'$, by part (2) of Proposition~\ref{prop:complexity-above-and-below}.  In that case $R^{(\ell)} \leq S^{(\ell)}$ itself is co-discrete and has SP representatives; since of course $R^{(\ell)} \leq \partial_{\ell+1}^{-1}(R^{(\ell+1)}) \leq S^{(\ell)}$, the same follows for $R^{(\ell)} \leq \partial_{\ell+1}^{-1}(R^{(\ell+1)})$.  This proves part (2).
\end{proof}

The remainder of the proof of Proposition~\ref{prop:internal-stuff} will be by $\prec$-induction.  The induction is enabled by the following relative of Lemma I.6.9.

\begin{lem}\label{lem:cobdry=image}
Let $(\scrR \leq \scrS)$ be a semi-functional $(Z',Y',\bfU')$-$\P$-module for any $(Z',Y',\bfU')$.  Let $\scrN$, $\partial_\ell$ and $\partial^\scrN_\ell$ be as in the statement of Proposition~\ref{prop:internal-stuff}, and let $i \in [k]$.

If $f \in S^{(\ell)}$ has the property that $\partial_{\ell+1}f \in S^{(\ell+1)}_\sp + R^{(\ell+1)}$, then
\[d^{U_i'}f \in \partial_\ell\big(\Z^1(U_i',S^{(\ell-1)},R^{(\ell-1)})\big) + \Z^1_\sp(U_i',S^{(\ell)},R^{(\ell)}) + \C^1(U_i',R^{(\ell)}).\]
\end{lem}

\begin{proof}
Let $(\phi_{a,e})_{a\subseteq e}$ be the structure morphisms of $\scrS$.  Abbreviate $\scrS^\llc := \scrS\llc_{[k]\setminus i}$, and let $S^{\llc(\ell)}$ and $\partial^\llc_\ell$, $0 \leq \ell \leq k$, be the modules and boundary morphisms of the top structure complex of $\scrS^\llc$.  Also, let $\t{\nabla}^{U_i'}$ denote $\t{\nabla}^{\scrS,e,e\setminus i}$ for any $e\subseteq [k]$, or any direct sum of these derivation-lifts over different $e$.

In this notation, we now have $\t{\nabla}^{U'_i}f \in \Z^1(U'_i,S^{\llc(\ell)})$, and
\[\partial_{\ell+1}^\llc \t{\nabla}^{U'_i}f = \t{\nabla}^{U_i'}\partial_{\ell+1} f \in \Z^1_\sp(U_i',S^{\llc(\ell+1)}) + \C^1(U_i',R^{\llc(\ell+1)}).\]
This is the point at which it was crucial that the derivation-actions be complexity-bounded.

By Lemma~\ref{lem:cplx-bdd-splitting}, the top structure complex of $\scrR^\llc \leq \scrS^\llc$ splits with stably complexity-bounded splitting homomorphisms.  Therefore, decomposing $\t{\nabla}^{U'_i}f$ using the splitting homomorphisms $S^{\llc(\ell+1)} \to S^{\llc(\ell)}$ and $S^{\llc(\ell)} \to S^{\llc(\ell-1)}$, one obtains
\begin{eqnarray}\label{eq:nablaU'}
\t{\nabla}^{U'_i}f = \partial^\llc_\ell\s' + \tau' + \g'
\end{eqnarray}
for some
\[\s' \in \Z^1(U'_i,S^{\llc(\ell-1)},R^{\llc(\ell-1)}), \quad  \tau' \in \Z^1_\sp(U_i',S^{\llc(\ell)},R^{\llc(\ell)}) \quad \hbox{and} \quad \g' \in \C^1(U_i',R^{\llc(\ell)}).\]

Now let
\[\phi := \bigoplus_{|e|=\ell-1}\phi_{e\setminus i,e}:S^{\llc(\ell-1)}\to S^{(\ell-1)},\]
and let $\s := \phi\circ \s'$, and similarly $\tau$ and $\g$.  Applying $\phi$ to equation~(\ref{eq:nablaU'}) gives
\begin{multline*}
d^{U'_i}f = \partial_\ell\s + \tau + \g\\
\in \partial_\ell\big(\Z^1(U_i',S^{(\ell-1)},R^{(\ell-1)})\big) + \Z^1_\sp(U_i',S^{(\ell)},R^{(\ell)}) + \C^1(U_i',R^{(\ell)}).
\end{multline*}
\end{proof}

\begin{proof}[Proof of Proposition~\ref{prop:internal-stuff}]
The proof is by $\prec$-induction.  The pure case has already been proved, so assume also that $i \in [k]$ is such that $U_i' \not\leq Y$.  For the non-pure case, the heavy lifting will be done by a sequence of auxiliary claims, which are also part of the $\prec$-induction

\vspace{7pt}

\noindent\emph{Claim 1.}\quad For any $p\geq 0$, the top structure complex of the semi-functional $\P$-module $\rmH^p_\m(U_i',\scrR \leq \scrS)$ admits stable finite-complexity decompositions, and admits stable SP representatives at all positions except the last.

\vspace{7pt}

\noindent\emph{Proof of claim.}\quad By the assumed previous cases of Theorems~\ref{thm:HpgoodSHC} and~\ref{thm:QP-prims}, all nontrivial restrictions of $\rmH^p_\m(U_i',\scrR \leq \scrS)$ are strictly meek.  Therefore $\rmH^p_\m(U_i',\scrR \leq \scrS)$ satisfies the assumptions of Proposition~\ref{prop:internal-stuff}.  Since $(Z',Y' + U_i',\bfU') \prec (Z',Y',\bfU')$, it is covered by a previous case of that proposition, of which parts (1) and (2) now give the desired conclusions. \qed$\phantom{i}_{\rm{Claim}}$

\vspace{7pt}

\noindent\emph{Claim 2.}\quad Suppose that $\ell < k$ and $(f,r) \in S^{(\ell)}\oplus R^{(\ell+1)}$ is such that $\partial_{\ell+1}f + r$ is a step polynomial.  Then there is some $f_0 \in S^{(\ell)}_\sp$ such that
\begin{multline*}
\partial_{\ell+1}f = \partial_{\ell+1}f_0 \quad \mod\Z^0(U_i',S^{(\ell+1)},R^{(\ell+1)})\\
\big(\hbox{i.e.,} \quad \partial_{\ell+1}(f - f_0) \in \Z^0(U_i',S^{(\ell+1)},R^{(\ell+1)})\big).
\end{multline*}

\vspace{7pt}

\noindent\emph{Proof of claim.}\quad Lemma~\ref{lem:cobdry=image} gives
\begin{eqnarray}\label{eq:fsigmatau}
d^{U_i'}f = \partial_\ell \s + \tau + r
\end{eqnarray}
for some $\s \in \Z^1(U_i',S^{(\ell-1)},R^{(\ell-1)})$, $\tau \in \Z_\sp^1(U_i',S^{(\ell)},R^{(\ell)})$ and $r \in \C^1(U_i',R^{(\ell)})$.

Now applying the finite-complexity decompositions given by Claim 1 with $p=1$, it follows that
\[\tau = -\partial_\ell\s_0 + d^{U_i'}f_0 - r_0,\]
where $\s_0 \in \Z^1_\sp(U_i',S^{(\ell-1)},R^{(\ell-1)})$ and
\[d^{U_i'}f_0 - r_0 \in (\B^1(U_i',S^{(\ell)},R^{(\ell)}))_\sp.\]
Since $\ell < k$, one of the assumed cases of Theorem~\ref{thm:QP-prims} gives that $R^{(\ell)} \leq S^{(\ell)}$ admits SP relative-coboundary-solutions, so we may in fact take
\[f_0 \in S^{(\ell)}_\sp \quad \hbox{and} \quad r_0 \in \C_\sp^1(U_i',R^{(\ell)})\]
above.

Now substituting the above decomposition of $\tau$ into~(\ref{eq:fsigmatau}) gives
\begin{eqnarray}\label{eq:d-f-minus-f0}
d^{U_i'}(f-f_0) = \partial_\ell(\s - \s_0) + (r - r_0),
\end{eqnarray}
and hence
\[d^{U_i'}\partial_{\ell+1}(f-f_0) \in \C^1(U_i',R^{(\ell+1)}) \quad \Longrightarrow \quad \partial_{\ell+1}(f - f_0) \in \Z^0(U_i',S^{(\ell+1)},R^{(\ell+1)}).\]
\qed$\phantom{i}_{\rm{Claim}}$

\vspace{7pt}

\noindent\emph{Claim 3.}\quad Let $\ell < k$, let $(f,r)$ be as in Claim 2, and assume in addition that $Y' + U'_{[k]} = Z'$.  Then there are some $g \in S^{(\ell-1)}$ and $f_1 \in S^{(\ell)}_\sp$ such that
\[f = \partial_\ell g + f_1 \quad \mod \Z^0(U_i',S^{(\ell)},R^{(\ell)}).\]

\vspace{7pt}

\noindent\emph{Proof of claim.}\quad Let $f_0$ be as produced for Claim 2.  The conclusion of Claim 3 is not disrupted if we replace $f$ by $f - f_0$, so we may assume that in fact $f_0 = 0$.  Similarly, we may lighten notation by replacing $\s$ by $\s - \s_0$ and $r$ by $r - r_0$.

Having done so, equation~(\ref{eq:d-f-minus-f0}) asserts that the class of $\s$ in ${\rmH^1_\m(U_i',R^{(\ell-1)}\leq S^{(\ell-1)})}$ actually lies in
\[\ker\Big(\rmH^1_\m(U_i',R^{(\ell-1)} \leq S^{(\ell-1)}) \ \stackrel{\rmH^1_\m(U_i',\partial_\ell)}{\to} \ \rmH^1_\m(U_i',R^{(\ell)} \leq S^{(\ell)})\Big).\]
Claim 1 with $p=1$ gives SP representatives at position $\ell-1$ in the top structure complex of $\rmH^1_\m(U_i',\scrR \leq \scrS)$.  Since we now assume $Y' + U'_{[k]} = Z'$, the above kernel contains the image of $\rmH^1_\m(U_i',R^{(\ell-2)} \leq S^{(\ell-2)})$ as a closed-and-open subgroup, by the strict modesty of $\rmH^1_\m(U_i',\scrN)$.  Therefore in fact all cosets of that subgroup in the above kernel contain SP representatives.  This yields a decomposition
\[\s = \s_1 + \partial_{\ell-1} \g + d^{U_i'}g + r_1,\]
where
\begin{itemize}
\item $\s_1 \in \Z^1_\sp(U_i',S^{(\ell-1)},R^{(\ell-1)})$,
\item $\g \in \Z^1(U_i',S^{(\ell-2)},R^{(\ell-2)})$,
\item $g \in S^{(\ell-1)}$, and
\item $r_1 \in \C^1(U_i',R^{(\ell-1)})$.
\end{itemize}

We may now replace $f$ with $f - \partial_\ell g$ without disrupting the desired conclusion, and so assume $g = 0$.  Having done so, substituting the above decomposition of $\s$ into the formula for $d^{U_i'}f$ gives
\begin{eqnarray*}
d^{U_i'}f &=& \partial_\ell\s_1 + (\partial_\ell r_1 + r)\\
&=& \partial_\ell\s_1 \mod \C^1(U_i',R^{(\ell)}).
\end{eqnarray*}
Thus, the step-polynomial $(S^{(\ell)},R^{(\ell)})$-relative cocycle $\partial_\ell\s_1$ is actually a relative coboundary.  Therefore, by an assumed case of Theorem~\ref{thm:QP-prims}, there is some $f_1 \in S^{(\ell)}_\sp$ such that
\begin{multline*}
\partial_\ell\s_1 = d^{U_i'}f_1 \mod \C^1(U_i',R^{(\ell)})\\ \Longrightarrow \quad d^{U_i'}(f - f_1) \ \hbox{is}\ R^{(\ell)}\hbox{-valued}\\
\Longrightarrow \quad f = f_1 \quad \mod \Z^0(U_i',S^{(\ell)},R^{(\ell)}).
\end{multline*}
\qed$\phantom{i}_{\rm{Claim}}$

\vspace{7pt}

\noindent\emph{Claim 4.}\quad Let $\ell < k$, let $(f,r)$ be as in Claim 3, but consider again general $Y'$, $\bfU'$ and $Z'$.  Then there is some $f_1 \in S_\sp^{(\ell)}$ such that
\[\partial_{\ell+1}f = \partial_{\ell+1} f_1 \quad \mod \big(\partial_{\ell+1}\big(\Z^0(U_i',S^{(\ell)},R^{(\ell)})\big) + R^{(\ell+1)}\big).\]

\vspace{7pt}

\noindent\emph{Proof of claim.}\quad Let
\[P := \partial_{\ell+1}(\Z^0(U_i',S^{(\ell)},R^{(\ell)})) + R^{(\ell+1)}.\]

First, Claim 2 gives some $f_0 \in S^{(\ell)}_\sp$ such that
\[\partial_{\ell+1}f = \partial_{\ell+1} f_0 \quad \mod \Z^0(U_i',S^{(\ell+1)},R^{(\ell+1)}).\]
Replacing $f$ by $f - f_0$, we may assume that $\partial_{\ell+1}f$ itself lies in $\Z^0(U_i',S^{(\ell+1)},R^{(\ell+1)})$.  Also, since it is a $\partial_{\ell+1}$-image, the function $h := \partial_{\ell+1}f + r$ actually lies in
\[Q := \Z^0(U_i',S^{(\ell+1)},R^{(\ell+1)})\cap \partial_{\ell+2}^{-1}(R^{(\ell+2)}).\]

Now consider the diagram
\begin{center}
$\phantom{i}$\xymatrix{
& Q \ar^\psi[d]\\
S^{(\ell)} \supseteq \partial_{\ell+1}^{-1}(Q) \ar^-\phi[r] & Q/P,
}
\end{center}
where $\psi$ is the quotient homomorphism and $\phi$ is the composition of $\partial_{\ell+1}$ with that quotient homomorphism.  Given the reductions already made above, we have elements $f \in \partial_{\ell+1}^{-1}(Q)$ and $h \in Q_\sp$ that have the same image in $Q/P$, and Claim 4 will be proved if we find some $f_1 \in (\partial_{\ell+1}^{-1}(Q))_\sp$ which also has that same image.

Now recall that for general $Y'$, $\bfU'$ and $Z'$, $(\scrR \leq\scrS)$ is the co-induction from $Y' + U'_{[k]}$ to $Z'$ of its lean version.  Therefore the same holds for the above diagram.  Let that lean version be $(\scrR_0 \leq \scrS_0)$, and let $P_0$ and $Q_0$ be the analogs of $P$ and $Q$ for the lean version.  Since we assume $(\scrR \leq \scrS)$ is strictly modest, the same holds for $(\scrR_0 \leq \scrS_0)$.  Since the quotient $Q_0/P_0$ is isomorphic to the structural homology of $\rmH^0_\m(U_i',\scrS_0/\scrR_0)$ at position $\ell$ in the top structure complex, Theorem I.6.7 now gives that $Q_0/P_0$ is discrete.  Therefore the desired conclusion about the above diagram will follow in the general case $Z' \geq Y' + U'_{[k]}$ if we prove it only in the lean case, by Corollary~\ref{cor:complexity-above-and-below}.

Finally, with that leanness assumption, suppose that $f \in \partial_{\ell+1}^{-1}(Q)$ is such that $\partial_{\ell+1}f$ agrees with $h \in Q_\sp$ modulo $P$.  Then we have
\[\partial_{\ell+1}f = h + \partial_{\ell+1}\g + r\]
for some $\g \in \Z^0(U_i',S^{(\ell)},R^{(\ell)})$ and $r \in R^{(\ell+1)}$, by the definition of $P$.  By applying Claim 3 to the pair $(f - \g,r)$ (as we may now that we assume $Z' = Y' + U'_{[k]}$), we obtain $g \in S^{(\ell-1)}$ and $f_1 \in S^{(\ell)}_\sp$ such that $f - \g = \partial_\ell g + f_1$ modulo $\Z^0(U_i',S^{(\ell)},R^{(\ell)})$, and hence
\begin{multline*}
\partial_{\ell+1} f = \partial_{\ell+1}f_1 + \partial_{\ell+1} \g = \partial_{\ell+1}f_1 \quad \mod \partial_{\ell+1}(\Z^0(U_i',S^{(\ell)},R^{(\ell)}))\\
\Longrightarrow \quad \partial_{\ell+1}f = \partial_{\ell+1}f_1 \quad \mod P.
\end{multline*}
This of course guarantees that $f_1 \in (\partial_{\ell+1}^{-1}(Q))_\sp$, so this completes the proof. \qed$\phantom{i}_{\rm{Claim}}$

\vspace{7pt}

\emph{Proof of part (1) of Proposition~\ref{prop:internal-stuff}.}\quad Suppose that $(f,r) \in S^{(\ell)}\oplus R^{(\ell+1)}$ is such that $g := \partial_{\ell+1}f + r$ is a step polynomial.  We must show that there are step-polynomials $(f',r')$ with the same image.

Claim 4 gives some $f_1 \in S^{(\ell)}_\sp$ such that
\[\partial_{\ell+1}(f - f_1) = \partial_{\ell+1}f_2 \quad \mod R^{(\ell+1)}\]
for some $f_2 \in \Z^0(U_i',S^{(\ell)},R^{(\ell)})$.  Since $f_1$ is a step polynomial, it therefore suffices to complete the proof with $f_2$ in place of $f$.  However, at this point, our task is completed by an inductive appeal to part (1) of Proposition~\ref{prop:internal-stuff} for $\rmH^0_\m(U_i',\scrR \leq \scrS)$ (the case $p=0$ in Claim 1).  This completes the proof of part (1).

\vspace{7pt}

\noindent\emph{Claim 5.}\quad If $Y' + U'_{[k]} = Z'$ and $f \in \partial_{\ell+1}^{-1}(R^{(\ell+1)})$, then there are $g \in S^{(\ell-1)}$ and $f_2 \in S^{(\ell)}_\sp \cap \partial_{\ell+1}^{-1}(R^{(\ell+1)})$ such that
\[f - \partial_\ell g - f_2 \in \Z^0(U_i',S^{(\ell)},R^{(\ell)}).\]

\vspace{7pt}

\noindent\emph{Proof of claim.}\quad Let $g \in S^{(\ell-1)}$ and $f_1 \in S^{(\ell)}_\sp$ be as provided by Claim 3.  Since $\partial_{\ell+1}\partial_\ell = 0$, we still have $f - \partial_\ell g\in \partial_{\ell+1}^{-1}(R^{(\ell+1})$.  We may therefore replace $f$ with $f - \partial_\ell g$ without disrupting the desired conclusion, and so simply assume that $f - f_1 \in \Z^0(U_i',S^{(\ell)},R^{(\ell)})$.

Having assumed this, we will show that $f_1$ can be modified to some $f_2\in S^{(\ell)}_\sp \cap \partial_{\ell+1}^{-1}(R^{(\ell+1)})$ which still satisfies $f - f_2 \in \Z^0(U_i',S^{(\ell)},R^{(\ell)})$.

Since $\partial_{\ell+1} f \in R^{(\ell+1)}$, we have
\[\partial_{\ell+1} (f - f_1) = \partial_{\ell+1} (-f_1)  \mod R^{(\ell+1)}\]
Since $\partial_{\ell+1}f_1$ is a step polynomial, this implies that
\[\partial_{\ell+1} (f - f_1) \in \Z_\sp^0(U_i',S^{(\ell+1)},R^{(\ell+1)}) + R^{(\ell+1)}.\]
Claim 1 with $p=0$ gives finite-complexity decompositions in the top structure complex of $\rmH^0_\m(U_i',\scrR \leq \scrS)$, so it follows that there is some $f_1' \in \Z_\sp^0(U_i',S^{(\ell)},R^{(\ell)})$ such that
\[\partial_{\ell+1} (f-f_1) = \partial_{\ell+1} f_1' \mod R^{(\ell+1)}.\]
Now $f_2 := f_1 + f_1' \in S^{(\ell)}_\sp$ satisfies
\[\partial_{\ell+1}f_2 = \partial_{\ell+1}f = 0 \mod R^{(\ell+1)},\]
and of course we still have $f - f_2 \in \Z^0(U_i',S^{(\ell)},R^{(\ell)})$, so this $f_2$ satisfies our requirements.  \qed$\phantom{i}_{\rm{Claim}}$

\vspace{7pt}

\emph{Proof of part (2) of Proposition~\ref{prop:internal-stuff}.}\quad The target quotient module here is
\[\frac{\partial_{\ell+1}^{-1}(R^{(\ell+1)})}{\partial_\ell(S^{(\ell-1)}) + R^{(\ell)}} \stackrel{\cong}{\to} \frac{\ker \partial^\scrN_{\ell+1}}{\img \partial^\scrN_\ell}.\]
Since $(\scrR \leq \scrS)$ is strictly modest, all modules here are co-induced over $Y' + U'_{[k]}$, and this quotient in particular is co-induced from a discrete $(Y' + U'_{[k]})$-module.  Therefore, by part (2) of Proposition~\ref{prop:complexity-above-and-below}, stable SP representatives will follow if we prove only SP representatives in the case of lean data: $Y' + U'_{[k]} = Z'$.  Assume this for the rest of the proof.

Now suppose that $f \in \partial_{\ell+1}^{-1}(R^{(\ell+1)})$.  We must decompose $f$ into a step-polynomial and a member of $\partial_\ell(S^{(\ell-1)}) + R^{(\ell)}$.  By Claim 5, there are some $g \in S^{(\ell-1)}$ and $f_2 \in S_\rm{sp}^{(\ell)}\cap \partial_{\ell+1}^{-1}(R^{(\ell+1)})$ such that
\[f - \partial_\ell g - f_2 \in \Z^0(U_i',S^{(\ell)},R^{(\ell)}).\]
Let $f_3 := f - \partial_\ell g - f_2$. Since $f_2$ is a step polynomial, and each term in $f_3$ is an element of $\partial_{\ell+1}^{-1}(R^{(\ell+1)})$, it now suffices to find an SP representative for $f_3$ modulo $\partial_\ell(S^{(\ell-1)}) + R^{(\ell)}$, instead of $f$.  However, that now follows from the $\prec$-inductive hypothesis: this time, the previous case of part (2) of Proposition~\ref{prop:internal-stuff} applied to $\rmH^0_\m(U_i',\scrR \leq \scrS)$ (the case $p=0$ in Claim 1).
\end{proof}

\subsection{Completed analysis of cohomology $\P$-modules}

\begin{proof}[Completed proof of Theorems~\ref{thm:HpgoodSHC} and~\ref{thm:QP-prims}]
These are proved together by an induction on $k$, the size of the tuple $\bfU$ of acting subgroups.  The base clause $k=0$ will be a special case of the recursion clause with vacuous inductive hypothesis, so we do not explain it separately.

Thus, suppose that $\scrP \leq \scrQ$ satisfies the assumptions of Theorems~\ref{thm:HpgoodSHC} and~\ref{thm:QP-prims}, and that those theorems are already known for any strictly smaller tuple of acting groups (if such exist).

First, this inductive hypothesis immediately gives the desired conclusions for any proper restriction of $\scrP \leq \scrQ$.  It therefore remains to prove that the top structure complex of $\rmH^p_\m(W,\scrP\leq \scrQ)$ is co-induced over $(W+Y + U_{[k]})$ from a locally SP complex which has stably fully SP homology, and that $(P_k \leq Q_k)$ admits stable SP relative-coboundary-solutions over $W$.

Next, knowing the desired conclusions for the proper restrictions of $\scrP \leq \scrQ$, this already implies that $(\scrR \leq \scrS) := {\rmH^p_\m(W,\scrP\leq \scrQ)}$ satisfies the hypotheses of Proposition~\ref{prop:internal-stuff} for any $p\geq 1$ (not yet $p=0$, since that proposition required \emph{strict} meekness of proper restrictions).  In this case ${\rmH^p_\m(W,\scrP\leq \scrQ)}$ is strictly modest, hence automatically also SP-modest, and so that proposition gives that it has stable finite-complexity decompositions in its top structure complex, and stably SP-represented homology at all positions except the last.

Now let $Z_1 := W+Y+U_{[k]}$.  Then $(\scrP\leq \scrQ)$ is equal to $\Cnd_{Z_1}^Z(\scrP^\circ \leq \scrQ^\circ)$ for some semi-functional $(Z_1,Y,\bfU)$-$\P$-module $(\scrP^\circ \leq \scrQ^\circ)$.  This $(\scrP^\circ \leq \scrQ^\circ)$ has all the same sub-constituents and same lean version as $(\scrP \leq \scrQ)$, so it also satisfies the hypotheses of Theorems~\ref{thm:HpgoodSHC} and~\ref{thm:QP-prims}.  Therefore the reasoning above also applies to ${\rmH^p_\m(W,\scrP^\circ\leq \scrQ^\circ)}$: it too has stable finite-complexity decompositions in its top structure complex, and stably SP-represented homology at all positions except the last.

Finally, this last conclusion provides the hypotheses needed to apply Proposition~\ref{prop:Hp-of-cplx} to the top structure complex of $(\scrP \leq \scrQ)$, which is likewise co-induced from that of $(\scrP^\circ \leq \scrQ^\circ)$.  That Proposition gives all the remaining desired conclusions (including when $p=0$).
\end{proof}

\section{Partial difference equations and zero-sum tuples}\label{sec:PDceE}

In Part I, {\PDE}-solution $\P$-modules and zero-sum $\P$-modules were shown to be $1$-almost and $2$-almost modest, respectively, and this implied the main structural results of that paper.  This was proved by showing how a general {\PDE}-solution $\P$-module, resp. zero-sum $\P$-module, can be constructed out of simpler $\P$-modules using cohomology and short exact sequences: since all of those preserve the desired almost modest structure, the result followed.

We will now prove Theorems A and B along similar lines, using the general results of the preceding sections about the preservation of almost meekness.

\begin{proof}[Proof of Theorem A]
Proposition~\ref{prop:full-rot-char} implies that if $A$ is compact-by-discrete, then it is SP; this is the key assumption for the proof.

Fix $Z$ and $\bfU = (U_1,\ldots,U_k)$, and let us recall the description of the solution $\P$-module for the resulting {\PDE} given in Subsection I.7.1.  For each $j \in \{0,1,\ldots,k\}$, let $\bfU^{(j)} = (U_\ell^{(j)})_{\ell=1}^k$ be the subgroup tuple defined by
\[U_\ell^{(j)} := \left\{\begin{array}{ll}U_\ell &\quad \hbox{if}\ \ell \leq j\\ \{0\} &\quad \hbox{if}\ \ell > j,\end{array}\right.\]
and let $\scrM^j$ be the solution $\P$-module of the {\PDE} directed by $\bfU^{(j)}$.  In particular, $\scrM^k$ is the $\P$-module we wish to analyze.

It follows from this definition that
\[M^0_\emptyset = 0 \quad \hbox{and}\quad M^0_e = \F(Z,A) \quad \hbox{whenever}\ e \neq \emptyset.\]
This is easily seen to be both $1$-almost SP-modest and $1$-almost meek as a $(Z,0,\bfU^{(0)})$-$\P$-module: indeed, its homology is equal to $\F(Z,A)$ in position $(e,1)$ whenever $e \neq \emptyset$, and vanishes at all other positions. (This meekness assertion is using the fact that $A$ is SP)

Also, let $\scrM^j_\llc := \scrM^j_{\llc[k]\setminus \{j+1\}}$ for each $j$.  In Subsection I.7.1, it was next shown that all these solution $\P$-modules $\scrM^j$ and their reductions $\scrM^j_\llc$ are related as follows.  For each $j$, if $\scrM^j$ is $1$-almost modest, then the $(Z,0,\bfU^{(j)})$-$\P$-module $\scrM^j_\llc$ and the $(Z,U_{j+1},\bfU^{(j)})$-$\P$-module $(\scrM^j/\scrM^j_\llc)^{U_{j+1}}$ are both also $(Z,0,\bfU^{(j+1)})$-$\P$-modules, and are still $1$-almost modest after this re-interpretation.  Interpreting $(\scrM^j/\scrM^j_\llc)^{U_{j+1}}$ as the semi-functional $\P$-module $\rmH^0_\m(U_{j+1},\scrM^j_\llc \leq \scrM^j)$, the same reasoning works for $1$-almost SP-modesty and $1$-almost meekness.

Then, for each $j \in \{0,1,\ldots,k-1\}$, one obtains a short exact sequence of $(Z,0,\bfU^{(j+1)})$-$\P$-modules:
\[\scrM_\llc^j\into \scrM^{j+1}\onto \rmH^0_\m(U_{j+1},\scrM^j_\llc \leq \scrM^j).\]
For each $e\subseteq [k]$, this short exact sequence arises from the concatenation
\[0 \leq (M^j_\llc)_e \leq M^{j+1}_e = \Z^0(U_{j+1},M^j_e,M^j_{\llc e}),\]
and so it is actually stably exact when regarded as a sequence of semi-functional $\P$-modules.  We know from Subsection I.7.1 that all three of these $\P$-modules are $1$-almost modest, and an easy check shows that they are actually all $1$-almost SP-modest.  This is because when $|e| = 1$, the modules $M^j_e$ and $(M^j_e/M^j_{\llc e})^{U_{j+1}}$ are functional, not just semi-functional: they always equal either $M^j_e$ (functional) or $0$.  Locally SP homology follows easily from this using Lemma~\ref{lem:almost-disc-and-locally-SP}.

Since we have seen that $\scrM^0$ is a $1$-almost meek $(Z,0,\bfU^{(0)})$-$\P$-module, it now follows by induction on $j$, using Corollary~\ref{cor:restriction-still-meek}, Proposition~\ref{prop:snakerep}, and Theorem~\ref{thm:HpgoodSHC}, that $\scrM^j$, $\scrM^j_\llc$ and then $\rmH^0_\m(U_{j+1},\scrM^j_\llc \leq \scrM^j)$ are $1$-almost meek for every $j$.  Once we reach $j=k$, this includes the conclusion we wanted.
\end{proof}

\begin{proof}[Proof of Theorem B]
Once again, it suffices to assume $A$ is a SP module. Let $\bfU^{(j)}$ for $j \in \{0,1,\ldots,k\}$ be as before, and now for each $j$ let $\scrN^j$ be the zero-sum $\P$-module associated to $(Z,\bfU^{(j)})$.  In this setting, Section I.7.3 showed that $\scrN^0$ is $2$-almost modest, and that these $\scrN^j$ are related to one another in just the same way as the {\PDE}-solution modules $\scrM^j$ in the proof of Theorem A.  Since $\scrN^0$ is also easily seen to be $2$-almost meek, the rest of the proof is now just like the proof of Theorem A.
\end{proof}

With Theorem A in hand, the following is an easy corollary.  The analogous corollary for Theorem B can be proved along exactly the same lines.

\begin{cor}\label{cor:sp-solns-dense}
If $\scrM$ is the {\PDE}-solution $\P$-module corresponding to $Z$ and $\bfU$, so that $M_{[k]}$ is the module of solutions, then $(M_{[k]})_{\rm{sp}}$ is dense in $M_{[k]}$.
\end{cor}

\begin{proof}
We prove this by induction on $k$.  When $k=1$, the solutions are just measurable functions lifted from $Z/U_1$, which may of course be approximated by lifts of step polynomials on $Z/U_1$.  So now suppose the result is known for all {\PDE}s of degree less than some $k\geq 2$, and that $f \in M_{[k]}$.  By Theorem A, there is a decomposition
\[f = f_0 + \sum_{i=1}^k f_i,\]
where $f_i \in M_{[k]\setminus i}$ for $i=1,2,\ldots,k$, and where the restriction $f_0|(z + U_{[k]})$ is a step polynomial for each coset $z + U_{[k]}$.

Each $f_i$ may be approximated by a step polynomial solution to its respective simpler {\PDE}, by the hypothesis of our induction on $k$, so it suffices to find a step-polynomial approximation to $f_0$.

To do this, let $V:= Z/U_{[k]}$, and let $\psi:V\times U_{[k]}\to Z$ be a $U_{[k]}$-equivariant bijection as given by Corollary~\ref{cor:fund-split}.  By considering $f_0 \circ \psi$ instead of $f_0$, it suffices to find a step polynomial approximation to $f_0$ in the special case $Z = V\times U_{[k]}$, so we now assume this.

With this assumption, some routine measure theory gives that if $\frP$ is a QP partition of $V$ into cells of sufficiently small diameter (for some choice of compact group metric on $V$), then $f_0$ may be approximated in probability by a function $f_1$ such that
\[f_1(v,u) = f_C(u) \quad \forall (v,u) \in C \times U_{[k]},\]
where $f_C$ is a step polynomial solution to our {\PDE} on $U_{[k]}$ itself for each $C \in \frP$.  The right-hand side clearly defines a step-polynomial element of $M_{[k]}$, so this completes the proof.
\end{proof}

\section{Steps towards quantitative bounds}\label{sec:quantitative}

Our next goal is Theorem C, the quantitative extension of Theorem A.  Before proving it, we must introduce the notion of `complexity' for a step polynomial, and then prove some of its basic properties.  This notion is not canonical, and other possibilities will be discussed following the proofs, but it is essentially the only notion for which I can prove a version of Theorem C.

\subsection{Complexity and its basic consequences}\label{subs:intro-cplxty}

Recall from Subsection~\ref{subs:step-poly} that a function $f:Z\to \bbT$ is a step polynomial if and only if $\{f\}:Z\to \bbR$ is a step polynomial, and that any step polynomial $[0,1)^d\to \bbR$ is a sum of basic step polynomials, one for each cell of some convex polytopal partition of $[0,1)^d$ that controls it.

\begin{dfn}[Complexity]\label{dfn:cplxty}
Fix $D \in \bbN$. If $C\subseteq [0,1)^d$ is a convex polytope, then it has \textbf{complexity at most $D$ relative to $[0,1)^d$} if $C = [0,1)^d\cap H_1\cap\cdots \cap H_D$ for some open or closed half-spaces $H_i \subseteq \bbR^d$, $i = 1,2,\ldots,D$.

A basic step polynomial $f:[0,1)^d\to \bbR$ has \textbf{basic complexity at most $D$} if
\begin{itemize}
\item $d \leq D$, and
\item $f$ may be represented as $p\cdot 1_C$, where $p:\bbR^d\to \bbR$ is a polynomial of degree at most $D$ and with absolute values of all coefficients bounded by $D$, and $C \subseteq [0,1)^d$ is a convex polytope of complexity at most $D$ relative to $[0,1)^d$.
\end{itemize}

A step polynomial $f:[0,1)^d\to \bbR$ has \textbf{complexity at most $D$} if it is a sum of at most $D$ basic step polynomials, all of basic complexity at most $D$.  The least $D$ for which this holds is denoted $\cplx(f)$.

Finally, if $Z$ is a compact Abelian group, then a function $f:Z\to \bbT$ has \textbf{complexity at most $D$} if $\{f\} = f_0\circ \{\chi\}$ for some affine map $\chi:Z\to \bbT^d$ and some $f_0:[0,1)^d\to \bbR$ of complexity at most $D$.  The least $D$ for which there is such a factorization is denoted $\cplx(f)$.

In either setting, if $f$ is not a step polynomial then we set $\cplx(f) := \infty$.
\end{dfn}

Our first estimate for this notion is obvious from the definition.

\begin{lem}\label{lem:sum-cplx-bdd}
If $g_0,h_0:[0,1)^d\to \bbR$ are step polynomials, then $\cplx(g_0 + h_0) \leq \cplx(g_0) + \cplx(h_0)$. \qed
\end{lem}

The next estimate requires a little more work.

\begin{lem}\label{lem:rotn-cplx-bdd}
For each $D \in \bbN$ there is a $D' \in \bbN$ with the following property.  Let $d\leq D$, let $Z$ be a compact Abelian group and let $\chi:Z\to \bbT^d$ be an affine map. If $g_0:[0,1)^d\to \bbR$ has complexity at most $D$, $g := g_0\circ \{\chi\}$, and $z \in Z$, then there is some $h_0:[0,1)^d\to \bbR$ of complexity at most $D'$ such that $R_zg = h_0\circ \{\chi\}$.
\end{lem}

\begin{proof}
First suppose that $Z = \bbT^d$ and $\chi$ is the identity.  By summing at most $D$ terms, it suffices to prove this when $g = p\cdot 1_C$ is a basic step polynomial of basic complexity at most $D$.  However, now the coordinates of $\{z\} = (\{z_1\},\ldots,\{z_d\})$ dissect $[0,1)^d$ into at most $2^d \leq 2^D$ smaller boxes, say $Q_\omega$ for $\omega \in \{0,1\}^d$, and for each $\omega$ one has a vector $v_\omega \in \bbR^d$ such that the following diagram commutes:
\begin{center}
$\phantom{i}$\xymatrix{
\{\cdot\}^{-1}(Q_\omega) \ar_{\{\cdot\}}[d]\ar^-{w\mapsto w-z}[rr] && \{\cdot\}^{-1}(Q_\omega) - z \ar^{\{\cdot\}}[d]\\
Q_\omega \ar_-{u\mapsto u + v_\omega}[rr] && Q_\omega + v_\omega.
}
\end{center}
It follows that for each $\omega$, the function $R_z((p\cdot 1_{C\cap Q_\omega})\circ \{\cdot\})$ is still a basic polynomial of complexity at most $D + d \leq 2D$ (since $C\cap Q_\omega$ could require up to $D + d$ linear inequalities to define it).  Summing over $\omega$ completes the proof when $\chi = \rm{id}_{\bbT^d}$.

Finally, for general $\chi$, simply observe that we may factorize $g = g_1\circ \chi$, where $g_1 := g_0\circ \{\cdot\}$ satisfies the assumption of the special case above.  Since $R_zg = R_{\chi(z)}g_1\circ \chi$, the result now follows from that special case.
\end{proof}

\begin{lem}[Compactness]\label{lem:cptness}
If $d \leq D$ then the set
\[\big\{f\circ \{\cdot\}\,\big|\ f:[0,1)^d\to \bbR,\ \cplx(f) \leq D\big\} \subseteq \F(\bbT^d)\]
is compact for the topology of convergence in probability for the measure $m_{\bbT^d}$.
\end{lem}

\begin{proof}
It suffices to prove this for the set of all compositions $f_0\circ \{\cdot\}$ with $f_0$ a basic polynomial of basic complexity at most $D$, and hence for the set of all functions $(p\cdot 1_C)\circ \{\cdot\}$ where $p$ and $C$ satisfy the bounds in the definition of basic complexity.  Also, since $\{\cdot\}:\bbT^d\to [0,1)^d$ defines an isomorphism of measure spaces, it suffices to prove this compactness on $[0,1)^d$ itself.  This conclusion now follows because these data are specified by $\rm{O}_D(1)$ coefficients for $p$, all lying in $[-D,D]$, and by at most $D$ open or closed half-spaces that intersect $[0,1]^d$, and the set of all such half-spaces is compact.
\end{proof}

\begin{cor}\label{cor:cptness}
Let $d\leq D$, let $Z$ be a compact Abelian Lie group and let $\chi:Z\to\bbT^d$ be affine. Then the set
\[\big\{\{f\}\circ \{\chi\}\,\big|\ f:[0,1)^d\to \bbR,\ \cplx(f) \leq D\big\} \subseteq\F(Z)\]
is pre-compact for the topology of convergence in $m_Z$-probability, and its closure is contained in the set
\[\{g \in \F(Z)\,|\ \cplx(g) \leq D'\}\]
for some $D' \in \bbN$ (which may depend on $Z$ and $\chi$).
\end{cor}

\begin{proof}
It clearly suffices to prove this for each coset of the identity component $Z_0 \leq Z$ separately, so we may assume that $Z = \bbT^r$ for some $r$.  Now Lemma~\ref{lem:little-step-aff-coords} gives a commutative diagram
\begin{center}
$\phantom{i}$\xymatrix{
\bbT^r \ar_{\{\cdot\}}[d]\ar^\chi[r] & \bbT^d \ar^{\{\cdot\}}[d]\\
[0,1)^r \ar_\psi[r] & [0,1)^d
}
\end{center}
for some step-affine map $\psi:[0,1)^r \to [0,1)^d$.  It is also clear that if $f:[0,1)^d\to\bbR$ has $\cplx(f) \leq D$, then $f\circ \psi$ has complexity bounded by some $D'$ that depends only on $\psi$ and $D$.  Therefore the set of functions in question is contained in
\[\big\{\{f'\}\circ \{\cdot\}\,\big|\ f':[0,1)^r \to \bbR,\ \cplx(f') \leq D'\big\},\]
so the result follows from Lemma~\ref{lem:cptness}.
\end{proof}

\begin{lem}[Bounded complexity and equidistribution]\label{lem:bdd-cplx-equidist}
Let $Z$ be a compact Abelian Lie group with a translation-invariant metric $\rho$, let $\chi:Z\to \bbT^d$ be affine, and let $D \in \bbN$ be fixed. Then for every $\eps > 0$ there is a $\delta > 0$ such that for any closed subgroup $Y \leq Z$, if $Y$ is $\delta$-dense in $Z$ for the metric $\rho$, then
\[|d^Y_0(0,g|_Y) - d^Z_0(0,g)| < \eps\]
whenever $g = g_0\circ \{\chi\}$ and $g_0:[0,1)^d\to \bbR$ has complexity at most $D$, where $d_0^Y$ and $d_0^Z$ are the metrics of convergence in $m_Y$-probability and $m_Z$-probability, respectively.
\end{lem}

\begin{proof}
For $\delta$ sufficiently small, it suffices to prove this on each connected component of $Z$, so we may assume that $Z = \bbT^r$.  Having done so, we may argue from Lemma~\ref{lem:little-step-aff-coords} as in the proof of Corollary~\ref{cor:cptness} to pull everything back to $\bbT^r$ itself, and assume that $\chi$ is the identity.

Lastly, it suffices to prove the result for $g = g_0\circ \{\cdot\}$ with $g_0$ a basic step polynomial of basic complexity at most $D$, say $g_0 = p\cdot 1_C$ as in Definition~\ref{dfn:cplxty}. At this point the result is clear, because
\begin{itemize}
\item on the one hand, the resulting function $(p\cdot 1_C)\circ \{\cdot\}:\bbT^d\to \bbR$ is locally Lipschitz, with constant bounded in terms of $D$, on both the interior and exterior of $\{\cdot\}^{-1}(C)$,
\item and on the other, $C \subseteq [0,1)^d$ is a convex set with boundary contained in a union of at most $D$ portions of hyperplanes in $\bbR^d$, implying that
\[m_{[0,1)^d}(\{v\,|\ \rm{dist}(v,\partial C) < \delta\})\to 0 \quad \hbox{as}\ \delta\downarrow 0\]
at a rate that can be bounded in terms of $D$ alone.
\end{itemize}
\end{proof}

\subsection{Proof by compactness}\label{subs:cptness-proof}

Theorem C will also be proved by compactness and contradiction.  It will involve the Vietoris topology on the set of closed subgroups of a torus $\bbT^d$, which is a compact and metrizable topology since $\bbT^d$ is compact and metrizable.  Letting $\rho$ be any standard choice of metric on $\bbT^d$ (such as the quotient of the Euclidean metric on $\bbR^d$), if $Z_n\to Z$ as subsets of $\bbT^d$ with this topology, then, by definition, for every $\eps > 0$ there is an $n_0$ such that $Z_n$ is $\eps$-dense in $Z$ for $\rho$ for all $n\geq n_0$.  It is also standard that this implies $m_{Z_n}\to m_Z$ in the vague topology.

In this setting we will need the following elementary lemma.

\begin{lem}\label{lem:gp-lim-from-within}
If $Z_n$ is a sequence of closed subgroups of $\bbT^d$ tending to another subgroup $Z$ in the Vietoris topology, then $Z_n \leq Z$ for all sufficiently large $n$.
\end{lem}

\begin{proof}
This is most easily seen in the Pontryagin dual.  The Vietoris convergence $Z_n\to Z$ implies the vague convergence $m_{Z_n}\to m_Z$, and hence the convergence $Z_n^\perp\to Z^\perp$ as subsets of $\hat{\bbT^d} = \bbZ^d$, in the sense of eventual agreement at any given point of $\bbZ^d$.  Since all subgroups of $\bbZ^d$ are generated by at most $d$ elements, this means that $Z_n^\perp$ must eventually contain a set of generators for $Z^\perp$, at which point one has $Z_n^\perp\geq Z^\perp$ and hence $Z_n \leq Z$.
\end{proof}

\begin{proof}[Proof of Theorem C]
We will first select an $\eps$ depending only on $k$.  The keys to this are the quantitative results from Part I.

Firstly, Theorem I.C implies that there is non-decreasing function $\k:(0,\infty)\to (0,\infty)$, depending only on $k$ and tending to $0$ at $0$, such that if $\scrM$ is the solution $\P$-module for the {\PDE} associated to some $Z$ and subgroup-tuple $\bfU = (U_1,\ldots,U_k)$, and if $f \in \F(Z)$ is such that
\[d_0(0,\ d^{U_1}\cdots d^{U_k}f) < \eps \quad \hbox{in}\ \F(U_1\times \cdots \times U_k\times Z),\]
then $d_0(f,M_{[k]}) < \k(\eps)$.

Second, Theorem I.A$'$ promises some $\eta > 0$, depending only on $k$, such that if $f' \in M_{[k]}$ and $d_0(0,f') < \eta$, then in fact $f' \in \partial_k(M^{(k-1)})$.

Putting these facts together, we may now choose $\eps > 0$ so small that
\[\eps + \k(2^{k+1}\eps) < \eta.\]
This still depends only on $k$.  We will prove that this $\eps$ has the property asserted in Theorem C.

This will be proved by assuming otherwise, and deriving a contradiction from Theorem A and a compactness argument.  Thus, suppose that $Z_n$ is a sequence of compact Abelian groups and $\bfU_n = (U_{n,i})_{i=1}^k$ a sequence of tuples of subgroups.  Let $\scrM_n = (M_{n,e})_e$ be the {\PDE} solution $\P$-module associated to $\bfU_n$, and let $f_n \in M_{n,[k]}$ and $g_n \in \F_{\rm{sp}}(Z_n,\bbT)$ be sequences such that
\[\cplx(g_n) \leq d \quad \hbox{and} \quad d_0(f_n,g_n) < \eps\]
for all $n$, but on the other hand such that
\[\min\{\cplx(f')\,|\ f' \in f_n + \partial_k(M_n^{(k-1)})\} \to \infty.\]

For each $n$, let $\chi_n:Z_n\to \bbT^{r_n}$ be an affine map and $g''_n:[0,1)^{r_n}\to \bbR$ be a step polynomial such that
\[g_n = g''_n\circ \{\chi_n\} \mod \bbZ \quad \hbox{and} \quad  \cplx(g_n'') = \cplx(g_n).\]
Also, let
\[g_n' := g_n''\circ \{\cdot\} \mod \bbZ:\bbT^{r_n}\to \bbT,\]
so $g_n = g_n'\circ\chi_n$.

\vspace{7pt}

\emph{Step 1.}\quad Since $r_n \leq d$ for every $n$, after passing to a subsequence we may suppose that
\begin{itemize}
\item $r_n = r$ for every $n$,
\item $Z_n' := \chi_n(Z_n)\to Z'$ for some $Z' \leq \bbT^r$ in the Vietoris topology,
\item similarly, $U_{n,i}' := \chi_n(U_{n,i})\to U_i' \leq Z'$ for each $i\leq k$ in the Vietoris topology.
\end{itemize}
Lemma~\ref{lem:gp-lim-from-within} implies that $Z_n' \leq Z'$ and $U'_{n,i} \le U'_i$ for all sufficiently large $n$; by omitting finitely many terms of our sequences, we may assume this holds for all $n$.

Using Corollary~\ref{cor:cptness}, another passage to a subsequence now allows us to assume in addition that $d_0(g'_n|_{Z'},g') \to 0$ in $\F(Z')$, where $g':Z' \to \bbT$ is another step polynomial.

Let $\scrM = (M'_e)_e$ be the solution $\P$-module for the {\PDE} associated to $Z'$ and $\bfU' = (U_1',\ldots,U_k')$.

\vspace{7pt}

\emph{Step 2.}\quad Since $d_0(f_n,g_n) < \eps$ and $f_n \in M_{n,[k]}$, we have
\[d_0(0,\ d_{u_1}\cdots d_{u_k}g_n ) < 2^k\eps \quad \hbox{in}\ \F(Z_n)\]
for all $(u_1,\ldots,u_k) \in \prod_iU_{n,i}$, for all $n\geq 1$.  Equivalently,
\begin{eqnarray}\label{eq:dddsmall}
d_0(0,\ (d_{u_1'}\cdots d_{u_k'}g'_n)|_{Z'_n}) < 2^k\eps \quad \hbox{in}\ \F(Z_n')
\end{eqnarray}
for all $(u_1',\ldots,u_k') \in \prod_i U_{n,i}'$, for all $n\geq 1$.

Now, since $g'_n = g''_n\circ \{\cdot\}$ mod $\bbZ$ for $g''_n:[0,1)^d\to \bbR$ having complexity at most $d$, applying each of Lemmas~\ref{lem:sum-cplx-bdd} and~\ref{lem:rotn-cplx-bdd} $k$ times gives that also
\[d_{u_1'}\cdots d_{u_k'}g'_n = g''_{n,u_1',\ldots,u_k'}\circ \{\cdot\} \mod \bbZ\]
for some functions $g''_{n,u_1',\ldots,u_k'}:[0,1)^d\to \bbR$ whose complexity is bounded only in terms of $D$ and $k$, uniformly in $(u_1',\ldots,u_k')$.  On the other hand, $Z_n'$ equidistributes in $Z'$ as $n\to\infty$.  Given this, Lemma~\ref{lem:bdd-cplx-equidist} and~(\ref{eq:dddsmall}) imply that
\begin{eqnarray}\label{eq:some-conv}
d_0(0,\ d_{u_1'}\cdots d_{u_k'}g'_n) < 2^{k+1/3}\eps \quad \hbox{in}\ \F(Z')
\end{eqnarray}
for all $(u_1',\ldots,u_k') \in \prod_i U_{n,i}'$, for all sufficiently large $n$.

Next, since $d_0(g'_n,g') \to 0$ in $\F(Z')$, the estimate~(\ref{eq:some-conv}) implies that
\[d_0(0,\ d_{u_1'}\cdots d_{u_k'}g') < 2^{k+2/3}\eps \quad \hbox{in}\ \F(Z')\]
for all $(u_1',\ldots,u_k') \in \prod_i U_{n,i}'$, for all sufficiently large $n$.  Finally, since $\bigcup_{n\geq 1}U'_{n,i}$ is dense in $U'_i$ for each $i$, this turns into
\begin{multline*}
d_0(0,\ d_{u_1'}\cdots d_{u_k'}g') < 2^{k+1}\eps \quad \hbox{in}\ \F(Z') \quad \forall (u_1',\ldots,u_k') \in \prod_i U_i'\\
\Longrightarrow \quad d_0\big(0,\ d^{U_1'}\cdots d^{U_k'}g'\big) < 2^{k+1}\eps \quad \hbox{in}\ \F(U_1'\times \cdots \times U_k'\times Z').
\end{multline*}

\vspace{7pt}

\emph{Step 3.}\quad Having reached this last estimate, we may apply Theorem I.C to conclude that there is some $h' \in M_{[k]}'$ such that
\[d_0(g',h') < \kappa(2^{k+1}\eps).\]
Moreover, by the density result of Corollary~\ref{cor:sp-solns-dense}, we may assume in addition that this $h'$ is a step polynomial on $Z'$.

\vspace{7pt}

\emph{Step 4.}\quad Finally, since $Z_n' \leq Z'$ and $U'_{n,i} \leq U'_i$ for all $i$ and $n$, the pulled-back functions
\[h_n := h'\circ \chi_n\]
are all well-defined, each $h_n$ solves the {\PDE} associated to $Z_n$ and $\bfU_n$, and each $h_n$ is still a step polynomial on $Z_n$ of complexity at most $\cplx(h')$ (since pre-composing with an affine function cannot increase complexity according to Definition~\ref{dfn:cplxty}).  Now another appeal to Lemma~\ref{lem:bdd-cplx-equidist}, the equidistribution of $Z_n'$ in $Z'$, and the fact that $g_n'\to g'$ in $\F(Z')$ give that
\[d_0(g_n,h_n) \leq d_0(g_n'\circ\chi_n,g'\circ \chi_n) + d_0(g'\circ \chi_n,h'\circ \chi_n) < \kappa(2^{k+1}\eps)\]
for all $n$ sufficiently large.  Therefore
\[d_0(f_n,h_n) < \eps + \k(2^{k+1}\eps) \quad \hbox{in}\ \F(Z_n).\]
However, we chose this right-hand side to be less than $\eta$, so this implies $h_n \in f_n + \partial_k(M_n^{(k-1)})$, yet $\cplx(h_n) = \cplx(h')$ remains bounded as $n\to\infty$.  This contradicts our initial assumptions, and so completes the proof.
\end{proof}

\subsection{Discussion}\label{subs:cplx-discussion}

Instead of proving Theorem C by compactness and contradiction, it is tempting to try to keep track of the complexities of the step polynomials that appear as we build up the machinery of semi-functional $\P$-modules.  This would hopefully lead to effective versions of our results about short exact sequences, cohomology $\P$-modules, and so on.  (In principle, our proof by compactness could be forced to yield explicit bounds by quantifier-elimination, but they would be atrocious.)

Unfortunately, the most na\"\i ve version of this idea cannot be carried out, because some of the necessary results are not true.

\begin{ex}\label{ex:fcd-bad}
Let $Z = (\bbZ/p\bbZ)^2\times \bbT$ for some large prime $p$, and let $\chi:\bbT\to\bbT$ be the identity character.  Let
\begin{multline*}
U_1 := ((1,0)\cdot (\bbZ/p\bbZ))\times \bbT, \quad U_2 := ((1,-1)\cdot (\bbZ/p\bbZ))\times \bbT\\ \hbox{and} \quad U_3 := ((0,1)\cdot (\bbZ/p\bbZ))\times \bbT,
\end{multline*}
and define $f:Z\to\bbT$ by
\[f(s_1,s_2,t) := \lfl \{s_1\}/p + \{s_2\}/p\rfl\cdot p \cdot \chi(t),\]
where $\{s_1\}$ now denotes the element of $\{0,1,\ldots,p-1\}$ that represents $s_1$ modulo $p$.

An easy check shows that
\begin{eqnarray}\label{eq:repf}
f(s_1,s_2,t) = \{s_1\}\chi(t) + \{s_2\}\chi(t) - \{s_1 + s_2\}\chi(t),
\end{eqnarray}
so $f$ is a degenerate solution of the {\PDE} associated to $(U_1,U_2,U_3)$.

This $f$ has complexity bounded independently of $p$. Indeed, one may simply observe that $f = f_1\circ \pi_p$, where
\[\pi_p:(\bbZ/p\bbZ)^2\times \bbT \to \bbT^3:(s_1,s_2,t) \mapsto (\{s_1\}/p \!\!\!\!\mod 1,\{s_2\}/p \!\!\!\!\mod 1,pt)\]
is a homomorphism, and
\[f_1(s_1',s_2',t') = \lfl \{s_1'\} + \{s_2'\}\rfl\cdot \chi(t').\]
Therefore $\cplx(f) \leq \cplx(f_1)$, and this does not depend on $p$.  However, the summands in~(\ref{eq:repf}) have complexity that grows as $p$ increases, because each of $\{s_1\}$, $\{s_2\}$ and $\{s_1 + s_2\}$ can take real values as large as $p-1$.  (This is not a complete proof that the summands grow in complexity, but it can be turned into one with a little extra work.  Alternatively, Corollary~\ref{cor:cptness} shows that if the complexities of these summands were bounded uniformly in $p$, then one could factorize them through affine maps to a fixed torus so that they are all pulled back from some pre-compact family of step polynomials, and one can also check by hand that this is not possible.) Similar reasoning also gives that there is no other representation as in~(\ref{eq:repf}) in which all summands have complexity bounded independently of $p$.

Before leaving this example, let us note that the function $f$ above is built from the function
\[(\bbZ/p\bbZ)^2\to \bbZ:(s_1,s_2)\to p\lfl \{s_1\}/p + \{s_2\}/p\rfl,\]
which is an element of $\B^2(\bbZ/p\bbZ,\bbZ)$.  This latter can also be cast as an example in which a low-complexity $2$-coboundary is not the coboundary of a low-complexity $1$-cochain, and from that perspective it can also be used to rule out certain complexity-dependences in Theorem~\ref{thm:QP-prims}. \fin
\end{ex}

The upshot of this example is that, although we have proved that a {\PDE}-solution $\P$-module $\scrM$ admits finite-complexity decompositions in its top structure complex, for a given step polynomial $f \in \partial_k(M^{(k-1)})$ it may not be possible to control the minimal complexity of $g \in \partial_k^{-1}\{f\}$ only in terms of $\cplx(f)$.  At the very least, this control must also depend on the choice of $Z$, and if $Z$ is infinite-dimensional then there may be no such control at all (for instance, by taking an infinite product of increasingly bad examples).

The best one can obtain is a control on the minimal complexity of $g \in \partial_k^{-1}\{f\}$ that depends both on $\cplx(f)$ and on an a priori bound for how easily some $g' \in \partial_k^{-1}\{f\}$ may be \emph{approximated} by controlled-complexity step functions. This would be much like the statement of Theorem C.  However, this is strictly stronger than a control on the ability to approximate $f$ itself (the above example witnesses this, too), and it seems to me a very complicated challenge to develop effective versions of the main results for $\P$-modules that take this technical necessity account.

\begin{ques}
In the setting above, if one fixes a Lie group $Z$ and subgroup-tuple $\bfU$, can one bound $\min\{\cplx(g)\,|\ g \in \partial_k^{-1}\{f\}\}$ only in terms of $\cplx(f)$?
\end{ques}

Looking again at Example~\ref{ex:fcd-bad}, the summands in~(\ref{eq:repf}) had large complexity (in the sense of Definition~\ref{dfn:cplxty}) only by virtue of requiring polynomials with large coefficients in their representation as step polynomials.  Those polynomials may still be taken to be quadratic, and the directing QP partitions involve only a small number of cells.

\begin{ques}
Is there a version of Theorem C in which one controls only the complexities of directing QP partitions and the degrees of polynomials, but not the size of their coefficients?
\end{ques}

I do not know how to approach this question.  Subsections~\ref{subs:intro-cplxty} and~\ref{subs:cptness-proof} both made essential use of some consequences of low complexity for the `regularity' of step polynomials, and this will not hold if one allows very large coefficients.  If one attempts to avoid the proof by contradiction-and-compactness, and instead keep track of this alternative kind of complexity through the work of all the previous sections, one still runs into this problem, because the regularity properties of step polynomials were also implicitly involved in the cohomological results of Subsection~\ref{subs:compar}.

Finally, recall from Theorems I.A and I.B that if $Z$ is a Lie group, then the structural homology of the $\P$-modules $\scrM$ and $\scrN$ is finitely generated in all positions above $1$ (resp. $2$).  Coupled with Theorems A and B of the present paper, this immediately implies the following.

\begin{cor}
Let $Z$ be a Lie group, $\bfU$ a tuple of at least three subgroups of $Z$, $\scrM$ the associated {\PDE}-solution $\P$-module, and $\scrN$ the associated zero-sum $\P$-module.  Then there is some $D \in \bbN$ such that every $f \in M_{[k]}$ (resp. $f \in N_{[k]}$) is of the form
\[f = n_1f_1 + \cdots + n_\ell f_\ell + g,\]
where $n_i \in \bbZ$, $f_i \in M_{[k]}$ (resp. $f_i \in N_{[k]}$) has complexity at most $D$ for all $i\leq \ell$, and $g \in \partial_k(M^{(k-1)})$ (resp. $g \in \partial_k(N^{(k-1)})$). \qed
\end{cor}

Of course, we cannot control the size of the coefficients $n_i$ in this corollary, and so we do not control the complexity of the whole sum $n_1f_1 + \cdots + n_\ell f_\ell$.  This begs the following question, which I suspect lies far beyond the methods of the present paper.

\begin{ques}
Does the preceding corollary hold for any compact Abelian $Z$, with a choice of $D$ that does not depend on $Z$ or $\bfU$, but now with no control over either the coefficients $n_i$ or the number of summands $\ell$?
\end{ques}

\appendix

\section{Step-polynomial cohomology for discrete modules}\label{app:sp-and-m-cohom}

This appendix proves Proposition~\ref{prop:disc-sp-and-m-cohom}.  We bring forward the notation used there.

First, let $V := Z/Z_1$.  A repeated application of Proposition~\ref{prop:s-p-coind-consistent} gives a diagram of cochain $Z_1$-complexes (hence also $W$-complexes)
\begin{center}
$\phantom{i}$\xymatrix{
0 \ar[r] & \Cnd_{Z_1}^Z A_1 \ar^-d[r]\ar_\cong[d] & \C^1(W,\Cnd_{Z_1}^Z A_1) \ar^-d[r]\ar_\cong[d] & \cdots\\
0 \ar[r] & \F(V,A_1) \ar^-d[r] & \C^1(W,\F(V,A_1)) \ar^-d[r] & \cdots,
}
\end{center}
where each vertical isomorphism is complexity-bounded in both directions.  This restricts to an isomorphism of the step-polynomial subdiagrams of the top and bottom rows.  It therefore suffices to prove the counterpart of Proposition~\ref{prop:disc-sp-and-m-cohom} for the bottom row: that is, to prove that the comparison maps
\[\rmH^p_\sp(W,0 \leq \F(V,A_1)) \to \F(V,\rmH^p_\m(W,A_1))\]
are injective with images equal to $\F_\sp(V,\rmH^p_\m(W,A_1))$.  Equivalently, this has reduced our work to the case $Z = V\times Z_1$.

The rest of the proof rests on defining two \emph{further} cohomological functors from $\PMod(Y)$ to Abelian groups, say $H^\ast = (H^p)_{p\geq 0}$ and $K^\ast = (K^p)_{p\geq 0}$, with the following properties:
\begin{itemize}
\item if $M$ is a discrete $Y$-module, then
\begin{multline}\label{eq:Hp-and-Kp}
H^p(M) = \rmH^p_\sp\big(W,0 \leq \F(V,M_1)\big)\\ \quad \hbox{and} \quad K^p(M) = \F_\sp\big(V,\rmH^p_\m(W,M_1)\big),
\end{multline}
where we will now let $M_1 := \Cnd_Y^{Z_1} M$ for any $M \in \PMod(Y)$,
\item $H^\ast$ and $K^\ast$ are both effaceable cohomological functors on the whole of $\PMod(Y)$ (these terms are defined in~\cite{Moo76(gr-cohomIII)}), and
\item there are comparison homomorphisms $H^p \to K^p$ in all degrees which respect long exact sequences and give an isomorphism in degree zero.
\end{itemize}
By the usual universality argument (see~\cite[Theorem 2]{Moo76(gr-cohomIII)}), the second and third of these properties together imply that all of the those comparison homomorphisms are actually isomorphisms.  Then the first property implies that the desired cohomology groups are equal for discrete modules, completing the proof of Proposition~\ref{prop:disc-sp-and-m-cohom}.

We need this rather abstract formulation because the universality argument works only for functors that are defined on a large enough category: in this case, $\PMod(Y)$.  It cannot be applied directly to the two right-hand sides in~(\ref{eq:Hp-and-Kp}), since these are well-defined only for discrete modules, which are too small a class to allow effacement.  Similar techniques were used in~\cite{AusMoo--cohomcty} to prove various cocycle-regularity results for the theory $\rmH^\ast_\m$.

The key to defining $H^\ast$ and $K^\ast$ on general Polish modules is the following.

\begin{dfn}
For any compact metrizable Abelian group $Z$ and Polish Abelian group $M$, a function $Z\to M$ is \textbf{almost-step} if it is a uniform limit of step functions $Z \to M$ (that is, finite-valued functions whose level-sets form a QP partition).  The set of Haar-a.e. equivalence classes of almost-step functions is denoted by $\F_{\rm{as}}(Z,M)$.
\end{dfn}

The definitions of $H^\ast$ and $K^\ast$ and the comparison between them will rely on a few basic properties of almost step functions.

\begin{lem}\label{lem:lifting-as}
If $\psi:M\to N$ is a continuous homomorphism of Polish Abelian groups and $F\in \F_{\rm{as}}(Z,M)$, then $\psi\circ F \in \F_{\rm{as}}(Z,N)$.  If $\psi$ is surjective, then for any $f \in \F_{\rm{as}}(Z,N)$ there exists $F \in \F_{\rm{as}}(Z,M)$ such that $f = \psi\circ F$.
\end{lem}

\begin{proof}
The first conclusion is obvious.

For the second, suppose $\psi$ is surjective, let $d$ be a bounded complete group metric generating the topology of $M$, let $\ol{d}$ be the resulting quotient metric on $N$, and let $\ol{d}_\infty$ be the uniform metric on functions $Z\to N$ that arises from $\ol{d}$.  Let $g_n:Z\to N$ be a sequence of step functions such that $\ol{d}_\infty(f,g_n) < 2^{-n-1}$ for all $n$, and let $\frP_n$ be a QP partition that refines the level-set partition of $g_n$ for each $n$.  By replacing each $\frP_n$ with $\bigvee_{m\leq n}\frP_m$ if necessary, we may assume that each $\frP_n$ refines its predecessors.

We now construct lifts $G_n:Z\to M$ of the functions $g_n$ recursively as follows.

To begin, let $G_1$ be any lift of $g_1$ which is constant on the cells of $\frP_1$.

Now assume that $g_m$ has already been defined for each $m \leq n$.  Since 
$\ol{d}_\infty(g_n,g_{n+1}) < 2^{-n}$, and since $g_n$ and $g_{n+1}$ are both constant on every cell of $\frP_{n+1}$, by the definition of $\ol{d}$ we may choose a lift of the value of $g_{n+1}$ on each of these cells which lies within $2^{-n+1}$ of the previously-chosen value of $G_n$ on this cell.  This new lift will be the value of $G_{n+1}$ on this cell.

These selections define the lifts $G_n$.  The construction gives
\[d_\infty(G_n,G_{n'}) < 2^{-n+1} + \dots + 2^{-n'}\]
whenever $n' > n$, so these lifts form a Cauchy sequence in the uniform topology.  Letting $F$ be their uniform limit completes the proof.
\end{proof}

\begin{lem}\label{lem:step-unif-cts}
Let $Z$ and $M$ be as above, and suppose that $f:Z\to M$ is a function for which there is a QP partition $\frP$ of $Z$ such that $f|C$ has a uniformly continuous extension to $\ol{C}$ for every $C \in \frP$.  Then $f$ is almost step.
\end{lem}

\begin{proof}
Fix $\eps > 0$, and let $d$ be a suitable metric on $M$.  Since $\frP$ has only finitely many cells, our assumption implies that there is a finite open cover $\cal{U}$ of $Z$, say by balls of sufficiently small radius for some choice of metric on $Z$, such that
\[\big(\ y \sim_\frP z \ \hbox{and}\ z,y \in U \in \cal{U}\ \big) \quad \Longrightarrow \quad d(f(y),f(z)) < \eps.\]
Now let $\frQ$ be a QP partition to which $\cal{U}$ is subordinate (Lemma~\ref{lem:q-p-fine}), and let $\frR := \frP\cap \frQ$.  Then the above property allows us to choose a function $g$ which is constant on every cell of $\frR$ and takes values uniformly within $\eps$ of those of $f$ everywhere.
\end{proof}

Finally, we will need the following result about `slicing' almost step functions.

\begin{lem}\label{lem:slicing-as}
If $Y$ and $Z$ are compact metrizable Abelian groups, $M$ is a Polish Abelian group with translation-invariant metric $d$, and $f:Y\times Z\to M$ is an almost-step function, then the function $F:Y\to \F(Z,M)$ defined by
\[F(y)(z) := f(y,z)\]
is also an almost-step function.
\end{lem}

\begin{proof}
Let $g_n \to f$ be a uniformly convergent sequence of step functions, and define $G_n:Y\to \F(Z,M)$ by
\[G_n(y)(z) := g_n(y,z).\]
These satisfy the uniform convergence $G_n\to F$ for the convergence-in-probability metric $d_0$ on $\F(Z,M)$.  The proof is finished by showing that each $G_n$ is an almost-step function (although possibly not a step function: consider again Example~\ref{ex:bad-slicing}).  This holds because Corollary~\ref{cor:connected-images} gives a QP partition $\frP$ of $Y$ such that $G_n$ restricts to a uniformly continuous function on each cell of $\frP$, so we may make another appeal to Lemma~\ref{lem:step-unif-cts}.
\end{proof}

Now let us return to the setting of two subgroups $Y,W \leq Z_1$ with $Y+W= Z_1$ and $Z = V\times Z_1$.  For each $p\geq 0$ and $M \in \PMod(Y)$ define
\[\C^p_{\rm{as}}(M) := \F_{\rm{as}}(W^p\times V\times Z_1,M)^Y,\]
where we take the fixed points for the diagonal $Y$-action on $Z_1$ and $M$. This is a submodule of $\C^p(W,\F(V,M_1))$.  Define the corresponding coboundary operators
\[d : \C^p_{\rm{as}}(M) \to \F(W^{p+1}\times V\times Z_1,M)^Y \cong \C^{p+1}(W,\F(V,M_1))\]
by the usual formula:
\begin{eqnarray}\label{eq:dfn-d2}
df(w_1,\ldots,w_{p+1},v,z) &:=& w_1\cdot (f(w_2,\ldots,w_{p+1},v,z+w_1))\nonumber \\ && + \sum_{i=1}^p(-1)^pf(w_1,\ldots,w_i + w_{i+1},\ldots,w_{p+1},v,z)\nonumber \\ 
&& + (-1)^{p+1}f(w_1,\ldots,w_p,v,z),
\end{eqnarray}
(this is really just a repeat of~(\ref{eq:dfn-d})).

\begin{lem}\label{cor:d-as-is-as}
If $f \in \C^p_\rm{as}(M)$ then $df \in \C^{p+1}_{\rm{as}}(M)$.
\end{lem}

\begin{proof}
Since $Z$ is compact, it is clear from formula~(\ref{eq:dfn-d2}) that if $g_n\to f$ uniformly, then $dg_n \to df$ uniformly, so it suffices to prove this in case $f$ is strictly a step function.

In this case, every summand in~(\ref{eq:dfn-d2}) is obviously also step, except for the first:
\[w_1\cdot (f(w_2,\ldots,w_{p+1},v,z+w_1)).\]
If $\frP$ is a QP partition of $W^p\times V\times Z_1$ that controls $f$, and $\frQ$ is its lift through the homomorphism
\[W^{p+1}\times (V\times Z_1)\to W^p\times (V\times Z_1): (w_1,\ldots,w_{p+1},v,z) \mapsto (w_2,\ldots,w_{p+1},v,z+w_1),\]
then that first term is uniformly continuous on every cell of $\frQ$, so it is almost step by the preceding lemma.

Finally, it is obvious that a sum of almost step functions is almost step.
\end{proof}

This lemma shows that one has a complex
\[0 \to \C^0_{\rm{as}}(M) \stackrel{d}{\to} \C^1_\rm{as}(M) \stackrel{d}{\to} \cdots.\]
Now we define $H^p(M)$ for $p\geq 0$ to be the homology of this complex.  In particular, $H^0(M)$ is the module of $W$-fixed points in $\C^0_\rm{as}(M) = \F_{\rm{as}}(V\times Z_1,M)^Y$, which is easily identified with $\F_{\rm{as}}(V,(\Cnd_Y^{Z_1}M)^W) = \F_{\rm{as}}(V,\rmH^0_\m(W,M_1))$.

We next define the functor $K^\ast$.  To this end, now set
\[\t{\C}^p(M) := \F_{\rm{as}}\big(V,\C^p(W,M_1)\big)\]
for all $p\geq 0$, and observe that these fit into a complex
\[0 \to \t{\C}^0(M) \stackrel{\t{d}}{\to} \t{\C}^1(M) \stackrel{\t{d}}{\to} \cdots,\]
where $\t{d}:\t{\C}^p(M) \to \t{\C}^{p+1}(M)$ is given by $f\mapsto d\circ f$ when $f \in \t{\C}^p(M)$ is regarded as an almost step function $V \to \C^p(W,M_1)$, and $d$ is the differential on $\C^\ast(W,M_1)$.  This complex is well-defined by the first part of Lemma~\ref{lem:lifting-as}.  We define $K^p(M)$ for $p\geq 0$ to be the homology of this complex.

In order to compare $H^\ast$ and $K^\ast$, observe from Lemma~\ref{lem:slicing-as} that if $f \in \C^p_\rm{as}(M) = \F_{\rm{as}}(W^p\times V\times Z_1,M)^Y$, and we regard it as a function $V\to \C^p(W,M_1)$, then it defines an element of $\F_{\rm{as}}(V,\C^p(W,M_1))$.  This re-interpretation therefore defines a sequence of homomorphisms $\C^p_\rm{as}(M)\to \t{\C}^p(M)$ that intertwine the differentials of these complexes, and these now descend to a sequence of comparison homomorphisms
\[H^p(M)\to K^p(M).\]

Now suppose that $M$ is discrete.  In this case, a uniformly convergent sequence of $M$-valued functions must actually stabilize.  Therefore
\[\C^p_{\rm{as}}(M) = \C^p_\sp\big(W,\F(V,M_1)\big) \quad \hbox{and} \quad H^p(M) = \rmH^p_\sp\big(W,0\leq \F(V,M_1)\big)\]
for discrete $M$.  (Note, however, that if $M$ is a non-discrete functional module, then the class of almost-step functions $W^p\times V\times Z_1\to M$ may be strictly larger than ${\F_\sp(W^p\times V\times Z_1,M)}$, and so it does not follow that $H^p(M) = \rmH^p_\sp(W,\F(V,M_1))$ for general functional modules $M$.)

On the other hand, for any $M \in \PMod(Y)$ for which
\[\rmH^p_\m(W,M_1) = \Z^p(W,M_1)/\B^p(W,M_1)\]
is Hausdorff, hence Polish, the second part of Lemma~\ref{lem:lifting-as} shows that
\begin{eqnarray}\label{eq:Kpequals}
K^p(M) = \frac{\F_{\rm{as}}(V,\Z^p(W,M_1))}{\F_{\rm{as}}(V,\B^p(W,M_1))} = \F_{\rm{as}}(V,\rmH^p_\m(W,M_1)),
\end{eqnarray}
regarded as a subgroup of $\F(V,\rmH^p_\m(W,M_1))$.  In particular, in case $M$ is discrete, the groups $\rmH^p_\m(W,M_1)$ are also discrete (see Corollary I.3.4), and so this identifies $K^p(M)$ with the subgroup $\F_\sp(V,\rmH^p_\m(W,M_1))$.

Thus, we have proved~(\ref{eq:Hp-and-Kp}), and Proposition~\ref{prop:disc-sp-and-m-cohom} becomes a consequence of the following more general result:

\begin{prop}\label{prop:as-and-m-cohom}
The comparison maps $H^p(M) \to K^p(M)$ are isomorphisms for every $M \in \PMod(Y)$.
\end{prop}

\begin{proof}
This is proved by showing that both $H^\ast$ and $K^\ast$ are effaceable cohomological functors on $\PMod(Y)$, and that the comparison map is an isomorphism when $p=0$.  Then a standard appeal to the universality properties of effaceable cohomological functors on $\PMod(Y)$ complete the proof: see~\cite[Theorem 2]{Moo76(gr-cohomIII)}.  This kind of argument was discussed a little more in Subsection I.3.2.

Thus, we must give the constructions of long exact sequences for both theories; prove that both are effaceable; and compare them in degree zero.

\vspace{7pt}

\emph{Long exact sequences.}\quad For $K^\ast$, long exact sequence simply follow from those for $\rmH^p_\m$ itself, using Lemma~\ref{lem:lifting-as}.

For $H^\ast$, the usual construction also works (compare~\cite{Moo76(gr-cohomIII)}): running through the construction of the switchback homomorphisms for the measurable theory, one sees that Lemma~\ref{lem:lifting-as} and Corollary~\ref{cor:d-as-is-as} enable one to stay among almost-step functions.

\vspace{7pt}

\emph{Effaceability.}\quad We will show that the inclusion $M\into \F(Z_1,M)$, where $\F(Z_1,M)$ is regarded as a $Y$-module with the diagonal action, is effacing for both $H^\ast$ and $K^\ast$.

Indeed, if $f:W^p\times V\times Z_1\to M$ is $Y$-equivariant and satisfies $df = 0$, then let $\xi:Z_1\to W$ be a $W$-equivariant step-affine function (from Proposition~\ref{prop:fund-lift}), and consider the function $F:W^{p-1}\times (V\times Z_1)\to \F(Z_1,M)$ defined by
\[F(w_1,\ldots,w_{p-1},v,z)(z') := (-1)^pf(w_1,w_2,\ldots,w_{p-1},\xi(z') - w_1 - \cdots - w_{p-1},v,z).\]
A direct calculation gives that $dF = f$ among $\F(Z_1,M)$-valued cochains.  If $f$ was almost step, then so is $F$, by Lemma~\ref{lem:slicing-as}, and similarly if $f \in \t{C}^p(M)$, then $F \in \t{\C}^{p-1}(\F(Z_1,M))$, so this inclusion effaces both theories.

\vspace{7pt}

\emph{Interpretation in degree zero.}\quad Finally, when $p=0$, we have seen that $H^0(M) = \F_{\rm{as}}(V,\rmH^0_\m(W,M_1))$, and equation~(\ref{eq:Kpequals}) gives the same for $K^0(M)$, since $\rmH^0_\m(W,M_1)$ is always Hausdorff.
\end{proof}

\bibliographystyle{abbrv}
\bibliography{bibfile}

\end{document}